\documentclass[reqno,a4paper,11pt]{article}

\usepackage{amsmath,amssymb,amsthm}

\usepackage{mathrsfs}

\usepackage{mathtools}
\usepackage{commath}

\usepackage{thmtools}
\usepackage{thm-restate}

\usepackage{cases}
\usepackage{enumitem}
\setlist[enumerate,1]{label=(\roman*)}

\usepackage[pdftex, pdfborderstyle={/S/U/W 0}]{hyperref}
\hypersetup{
    colorlinks=true,
    linkcolor=red,
    citecolor=cyan,
	hypertexnames=false,
}

\usepackage{cleveref}

\usepackage{etoolbox}

\usepackage{comment}

\usepackage[numbers]{natbib}

\numberwithin{equation}{section}

\usepackage{xcolor}
\usepackage[framemethod=default]{mdframed}
\usepackage{marginnote}
\usepackage{environ} 
\usepackage{geometry}

\makeatletter

\newif\ifthmboxes
\thmboxesfalse

\newif\ifproofbar
\proofbarfalse

\newif\iframbling
\ramblingfalse

\colorlet{clr@theorem}{red}
\colorlet{clr@proposition}{red}
\colorlet{clr@lemma}{red}
\colorlet{clr@corollary}{red}
\colorlet{clr@claim}{red}
\colorlet{clr@fact}{red}
\colorlet{clr@observation}{red} 

\colorlet{clr@question}{magenta}
\colorlet{clr@conjecture}{magenta}
\colorlet{clr@openproblem}{magenta}

\colorlet{clr@definition}{green}
\colorlet{clr@construction}{green}

\colorlet{clr@example}{blue}

\definecolor{clr@remark}{HTML}{9D4EDD} 

\colorlet{clr@proof}{black!30}   

\definecolor{clr@comment}{HTML}{2A9D8F}   

\newcommand{\marginSideNormal}{30mm}
\newcommand{\marginSideLeftRambling}{25mm}
\newcommand{\marginSideRightRambling}{45mm}
\newcommand{\marginNoteWidth}{35mm}
\newcommand{\marginNoteSep}{5mm}
\newcommand{\marginTopNormal}{25mm}
\newcommand{\marginBottomNormal}{25mm}
\iframbling
\else
\fi

\newcommand{\thmboxsetup}[1]{%
  linecolor=#1,%
  linewidth=0.6pt,%
  backgroundcolor=#1!8,%
  roundcorner=2pt,%
  innertopmargin=6pt,%
  innerbottommargin=6pt,%
  innerleftmargin=8pt,%
  innerrightmargin=8pt,%
  skipabove=8pt,%
  skipbelow=8pt,%
}

\newcommand{\newthmboxstyle}[3]{%
  \ifthmboxes
    \begingroup
    \edef\thmbox@opts{\thmboxsetup{#2}}%
    \expandafter\endgroup
    \expandafter\newthmboxstyle@i\expandafter{\thmbox@opts}{#1}{#3}%
  \else
    \declaretheoremstyle[
      headfont=\normalfont\bfseries,
      bodyfont=#3,
    ]{#1}%
  \fi
}
\newcommand{\newthmboxstyle@i}[3]{%
  \declaretheoremstyle[
    headfont=\normalfont\bfseries,
    bodyfont=#3,
    mdframed={#1},
  ]{#2}%
}

\newthmboxstyle{box-theorem}{clr@theorem}{\itshape}
\newthmboxstyle{box-proposition}{clr@proposition}{\itshape}
\newthmboxstyle{box-lemma}{clr@lemma}{\itshape}
\newthmboxstyle{box-corollary}{clr@corollary}{\itshape}
\newthmboxstyle{box-claim}{clr@claim}{\itshape}
\newthmboxstyle{box-fact}{clr@fact}{\itshape}
\newthmboxstyle{box-observation}{clr@observation}{\itshape}
\newthmboxstyle{box-question}{clr@question}{\itshape}
\newthmboxstyle{box-conjecture}{clr@conjecture}{\itshape}
\newthmboxstyle{box-openproblem}{clr@openproblem}{\itshape}
\newthmboxstyle{box-definition}{clr@definition}{\normalfont}
\newthmboxstyle{box-construction}{clr@construction}{\normalfont}
\newthmboxstyle{box-example}{clr@example}{\normalfont}
\newthmboxstyle{box-remark}{clr@remark}{\normalfont}

\makeatother

\ifproofbar
  \surroundwithmdframed[
    topline=false,
    bottomline=false,
    rightline=false,
    leftline=true,
    linewidth=1.2pt,
    linecolor=clr@proof,
    innertopmargin=2pt,
    innerbottommargin=2pt,
    innerleftmargin=10pt,
    innerrightmargin=0pt,
    skipabove=6pt,
    skipbelow=6pt,
  ]{proof}
\fi

\NewEnviron{commentfigure}[1][htbp]{%
  \ifbool{rambling}{%
    \begin{figure}[#1]%
      \BODY
    \end{figure}%
  }{}%
}

\newcommand{\sidecomment}[1]{%
  \iframbling
    \marginnote{\raggedright\footnotesize
	\color{clr@comment}#1}%
  \fi
}

\iframbling
  \NewEnviron{rambling}{%
    \begin{mdframed}[
        linecolor=clr@comment,
        backgroundcolor=clr@comment!8,
        topline=false, bottomline=false, rightline=false, leftline=true,
        linewidth=1pt,
        roundcorner=0pt,
        innertopmargin=5pt, innerbottommargin=5pt,
        innerleftmargin=8pt, innerrightmargin=8pt,
        skipabove=6pt, skipbelow=6pt,
      ]%
	  \small
      \textcolor{clr@comment}{\bfseries Commentary.}\hspace{0.5em}%
      \ignorespaces\BODY
    \end{mdframed}%
  }
\else
  \NewEnviron{rambling}{}
\fi

\declaretheorem[style=box-theorem,numberwithin=section,name=Theorem]{theorem}
\declaretheorem[style=box-proposition,sibling=theorem,name=Proposition]{proposition}
\declaretheorem[style=box-lemma,sibling=theorem,name=Lemma]{lemma}
\declaretheorem[style=box-corollary,sibling=theorem,name=Corollary]{corollary}
\declaretheorem[style=box-conjecture,sibling=theorem,name=Conjecture]{conjecture}
\declaretheorem[style=box-claim,sibling=theorem,name=Claim]{claim}
\declaretheorem[style=box-fact,sibling=theorem,name=Fact]{fact}

\declaretheorem[style=box-question,sibling=theorem,name=Question]{question}
\declaretheorem[style=box-observation,sibling=theorem,name=Observation]{observation}

\declaretheorem[style=box-theorem,numbered=no,name=Theorem]{theorem-n}
\declaretheorem[style=box-proposition,numbered=no,name=Proposition]{proposition-n}
\declaretheorem[style=box-lemma,numbered=no,name=Lemma]{lemma-n}
\declaretheorem[style=box-corollary,numbered=no,name=Corollary]{corollary-n}
\declaretheorem[style=box-conjecture,numbered=no,name=Conjecture]{conjecture-n}
\declaretheorem[style=box-claim,numbered=no,name=Claim]{claim-n}
\declaretheorem[style=box-fact,numbered=no,name=Fact]{fact-n}
\declaretheorem[style=box-openproblem,numbered=no,name=Open Problem]{openproblem-n}
\declaretheorem[style=box-question,numbered=no,name=Question]{question-n}
\declaretheorem[style=box-observation,numbered=no,name=Observation]{observation-n}

\declaretheorem[style=box-remark,sibling=theorem,name=Remark]{remark}
\declaretheorem[style=box-definition,sibling=theorem,name=Definition]{definition}

\declaretheorem[style=box-example,sibling=theorem,name=Example]{example}

\declaretheorem[style=box-remark,numbered=no,name=Remark]{remark-n}
\declaretheorem[style=box-definition,numbered=no,name=Definition]{definition-n}
\declaretheorem[style=box-construction,numbered=no,name=Construction]{construction-n}
\declaretheorem[style=box-example,numbered=no,name=Example]{example-n}

\newcommand{\defined}{\mathrel{\coloneqq}}

\DeclarePairedDelimiter{\p}{\lparen}{\rparen}

\newcommand{\st}{\mathbin{\colon}}

\undef{\set}
\DeclarePairedDelimiter{\set}{\lbrace}{\rbrace}

\undef{\emptyset}
\newcommand{\emptyset}{\varnothing}

\DeclarePairedDelimiter{\card}{\lvert}{\rvert}

\newcommand{\union}{\mathbin{\cup}}
\newcommand{\inter}{\mathbin{\cap}}

\newcommand{\from}{\colon}

\newcommand{\setm}[1]{\setminus\set{#1}}

\undef{\abs}
\DeclarePairedDelimiterX{\abs}[1]
  {\lvert}{\rvert}{\ifblank{#1}{\,\cdot\,}{#1}}

\undef{\norm}
\DeclarePairedDelimiterX{\norm}[1]
  {\lVert}{\rVert}{\ifblank{#1}{\,\cdot\,}{#1}}

\DeclarePairedDelimiterX{\inner}[2]
  {\langle}{\rangle}{\ifblank{#1}{\,\cdot\,}{#1},\ifblank{#2}{\,\cdot\,}{#2}}

\DeclareMathDelimiter{\given}
  {\mathbin}{symbols}{"6A}{largesymbols}{"0C}

\DeclareMathOperator{\Prob}{\mathbb{P}}
\DeclarePairedDelimiterXPP{\prob}[1]
  {\Prob}{\lparen}{\rparen}{}
  {\renewcommand{\given}{\nonscript\;\delimsize\vert\nonscript\;\mathopen{}}#1}

\DeclareMathOperator{\Expec}{\mathbb{E}}
\DeclarePairedDelimiterXPP{\expec}[1]
  {\Expec}{\lparen}{\rparen}{}
  {\renewcommand{\given}{\nonscript\;\delimsize\vert\nonscript\;\mathopen{}}#1}

\DeclareMathOperator{\Var}{Var}
\DeclarePairedDelimiterXPP{\var}[1]
  {\Var}{\lparen}{\rparen}{}
  {\renewcommand{\given}{\nonscript\;\delimsize\vert\nonscript\;\mathopen{}}#1}

\DeclareMathOperator{\Cov}{Cov}
\DeclarePairedDelimiterXPP{\cov}[2]
  {\Cov}{\lparen}{\rparen}{}{#1,#2}

\newcommand{\eps}{\varepsilon}
\newcommand{\sseq}{\subseteq}
\newcommand{\ssneq}{\varsubsetneq}
\let\l\relax
\newcommand{\l}{\ell}

\newcommand{\NN}{\mathbb{N}}

\newcommand{\QQ}{\mathbb{Q}}
\newcommand{\RR}{\mathbb{R}}

\newcommand{\ZZ}{\mathbb{Z}}

\newcommand{\cA}{\mathcal{A}}
\newcommand{\cB}{\mathcal{B}}

\newcommand{\cF}{\mathcal{F}}

\newcommand{\cI}{\mathcal{I}}

\newcommand{\cS}{\mathcal{S}}

\newcommand{\cW}{\mathcal{W}}

\let\ll\relax

\newcommand{\zz}{\zeta}

\newcommand{\kk}{\kappa}
\newcommand{\ll}{\lambda}

\newcommand{\ww}{\omega}

\usepackage{centernot}

\usepackage{tikz}
\usetikzlibrary{quotes} 
\usetikzlibrary{calc}

\usepackage{scrextend}

\usepackage[T1]{fontenc}
\usepackage{titlesec}

\titleformat{\section}{\centering\bfseries\scshape\Large}{\thesection}{1em}{}
\titleformat{\subsection}{\bfseries\scshape\large}{\thesubsection}{1em}{}

\definecolor{clanking}{HTML}{a85f15} 

\newcommand{\esmc}{$\eta$-strongly maximal chain}

\newcommand{\wlg}{without loss of generality}

\newcommand{\comp}{\sim}
\newcommand{\incomp}{\parallel}

\newcommand{\leqlex}{\leq_{\text{lex}}}
\newcommand{\lextimes}{\times_{\text{lex}}}
\newcommand{\RRP}{\overline{\mathbb{R}}}
\newcommand{\ZZP}{\overline{\mathbb{Z}}}
\newcommand{\quo}[1]{#1\,/\finsim_{#1}}
\newcommand{\canon}[3]{#1 \from \quo{#2} \to \quo{#3}}
\newcommand{\mfi}{mutually finite incomparability}

\makeatletter
\newcommand{\oset}[3][0ex]{%
  \mathrel{\mathop{#3}\limits^{
    \vbox to#1{\kern-2\ex@
    \hbox{$\scriptstyle#2$}\vss}}}}
\makeatother

\newcommand{\covers}{\gtrdot}
\newcommand{\covered}{\lessdot}

\newcommand{\finsim}{\oset{*}{\sim}}

\newcommand{\etalt}{\preceq_\eta}
\newcommand{\etawitness}[3]{#1 \from #2 \to \cI(#3)}

\newcommand{\atomicincr}{N_+^0}
\newcommand{\atomicdecr}{N_-^0}
\newcommand{\atomiceps}{N_\eps^0}
\newcommand{\atomic}{N^0}

\newcommand{\consolidationof}{\unlhd}
\newcommand{\consolidatesto}{\unrhd}
\newcommand{\consolidationwitness}[3]{#1 \from \atomic(#2) \to \atomic(#3)}

\let\above\relax
\newcommand{\above}[1]{(\,>\!#1)}
\newcommand{\aboveeq}[1]{(\,\geq\!#1)}
\let\below\relax
\newcommand{\below}[1]{(\,<\!#1)}
\newcommand{\beloweq}[1]{(\,\leq\!#1)}
\newcommand{\compset}[1]{(\,\comp\!#1)}
\newcommand{\incompset}[1]{(\,\incomp\!#1)}

\DeclareMathOperator{\hite}{ht}

\let\and\relax
\newcommand{\and}{\wedge}

\DeclareMathOperator{\Conv}{Conv}

\begin{document}

\title{\textsc{\bfseries The structure of FAC posets and the Aharoni--Korman conjecture}}

\renewcommand{\thefootnote}{\fnsymbol{footnote}}

\author{\textsc{Lawrence Hollom}\footnotemark[1]}

\footnotetext[1]{\href{mailto:lawrence.hollom@epfl.ch}{lawrence.hollom@epfl.ch}. Institute of Mathematics, EPFL, Lausanne, Switzerland}

\renewcommand{\thefootnote}{\arabic{footnote}}

\maketitle

\date{}



\sidecomment{This is an annotated form of the paper. There are extra decorations which are supposed to make it easier to scan, and side-comments (like this one) provide informal and non-rigorous commentary and thoughts.}
\begin{abstract}
    A poset $P$ is said to satisfy the \emph{finite antichain condition}, or \emph{FAC} for short, if it has no infinite antichain.
    Such posets exhibit rich and complex structure, and it was conjectured by Aharoni and Korman in 1992 that any FAC poset $P$ possesses a chain $C$ and a partition into antichains such that $C$ meets every antichain of the partition.
    While a counterexample to this conjecture was recently discovered, in this paper we prove that the conjecture does hold true for a broad class of countable posets.
    In particular, we prove that the Aharoni--Korman conjecture holds for countable FAC posets containing no saturated chain $D$ such that either $D$ or its reverse $D^*$ is of the form $\bigoplus_{x\in\omega} D_x$, where each $D_x$ is infinite and co-well-founded.

	The proof of this result rests on novel structural theory for countable FAC posets, which may be of independent interest.
	One result states that every countable FAC poset is either scattered (contains no copy of the rationals) or may be broken up into a dense collection of scattered posets which reflect the structure of the poset as a whole. 
	This provides a general reduction from arbitrary countable FAC posets to the far more tractable scattered ones.

	We also develop an extension $H(P)$ of a poset $P$ which includes ``limit points'' of $P$ and allows for intricate arguments involving these points. 
	Amongst other points, one way in which $H(P)$ differs from existing poset completions and extensions is that it is sensitive to the order-type of a chain.

	Finally, we highlight a number of intriguing problems in this area which remain open.
\end{abstract}
\clearpage

\begingroup
\hypersetup{linkcolor=black}
\tableofcontents
\endgroup

\clearpage


\section{Introduction}
\label{sec:intro}
One of the most fundamental combinatorial results on posets is Dilworth's theorem, proved in 1950 \cite{dilworth1950decomposition}, which may be phrased as saying that every finite poset $P$ contains an antichain $A$ and admits a partition into disjoint chains such that every chain in the partition meets $A$.
We recall that a subset $X$ of a poset $P$ is a \emph{chain} if the elements of $X$ are pairwise comparable, and it is an \emph{antichain} if its elements are pairwise incomparable.

\sidecomment{It really is a theme across all of this stuff that life is much harder in the infinite case. In our setting, we can also only deal with the countable case; perhaps larger infinities are more difficult again.}
One central question is the extent to which the above theorem, and other results on finite posets, generalise to the infinite setting.
Certainly some extra condition is needed to extend the above structural formulation of Dilworth's theorem to the infinite setting, as all antichains in the poset $\ww\times\ww$ are finite, but any partition into chains is infinite.
One possible fix for this issue, again due to Dilworth \cite{dilworth1950decomposition}, is to place an upper-bound on the largest antichain of $P$, in which case Dilworth's theorem does generalise to infinite posets. 
In a different direction, it was proved by Aharoni and Korman \cite{aharoni1992greene} that adding an assumption of no infinite chains also suffices.
The following is a special case of this theorem.

\begin{theorem}[\cite{aharoni1992greene}]
    \label{thm:infinite-greene-kleitman-special}
    Let $P$ be a poset with no infinite chains. There exists then a partition $(C_i\st i\in I)$ of $P$ into disjoint chains and an antichain $A\sseq P$ such that every chain $C_i$ meets $A$.
\end{theorem}

Moreover, Aharoni and Korman asked about the natural dual to this problem, conjecturing that this should also hold.

\sidecomment{One key takeaway from all of this could be that insisting that a poset has no infinite antichain leads to a world with lots of interesting structure.}
\begin{conjecture}[\cite{aharoni1992greene}]
    \label{conj:ak}
    If a poset $P$ contains no infinite antichain then there exists a chain $C$ and a partition of $P$ into disjoint antichains $(A_i\st i\in I)$ such that each $A_i$ meets $C$.
\end{conjecture}

This condition of having no infinite antichains is an important property that we will consider extensively; if $P$ has no infinite antichain, then we say that it satisfies the \emph{finite antichain condition}, and for brevity we will refer to $P$ as an \emph{FAC poset}.
In fact, a more general conjecture was posed, concerning $k$ disjoint chains, which will be discussed in due course.

\sidecomment{i.e.\ $P$ is the standard product poset $\ww\times\ww$.}
For example, if $P$ is the poset on the set $\NN\times\NN$ with ordering given by setting $(x,y)\leq(u,v)$ if and only if $x\leq u$ and $y\leq v$, then \Cref{conj:ak} holds by taking $C = \set{(0,y)\st y\in\NN}$ and $A_i=\set{(x,y)\in P\st x+y=i}$ for all integers $i\geq 0$.
This conjecture is often referred to as \emph{the Aharoni--Korman conjecture} or \emph{the fishbone conjecture}; we will use the former name.
This conjecture is now known to be false: the author \cite{Hol26} constructed a counterexample, and this disproof has been formally verified in the Lean4 theorem prover by Mehta \cite{mehta2025formal}.
However, this still leaves open the question of when the conclusion of the Aharoni--Korman conjecture does hold.
Indeed, we show that the conclusion of \Cref{conj:ak} does hold for a broad class of countable FAC posets: those which have a property which we will call ``vacillating''.
Together with the counterexample in \cite{Hol26}, this provides a thorough resolution of the conjecture.
\sidecomment{There is reason behind the name ``vacillation'', and this is discussed in due course.}

Our notation mirrors that of Stanley \cite{stanley2011enumerative} when standard, but we repeat the definitions here for completeness.
Indeed, we say that a chain $C$ in a poset $P$ is \emph{saturated} if there is no $x\in P\setminus C$ and $y,z\in C$ with $y<x<z$ and $C\union\set{x}$ a chain.
The \emph{linear sum} $\bigoplus_{i\in I}P_i$ of a collection $(P_i\st i\in I)$ of posets indexed by a total order $I$ is the ordering on $\bigsqcup_{i\in I}P_i$ wherein all elements of $P_j$ are set to be greater than all elements of $P_i$ whenever $j > i$ in $I$.
\sidecomment{The key non-standard definition in here is that of being \emph{parallel}. This is surprisingly effective at streamlining later proofs; before this definition was introduced, there was a lot of ``if $x=y$ then do this, if $x$ and $y$ are incomparable then do that'' which is now pared down to a single case.}
A poset is \emph{well-founded} if it has no infinite decreasing sequence, and \emph{co-well-founded} if it has no infinite increasing sequence.

We write $x\incomp y$ to mean that $x,y\in P$ are incomparable or equal (we refer to such elements as \emph{parallel}), and $x \comp y$ to mean that $x$ and $y$ are comparable (i.e.\ $x\leq y$ or $y\leq x$).
Finally, $\ww$ is the order-type of the naturals.

Our first main result shows that \Cref{conj:ak} holds for a large class of posets.

\begin{theorem}
    \label{thm:main-on-first-page}
    Let $P$ be a countable FAC poset such that, for any saturated chain $C\sseq P$, neither $C$ nor $C^*$ can be written as $\bigoplus_{i\in\ww} C_i$ with each $C_i$ infinite and co-well-founded.
    Then there is a chain $C\sseq P$ and a partition $P=\bigcup_{i\in I} A_i$ into antichains such that $C$ meets every antichain $A_i$.
\end{theorem}

\Cref{thm:main-on-first-page} will be deduced from a series of structural results concerning countable FAC posets, which we believe to be of independent interest.
Chief among these is the following theorem, which shows that an arbitrary countable FAC poset is either \emph{scattered} -- that is, contains no suborder isomorphic to the rationals -- or is organised around a densely ordered family of scattered pieces, each of which is convex, which captures much of the structure of the whole poset.

\sidecomment{This result looks big and scary, but it really is just saying we can chop up a poset into a reals-indexed chain of scattered blobs, and some stuff outside which won't matter. The conditions are all just formalisations of what it means to be ``a blob''.}
\begin{theorem}
\label{thm:structural}
    Let $P$ be a countable FAC poset and let $\RRP$ be the extended reals, that is, $\RR\union\set{-\infty, +\infty}$.
    Then either $P$ is scattered, or there is a collection $\set{A_r\st r \in \RRP}$ of subsets of $P$ satisfying the following properties.
    \begin{itemize}
		\itemsep=0pt
        \item For all $q\in \QQ$, the set $A_q$ is non-empty,
        \item For all $r,s\in\RRP$ with $r<s$, we have $A_r < A_s$; these sets are in particular disjoint,
        \item For all $x \in P \setminus \bigcup_{r\in\RRP} A_r$, there is an infinite interval $I_x \sseq \RRP$ such that $x$ is incomparable to all of $\,\bigcup_{y\in I_x} A_y$,
        \item For all $r\in \RRP$, $A_r$ is a convex subset of $P$, and 
        \item For all $r\in \RRP$, the poset $A_r$ is scattered.
    \end{itemize}
\end{theorem}

Loosely speaking, \Cref{thm:structural} allows questions about general countable FAC posets to be reduced to the scattered case, where powerful tools such as Hausdorff's classical classification of scattered linear orders \cite{hausdorff1908classification} become available; this reduction will be a key step in our proof of \Cref{thm:main-on-first-page}.




We now review the existing literature on this topic, and then provide a more thorough overview of our contributions.


\subsection{Known results}
\label{subsec:lit-review}

Returning to Dilworth's theorem, as discussed at the start of the introduction, a standard phrasing of the result is that every finite poset $P$ of width $w$ has a partition into $w$ chains, where we recall that the \emph{width} of a poset $P$ is the supremum over the sizes of antichains in $P$.
As will become a theme throughout this introduction, the generalisation to infinite posets presents some further challenges.
Indeed, Perles \cite{perles1963dilworth} noted that the strongest ``infinite Dilworth's theorem'' one could hope for -- that any poset of width at most $\ll$ can be partitioned into at most $\ll$-many chains -- is false, as, for example, for any infinite cardinal $\kk$, the poset $\kk\times\kk$ has arbitrarily large finite antichains but no infinite antichain, and thus width $\aleph_0$, but cannot be partitioned into fewer than $\kk$-many chains.
In this direction, Abraham \cite{abraham1987note} characterised which FAC posets could be partitioned into finitely many chains, which are in particular those posets which do not contain a sub-poset of a certain type similar to the product posets discussed above.

In 1976, Greene and Kleitman \cite{greene1976structure} generalised Dilworth's theorem to families of antichains, proving that, for any finite poset $P$ and integer $k$, the maximal size of the union of $k$ antichains of $P$ is equal to the minimal value of $\sum_{i\leq m}\min(\abs{C_i},k)$ over all partitions $\set{C_1,\dotsc,C_m}$ of $P$ into chains.
The Greene--Kleitman theorem was then generalised to the infinite setting by Aharoni and Korman \cite{aharoni1992greene}.
More precisely, they proved the following theorem, which they described as ``the `correct' infinite version of Greene--Kleitman's theorem.''

\begin{theorem}[\cite{aharoni1992greene}, Theorem 3.1]
    \label{thm:infinite-greene-kleitman}
    Let $P$ be a poset with no infinite chains and let $k$ be a positive integer. There exists then a partition $(C_i\st i\in I)$ of $P$ into disjoint chains and $k$ disjoint antichains $A_j$ $(1 \leq j \leq k)$ such that every chain $C_i$ meets $\min(k, \abs{C_i})$ antichains $A_j$.
\end{theorem}

Having proved \Cref{thm:infinite-greene-kleitman}, Aharoni and Korman then asked about the dual version, interchanging the roles of chains and antichains, which led them to \Cref{conj:ak}, and the more general version, which we state here.

\sidecomment{This is the general form of the conjecture, but everyone works with the $k=1$ case. I'm not aware of any work seriously studying the $k\geq 2$ case; is there genuinely more going on there?}
\begin{conjecture}[\cite{aharoni1992greene}, Conjecture 4.1]
	\label{conj:ak-general}
    If a poset $P$ contains no infinite antichain then, for every positive integer $k$, there exist $k$ chains $C_1,\dotsc,C_k$ and a partition of $P$ into disjoint antichains $(A_i\st i\in I)$ such that each $A_i$ meets $\min(\abs{A_i},k)$ chains $C_j$.
\end{conjecture}

\Cref{conj:ak} was later referred to by Aharoni as the ``fishbone conjecture'' (see for example \cite{aharoni2022strongly}, apparently due to how a schematic diagram of a chain intersecting a family of antichains resembles the skeleton of a fish), but is often called ``the Aharoni--Korman conjecture''.
The inspiration for the conjecture in fact came from the consideration of strongly minimal covers and strongly maximal matchings in hypergraphs (see for example \cite{aharoni1991infinite}), and has close links to infinite graph theory as a whole.
One such related area is the generalisation of Menger's theorem to infinite graphs.
This problem has a rich history, and led to many partial results \cite{aharoni1987menger,aharoni1994menger,diestel2003countable} before the seminal paper of Aharoni and Berger \cite{aharoni2009menger} proved the full generalisation.
Indeed, in \cite{aharoni2009menger}, \Cref{conj:ak} was described as ``one of the most attractive'' conjectures in the area of infinite matching theory.

Aharoni and Korman also posed several other, more general conjectures in \cite{aharoni1992greene}.
In 1995, Aharoni and Loebl \cite{aharoni1995strongly} resolved one of these conjectures, concerning perfect graphs, in some special cases.
\sidecomment{It's really not looking good for general conjectures positing the existence of strongly maximal/minimal structures anywhere!}
The strongest of the conjectures from \cite{aharoni1992greene}, concerning strongly minimal covers of hypergraphs, has been disproved by van der Zypen \cite{van2022counterexample}.
More recently, the author and Randall Shaw \cite{HR26} constructed counterexamples to several further conjectures in this area, concerning strongly maximal and strongly minimal structures.

Returning to \Cref{conj:ak}, before the conjecture was resolved in the negative \cite{Hol26}, there was a line of results proving that the conjecture does hold for some classes of poset.
The first of these was due to Aharoni and Korman \cite{aharoni1992greene} in the same paper which introduced the problem, in which they proved that it holds for posets of width at most $2$ by application of an infinite version of K\"{o}nig's theorem due to Aharoni \cite{aharoni1984konig}.

The next result after this was due to Duffus and Goddard \cite{duffus2002intervals} who proved that \Cref{conj:ak} holds in two different cases.
Firstly, they proved the conjecture for posets of the form $C\times P$, where $C$ is a chain and $P$ is finite (this result in fact first appeared in Goddard's doctoral thesis \cite{goddard1996ordered}).
Secondly, they proved the conjecture for FAC posets with no infinite intervals, as follows.

\begin{theorem}[\cite{duffus2002intervals}, Theorem 4.1]
    \label{thm:no-infinite-interval}
    If $P$ is an ordered set with no infinite intervals and no infinite antichains then there is a partition of $P$ into antichains and there is a chain of $P$ that intersects every member of the partition.
\end{theorem}

Since then, Zaguia has recently proved a series of results, now collected into a single paper \cite{zaguia2024progress}, showing that \Cref{conj:ak} is true when the FAC poset $P$ is $N$-free (that is, if $w,x,y,z\in P$ have $w\leq y$, $x\leq y$, and $x\leq z$, then some further comparison must hold between these four points), or has locally finite incomparability graph.
The \emph{incomparability graph} of a poset $P$ is a graph $G$ with vertex set $P$, and distinct vertices $x,y$ are adjacent in $G$ if and only if $x\incomp y$ in $P$.
We will also find use for the latter of these two results, which Zaguia deduced from \Cref{thm:no-infinite-interval}.
For ease of referencing the theorem, we now state this result here as well.

\sidecomment{This theorem will have a much shorter statement in the language we develop; see \Cref{thm:zaguia-spine-reformed}.}
\begin{theorem}[\cite{zaguia2024progress}, Theorem 6]
    \label{thm:zaguia-spine}
    If $P$ is a poset whose incomparability graph is locally finite, then there is a partition of $P$ into antichains and there is a chain of $P$ that intersects every member of the partition.
\end{theorem}

For completeness, we will also give a proof of \Cref{thm:zaguia-spine} using our machinery in \Cref{sec:tubes}.

\sidecomment{In general, it seems that completions of these posets often behave in unexpected ways; I find that I am still often surprised by what the completion/extension of a fairly nice-looking poset can managed to do.}
Taking a more broad view, beyond just partial results on \Cref{conj:ak}, and as mentioned at the start of this introduction, the rich structure of general FAC posets has received only sparse attention.
One topic of particular interest has concerned completions of these posets, that is, extensions of the posets for which certain families of subsets are guaranteed to have suprema and infima.
Duffus, Pouzet, and Rival \cite{duffus1981complete} studied several different extensions of posets, and gave an example in which the MacNeille completion of a poset $P$ (the minimal complete poset containing $P$) contains an infinite antichain even when $P$ has no antichain of order $3$.
Some further results on this topic were obtained by Lawson, Mislove, and Priestley \cite{lawson1987ordered}.


\subsection{Our contributions}
\label{subsec:our-results}

The contributions of this paper are threefold.
First, we prove that the Aharoni--Korman conjecture holds for a class of countable posets far broader than any previously known.
Second, in the course of doing so, we establish several general structural results on countable FAC posets -- most notably \Cref{thm:structural} -- which we believe to be of interest in their own right, independently of their applications to the Aharoni--Korman conjecture.
Thirdly, we construct the ``chain extension'' $H(P)$ in \Cref{sec:completions}, which seems to be the correct tool for working closely with chains and their limit points in scattered infinite posets.
In particular, they are sensitive to the order-type of the chains (unlike most poset/chain completions, which care more about the down-set generated by the chain).

We begin with a discussion of our work on the Aharoni--Korman conjecture.
As is the case in all work on this conjecture to date, we focus on the special case \Cref{conj:ak}, where the value of $k$ in the more general \Cref{conj:ak-general} is taken to be 1.
To ease the discussion of these results, we make the following definition.

\sidecomment{The name spine comes from the name ``the fishbone conjecture'': the chain $C$ is the spine of the fish's skeleton, and the antichains are the ribs.}
\begin{definition}
    \label{def:spine}
    Given a poset $P$, we say that a subset $C\sseq P$ is a \emph{spine} of $P$ if $C$ is a chain and there is a partition of $P$ into antichains $(A_i\st i\in I)$ such that $C$ meets every antichain $A_i$.
\end{definition}

In this language, the Aharoni--Korman conjecture states that every FAC poset has a spine.
While work of the author \cite{Hol26} has demonstrated that this is false in general, the counterexample constructed there is rather delicate and specific, which suggests that the failure of the conjecture may be the exception rather than the rule.
Our main result confirms this suggestion for countable posets: we identify a single structure whose absence implies that the conjecture holds.
We call posets avoiding this structure ``vacillating''. 
The intuition for this name is that chains in a scattered vacillating posets must, in a certain sense, alternate between well-founded and co-well-founded pieces (see \Cref{lem:witness-vacillation}).
Formally, the definition is as follows.

\sidecomment{Note the ``saturation'' in this definition: we can have an infinite decreasing sequence of infinite well-founded pieces, as long as they're broken up by co-well-founded pieces.}
\begin{definition}
    \label{def:vacillating}
    We say that a poset $P$ is \emph{vacillating} if there is no saturated chain $C$ in $P$ for which either $C$ or its reverse $C^*$ can be written as $\bigoplus_{i\in \ww} C_i$, where each $C_i$ is infinite and co-well-founded.
\end{definition}

For example, the Cartesian product $\ww\times \ww^*$ is vacillating, as it contains no copy of either $\ww + 1$ or $(\ww + 1)^*$, and so we cannot find an infinite linear sum of infinite chains.
However, the lexicographic product of $\ww$ and $\ww^*$ is not vacillating, as it can by definition be expressed as $\bigoplus_{i\in \ww} \ww^*$.

With \Cref{def:vacillating} in hand, we can restate \Cref{thm:main-on-first-page} as follows.

\sidecomment{And here's why it's worth defining all this terminology; now the statement is much less of a mouthfull.}
\begin{theorem}
    \label{thm:main}
    If $P$ is a countable vacillating FAC poset, then $P$ contains a spine.
\end{theorem}

We note before continuing that a poset which fails to be vacillating must contain an infinite interval, and that any FAC poset with no infinite interval is countable (as noted in \cite{duffus2002intervals}, see also \Cref{fact:infinite-chain}), and so \Cref{thm:main} implies \Cref{thm:no-infinite-interval} of Duffus and Goddard.

We now give a high-level overview of our proof strategy for proving \Cref{thm:main}; in \Cref{subsec:outline} we give a more thorough section-by-section breakdown of how the proof proceeds.

The proof of \Cref{thm:main} works in several steps, whittling the problem down to progressively more amenable structures until we reach a statement we can prove. 
In particular, we reduce the problem of finding a spine into a different problem, of finding a \emph{maximal tube}, the structure of which is easier to work with.
We then break up the poset in question into simpler pieces; \Cref{thm:structural} allows us to reduce to the case of scattered posets, and from there we reduce to the case of posets with a property which we call \emph{quasifoundedness}: $P$ is quasifounded if it contains neither $\bigoplus_{i\in \ww} \ww^*$ nor its reverse as a sub-chain (unlike vacillation, which requires this sub-chain to be saturated).
Both of these reductions work by chopping up a large poset into smaller pieces and showing that, if we can find a maximal tube in all the small pieces, then we can find a maximal tube in the whole.

Quasifounded posets are then suitably easy to work with that we can find a structure which we call a \emph{consolidated chain}.
We will give intuition for this in \Cref{subsec:outline}, but for now all we mention is one of the key properties of such a chain: from it one can build a maximal tube.
This allows us to find a maximal tube in any quasifounded sub-poset of $P$, and thus by the work discussed in the previous paragraph, in our original poset $P$.

On the way to proving this result, we also construct a novel extension $H(P)$ of a (scattered) poset $P$, which allows for arguments to work with ``limit points'' of $P$.


\subsection{Proof outline and motivation}
\label{subsec:outline}

We now give a more thorough outline of how we prove \Cref{thm:main}, and the tools that we build up along the way.
This outline does not follow the order of the paper (which is designed for ease of presenting everything formally), but in an order which the ideas naturally arise.

\sidecomment{An important thing to keep in mind here: there is no condition that the antichains in the partition $(A_i \st i \in I)$ must be ``ordered nicely''; unlike the image of the fish skeleton, we can have $A_i$ and $A_j$ in the partition where some element of $A_i$ is above an element of $A_j$ and vice-versa.}
We want to find a spine in a poset.
This is challenging for many reasons, but one of these is that we need to identify not only a chain, but a partition of the entire poset into antichains.
For this reason, we seek to reduce the problem of finding a spine into a more direct structural problem; the tool we use for this is a \emph{maximal tube}.
To motivate this, note that if we have countable posets $Q\sseq P$ and a partition $Q = \bigcup_i A_i$ into antichains, and every element of $P\setminus Q$ is entirely incomparable with infinitely many $A_i$, then we can easily extend this partition to cover all of $P$: enumerate $P \setminus Q$ and add the elements to antichains one at a time.
As every element has infinitely many options for where it can be added, we can always continue extending the partition.

We would thus like to find a structure $Q\sseq P$ such that we can find a spine in $Q$ and extend the antichain partition in the above way to all of $P$.
We know due to a result of Zaguia \cite{zaguia2024progress} that, if all elements of a poset $P$ are incomparable with only finitely many others, then $P$ has a spine. 
For that reason, we hope for the most ambitious thing we can:\sidecomment{Think: chains can be placed in tubes.} a set $T\sseq P$ (a \emph{tube}) where all elements of $T$ are incomparable to only finitely many other elements of $T$, and all elements of $P\setminus T$ are incomparable with infinitely many elements of any antichain partition of $T$.
In fact, it turns out that the latter condition follows simply from every element of $P\setminus T$ being incomparable with infinitely many elements of $T$; note that this means that $T$ cannot be extended by adding an element while remaining a tube.
This result is proved in \Cref{sec:tubes}.
We thus call such objects \emph{maximal tubes}.

\sidecomment{I try to motivate this step here, but it is a bit difficult as it's not obvious that either we should be able to find maximal tubes or that their spines should extend to all of $P$. Discovering them as a good idea took quite some time in the project.}
It thus suffices to find a maximal tube; all internal elements have finite incomparability, and external elements have infinite incomparability.
One thing to note about these objects is that they ``glue together'' nicely: if we can cut $P = P_1 \union P_2$ with $P_1 > P_2$ and find maximal tubes in each of $P_1$ and $P_2$, then we can take the union of these tubes to find a maximal tube in $P$.
It is for this reason that we develop the structural machinery of \Cref{thm:structural}: it allows us to break a non-scattered poset up into comparable scattered ``chunks'', plus other elements which are incomparable with infinitely many chunks; it thus suffices to find a maximal tube in a scattered vacillating FAC poset.
This result is proved in \Cref{sec:structural}.

Now consider how we might go about trying to build a maximal tube.
One idea would be to start with some saturated chain and build $T$ by iteratively throwing in elements which are incomparable to only finitely many elements already inside $T$.
However, this process can fail in the following way: say that we are building $T$, and at some stage it contains elements $x_0 < x_1 < x_2 < \dotsb$ and $P\setminus T$ contains $y_0 < y_1 < y_2 < \dotsb$, where $y_i \incomp x_i$ and $y_i < x_{i+1}$ for all $i$, but, for all $i$ and $j$, we have $x_i \not< y_j$.
\sidecomment{If you want a picture of this, look at just $X_1$ and $Y_0$ in \Cref{fig:example-2} (subscripts chosen so the ordering agrees).}
Assume for simplicity that these are the only elements in the poset $P$.
Every $y_i$ is incomparable with only the elements $x_0,x_1,\dotsc,x_i$ in $T$, and so should be added to $T$ by our above rule.
But when all $y_i$ are added to $T$, we run into a problem: $x_j$ is incomparable with $y_j,y_{j+1},\dotsc$, an infinite sequence, so $T$ is no longer a tube and our process has failed.

Notice that the above step had one increasing chain in $T$ cofinally above (but not entirely above) another increasing chain outside of $T$.
\sidecomment{I'm going to be a bit vague here because the jump from the idea to the formal definition of consolidation is a big one.}
Preventing this from happening is the intuition behind the relation of \emph{consolidation} that we consider in \Cref{sec:consolidation}: we would say that the chain $\set{x_i \st i < \ww}$ \emph{consolidates to} $\set{y_i \st i < \ww}$ and we want to find a chain which admits no consolidation.
There are two things here which are not at all obvious: what ``consolidation'' or ``a consolidated chain'' actually mean formally, and why such a chain would necessarily give rise to a maximal tube; the above obstruction is ruled out, but why can't there be some other issue?

The first of these to points uses some significant machinery: we essentially care about chains only up to replacing finite sections of them, and for this reason it is much more convenient to talk about the set of limit points of the chain, rather than the chain itself. 
\sidecomment{It probably should not be obvious how to do this formally; I say this knowing the number of different versions of the definition of $H(P)$ I went through before settling on the current one!}
The natural structure to use to do this is some kind of poset completion or chain completion, but the existing tools do not quite capture what we want: while they would distinguish the above chains $\set{x_i \st i < \ww}$ and $\set{y_i \st i < \ww}$, we can easily produce an example of a poset $P$ with saturated chains of order-type $\ww^2$ and $\ww$ respectively both inducing the same down-set, and so not being distinguished by standard poset completions, whereas we want to consider these chains as different (see \Cref{ex:n-cross-n-star}).
For this reason we develop a new structure to capture these limit points: the \emph{chain extension} $H(P)$, which we define in \Cref{sec:completions}. 
There is much which can be said about this structure, so we defer further discussion to that section.

\sidecomment{Note that, though we build a spine via constructing a consolidated chain, a consolidated chain need not be a spine (for one thing, it pays no attention to local finite differences, which are important for building a spine.)}
This machinery allows us, with some further definitions, to formally define the consolidation order, which we do in \Cref{sec:consolidation}.
It precisely captures the intuition we discussed above, and in \Cref{sec:finding-tubes} we are able to prove that we can indeed greedily build a maximal tube out of a consolidated chain: it thus suffices to find a consolidated chain in our poset.
The construction of a consolidated chain is performed in \Cref{sec:consolidation} by proving that, in certain situations, the consolidation order on limit points in the chain extension of $P$ is well-founded, and so we can build a chain by gluing together consolidation-minimal fragments (e.g. $\set{y_i \st i < \ww}$ in the above example.)

But what are the ``certain situations'' required by the above argument?
We wanted consolidation to be a well-founded partial order, so any infinite increasing sequence of decreasing chains (e.g.\ $\bigoplus_{n< \ww} \ww^*$) would be a problem here.
\sidecomment{The precise problem is when we can find a nontrivial consolidation of a chain to itself: split it into blocks and move every block in the same direction.}
Note that vacillation gives us that there is no saturated chain of this form, but here we want that there is no chain at all.
We call a poset avoiding such chains \emph{quasifounded}; this is a significantly stronger condition than vacillation, and it should not be at all obvious why we can reduce to this case.
This reduction is handled by \Cref{sec:separating}, which broadens the definition of maximal tubes into a concept which can be thought of as ``local maximality'': \emph{separating tubes}.
The idea here is that, if $Q\sseq P$ is a convex poset with no minimal or maximal elements, and $T\sseq Q$ is a maximal tube in $Q$, then $T$ still has good properties relative to $P$. 
In particular, every point of $P$ is either in $T$, incomparable with infinitely many elements of $T$, entirely above $T$, or entirely below $T$.
We call such a tube $T$ a separating tube.

This allows us to, provided that we can find a $Q\sseq P$ in the above sense, construct a separating tube of $P$ and split $P$ into two parts we need to consider: the part above $T$ and the part below $T$.
By iteratively running this process on these parts, we can build up a maximal tube one separating part at a time.
This argument is formalised in \Cref{sec:separating}.

With all these parts of the argument complete, we can then simply combine them to prove \Cref{thm:main}: start with any poset $P$, break it into scattered parts, break each scattered part up into quasifounded parts, find a consolidated chain in each of these parts, build a maximal tube around the consolidated chain, and then stick the maximal tubes found in each part together to build a maximal tube of $P$.
Then, as maximal tubes contain spines, $P$ has a spine.

Finally, we conclude in \Cref{sec:conclusion} with a discussion of the many open problems in the area, both those which remain open, and those which have been brought to light by the work presented here.


\subsection{AI usage and proof formalisation}

The author used a variety of commercially-available AI systems in the preparation of this paper. 
While all ideas were initially developed by the author, several approaches to streamline the proof (most notably the separating tubes discussed in \Cref{sec:separating}) resulted from interactions with AI systems. 
All the writing present in this document was written by the author (having understood all concepts involved); in the context of writing, AI was used only to check for correctness and to convert hand-drawn diagrams into tikz.

\sidecomment{To be abundantly clear, the fact that this formalisation exists is no guarantee that the arguments in the form presented here are correct (though I do believe that they are correct!)}
AI systems were also used to produce a formalisation of the results of this paper in Lean4, which can be accessed online at \cite{Lean}. 
The verification contains a version of every lemma, proposition, and theorem in this paper, and the statements have been checked by the author to align with those here.
However, we note that the proof strategies do in some places deviate between the formalised and present versions.


\section{Preliminaries}
\label{subsec:basics}

We now recall some fundamental facts and definitions that we will make use of throughout the rest of this paper.
We are concerned here with posets $P$ containing no infinite antichain, although their \emph{width} -- the supremum of the sizes of antichains in $P$ -- may be infinite.

For any order $X$, we will write $X^*$ to denote the \emph{reverse ordering}.
If chains $C,D\sseq P$ have $C\sseq D$, then we will say that $C$ is \emph{cofinal} in $D$ if for all $y\in D$ there is $x\in C$ such that $y\leq x$.
We will say that $C$ is a \emph{final segment} of $D$ if $C$ is saturated in $D$ and if, whenever $x,y\in D$ have $x\leq y$ and $x\in C$, then we have $y\in C$.
The terms \emph{coinitial} and \emph{initial segment} are then the dual statements, with the ordering reversed.
If $X,Y\sseq P$ are arbitrary subsets of $P$, then by $X<Y$, we mean that for all $x\in X$ and $y\in Y$, we have $x<y$.

\sidecomment{This notation is non-standard, but I like it; it's very clear whether or not the point $x$ is included, for one thing.}
We will denote the \emph{up-set} of $x$ -- the set of all elements in $P$ above $x$ -- as $\above{x}_P$ and the \emph{down-set} as $\below{x}_P$.
Formally, for any subset $Q\sseq P$, we have
\begin{equation*}
	\above{x}_Q\defined \set{y\in Q\st y > x} \quad \text{ and } \quad \below{x}_Q \defined \set{y\in Q\st y < x}.
\end{equation*}
The sets $\aboveeq{x}_Q$ and $\beloweq{x}_Q$ are defied similarly.

We write $x\comp y$ to mean that $x$ and $y$ are comparable, and will also use a dual version of this for incomparability. 
However, it is convenient to use a relation which also includes equality (i.e.\ incomparability under the strict relation $<$):
\sidecomment{There doesn't seem to be an established convention for this. Some authors just call this ``incomparable'', but I don't like the statement ``$x$ is incomparable with $x$.''}
we say that elements $x,y\in P$ are \emph{parallel} and write $x\incomp y$ if either $x=y$ or $x$ and $y$ are incomparable in $P$. Formally, one can define
\begin{equation*}
	x\incomp y \quad \iff \quad \neg (x < y \vee y < x).
\end{equation*}

We write $\compset{x}_P$ and $\incompset{x}_P$ for the subsets of the poset $P$ of all elements comparable and parallel to $x$ respectively, and note that $\compset{x}_P \inter \incompset{x}_P = \set{x}$.

One final piece of notation that we will use concerns the limit of a sequence of sets: if $X_1,X_2,\dotsc$ are subsets of $P$, then
\begin{align*}
    \liminf_{n\to\infty} X_n\defined \bigcup_{m=1}^\infty \bigcap_{n=m}^\infty X_n = \set{y\in P\st \exists N\; \forall n\geq N\; y\in X_n}.
\end{align*}
In words, the set of those points that are in every $X_n$ for sufficiently large $n$.

\medskip 
Another property of posets to which we will give special attention is that of being scattered.
A poset $P$ is said to be \emph{scattered} if it does not contain a poset order-isomorphic to $\QQ$ as a sub-order.

Scattered posets have been studied extensively.
A useful property of scattered posets is that we can find covers:
for elements $x,y\in P$, we say that $x$ \emph{covers} $y$, and write $x\covers y$, if $x>y$ and there is no $z\in P$ with $x>z>y$.

\begin{fact}
    \label{fact:covers}
    If $P$ is scattered, then for any $x,y\in P$ with $x<y$, there are $u,v\in P$ with $x\leq u < v\leq y$ such that $v\covers u$.
\end{fact}

Another fact that we will use extensively is the following.

\sidecomment{Equivalently, an infinite poset contains either an infinite chain or an infinite antichain.}
\begin{fact}
    \label{fact:infinite-chain}
    If $P$ is an infinite FAC poset, then $P$ has an infinite chain.
\end{fact}

This follows easily from basic infinite Ramsey theory.
In particular, if $x_1,x_2,\dotsc$ is an infinite sequence in an FAC poset $P$, then we can then find a sequence $n_1 < n_2 <\cdots$ such that either $x_{n_1}\geq x_{n_2}\geq\cdots$ or $x_{n_1}\leq x_{n_2}\leq\cdots$, noting that an infinite sequence of pairwise incomparable elements would by definition give an infinite antichain.

Throughout, we will write $\ww$ for the order type of the naturals and the first infinite ordinal, $\zz$ for the order type of the integers, and $\eta$ for the order type of the rationals.

Given two posets $P$ and $Q$ with orders $\leq_P$ and $\leq_Q$ respectively, the \emph{Cartesian product} poset $P\times Q$ is the poset on the Cartesian product set $P\times Q$, with order $\leq$ given by 
\begin{equation*}
	(p,q)\leq (p',q') \quad\text{ if and only if }\quad p\leq_P p' \text{ and } q\leq_Q q'.
\end{equation*}
\sidecomment{e.g.\ $\ww\times\ww$ is the square grid, $\ww\lextimes\ww$ is $\ww^2$; lexicographic multiplication is ordinal multiplication.}
For example, if $P$ is the two-element chain $\set{0,1}$, then $P\times P$ is the diamond poset on $\set{0,1}\times\set{0,1}$: the elements $(0,1)$ and $(1,0)$ are incomparable.

The \emph{lexicographic product} $P\lextimes Q$ is also a poset on the Cartesian product set $P\times Q$, except now the ordering $\leqlex$ is given by
\begin{equation*}
	(p,q)\leqlex (p',q') \quad\text{ if and only if }\quad p<_P p' \;\text{ or }\; \p[\big]{\,p=p' \text{ and } q\leq_Q q'\,}.
\end{equation*}
This can be thought of as replacing every element of $P$ with a copy of $Q$.
When $P = C$ is a chain, $C\lextimes Q$ can also be written as the linear sum $\bigoplus_{x\in C} Q$.

For an arbitrary poset $P$, a subset $X\sseq P$ is said to be \emph{convex}, or an \emph{interval} of $P$ if for all $x,y\in X$, if $z\in P$ has $x<z<y$, then $z\in X$.
We will use the notation $\cI(P)$ to represent the set of all non-empty intervals of $P$.

There are several particular types of intervals that we will consider here, each defined for elements $x,y\in P$ with $x<y$ and a poset $Q$ with $x,y \in P\inter Q$ (often $Q$ will be a chain in $P$ or an extension of $P$).
The first is the \emph{open interval}
\begin{align}
    \label{eq:def-open}
    (x,y)_Q\defined\set{z\in Q\st x<z<y},
\end{align}
and the second is the \emph{closed interval}
\begin{align}
    \label{eq:def-closed}
    [x,y]_Q\defined\set{z\in Q\st x\leq z\leq y}.
\end{align}
We will also take intervals where the limits are sets rather than elements:
\sidecomment{Here, and throughout, when we write $\dotsb < X$ for a set $X$, we mean that the relation $\dotsb < x$ holds for all $x\in X$; this is very different from saying that it holds for some $x\in X$. In particular, $\below{X} = \bigcap_{x\in X}\below{x}$ is \emph{not} the down-set generated by $X$; that is $\bigcup_{x\in X} \below{x}$.}
\begin{equation*}
	(X,Y)_Q \defined\set{z\in Q \st (\forall x\in X, \; x < z) \text{ and } (\forall y\in Y, \; y > z)}
\end{equation*}
We can also take the closed interval of a set without endpoints, resulting in the \emph{convex hull} of the set:
\begin{align}
    \label{eq:def-conv}
    \Conv_Q(X)\defined\set{x\in Q \st \exists y,z\in X,\; y\leq x\leq z}.
\end{align}
The subscripts may be omitted if it is clear from context what poset we are working in.
We will say that an interval $I$ of $P$ is \emph{proper} if there are $x,y\in P$ with $x<I<y$.
Finally, we say that a poset $Q$ is \emph{open} if it has neither minimal nor maximal elements.

\section{Poset extensions and chain completeness}
\label{sec:completions}

In this section, we introduce the tools and machinery that we will be using to prove our main results.
The key object introduced here, described in \Cref{subsec:completion}, is $H(P)$, which can be thought of as similar to a chain completion of $P$; while $H(P)$ is not in fact necessarily a chain complete poset, if $C\sseq P$ is a saturated chain then $H(C)$ is complete and naturally embeds into $H(P)$.
We will in particular be concerned with chains which are saturated with no maximal element, or similarly no minimal element.
We now provide an overview of this somewhat technical section before proceeding to the technicalities themselves.

We first lay some groundwork in \Cref{subsec:equivalence}, defining an equivalence relation between increasing chains.
Two chains will be considered equivalent if, in essence, elements of one chain are incomparable with only finitely many elements of the other, and vice-versa, a property which we call \emph{\mfi}.

With this equivalence relation in hand, in \Cref{subsec:completion} we will then consider the quotient of the set of increasing chains by this equivalence relation; we will use this set to augment $P$ and thus form our chain extension $H(P)$.
We also demonstrate how $H(P)$ has a natural order (the \emph{domination} order) which agrees with the order on $P$; in this way, we will be able to consider $P$ as a sub-poset of $H(P)$.

Finally, in \Cref{subsec:witnesses}, we consider when a chain consisting of non-principal points of $H(P)$ (that is, points in $H(P) \setminus P$) can be witnessed by a single chain $C\sseq P$; this is not always possible.
We finish the section with a short lemma tying these results in to vacillation, which will highlight the logic behind the name ``vacillating''.


\subsection{Chain equivalence}
\label{subsec:equivalence}

We want to add elements to our poset $P$ which will behave as the suprema of chains, as well as points for the infima; for this discussion, we will consider only suprema.
\sidecomment{I find that AI loves calling these arbitrary final segments ``germs'', which is completely reasonable, but it's not terminology I wanted to use.}
As we care about the supremum, we are generally interested only in properties of a chain which hold for arbitrary final segments: any ``unusual'' sections bounded away from the top of the chain can be skipped over.
We want to produce an equivalence relation for when two chains have final segments which are ``locally similar''.
Unlike other standard poset completions, we will in particular care about the order-type of these chains.
There are two examples to keep in mind here.
A diagram of the two examples can be found in \Cref{fig:section3-example}.

\begin{example}
	\label{ex:n-cross-2}
	Let $P$ be the poset on $\NN\times 2$ where $(n,i) > (m,j)$ if and only if $n > m$, i.e.\ take the poset $\ww$ and replace every element with an antichain of size 2.
\end{example}

The poset defined in \Cref{ex:n-cross-2} has uncountably many chains of the form $\set{(n,i_n) \st n\in \NN, i_n \in 2}$ (i.e.\ select one element from each antichain of size 2), but these chains all have roughly the same shape, and all differences between them are local.
As we will see, our definition will make these two chains equivalent, and give them the same supremum.
\sidecomment{All maximal chains in $\NN\times 2$ are equivalent for our purposes, but in the right-hand example, they are not. In fact, in the right-hand example, there are uncountably many pairwise non-equivalent maximal chains (quotient out by the equivalence relation of agreeing on an initial segment; classes are countable).}

\begin{figure}[ht]
	\centering
	\begin{tikzpicture}[
  font=\normalsize,
  line width=0.7pt,line cap=round,line join=round,
  continuation/.style={line width=1.3pt,line cap=round,
    dash pattern=on 0pt off 3pt},
  connection/.style={shorten <=8pt,shorten >=8pt},
  block connection/.style={shorten <=3pt,shorten >=3pt},
  pics/point/.style={code={
    \draw (-2pt,-2pt) -- (2pt,2pt)
          (-2pt,2pt) -- (2pt,-2pt);
  }},
  pics/chain region/.style={code={
    \draw (-0.30,0) -- (0,1.35) -- (0.30,0) -- cycle;
    \pic at (0,0.20) {point};
    \pic at (0,0.46) {point};
    \draw[continuation] (0,0.70) -- (0,1.05);
  }}
]
  \foreach \i in {0,...,3} {
    \coordinate (left-\i) at (0,1.2*\i);
    \coordinate (right-\i) at (1.2,1.2*\i);
    \pic at (left-\i) {point};
    \pic at (right-\i) {point};
  }
  \foreach \i [evaluate=\i as \j using int(\i+1)] in {0,1,2} {
    \draw[connection] (left-\i) -- (left-\j);
    \draw[connection] (right-\i) -- (right-\j);
    \draw[connection] (left-\i) -- (right-\j);
    \draw[connection] (right-\i) -- (left-\j);
  }
  \foreach \x in {0,1.2} {
    \draw[connection] (\x,3.6) -- (\x,4.8);
    \draw[continuation] (\x,4.95) -- (\x,5.65);
  }
  \draw[connection] (0,3.6) -- (1.2,4.8);
  \draw[connection] (1.2,3.6) -- (0,4.8);
  \node at (0.6,-0.75) {$\mathbb{N}\times 2$};

  \begin{scope}[xshift=5.6cm]
    \foreach \y in {0,1.9,3.8} {
      \pic at (0,\y) {chain region};
      \pic at (1.35,\y+0.65) {point};
    }
    \foreach \y in {0,1.9} {
      \draw[block connection] (0,\y+1.35) -- (0,\y+1.9);
      \draw[connection] (1.35,\y+0.65) -- (1.35,\y+2.55);
      \draw[block connection,shorten >=8pt] (0,\y+1.35) -- (1.35,\y+2.55);
      \draw[block connection,shorten <=8pt] (1.35,\y+0.65) -- (0,\y+1.9);
    }
    \foreach \x in {0,1.35}
      \draw[continuation] (\x,5.45) -- (\x,6.15);
    \node at (0.675,-0.75)
      {$\mathbb{N}\times\bigl(\mathbb{N}\cup\{*\}\bigr)$};
  \end{scope}
\end{tikzpicture}
	\caption{A diagram of the posets constructed in \Cref{ex:n-cross-2,ex:n-cross-n-star} respectively.}
	\label{fig:section3-example}
\end{figure}

\sidecomment{Just look at \Cref{fig:section3-example} rather than trying to read this.}
\begin{example}
	\label{ex:n-cross-n-star}
	Let $P$ be the poset on $\NN \times (\NN \union \set{*})$ with relations given by
	\begin{equation*}
		(a, b) < (c, d) \iff a < c \text{ or } \p[\big]{a = c \text{ and } b, d \in \NN \text{ and } b < d}.
	\end{equation*}
	Let $C_0 = \NN \times \NN$ and $C_1 = \NN \times\set{*}$, which are both naturally contained in $P$.
\end{example}

In the poset $P$ defined in \Cref{ex:n-cross-n-star}, the two chains $C_0$ and $C_1$ are both cofinal in the whole of $P$: every $x\in P$ is below some element of $C_0$ and some element of $C_1$.
Thus, in many poset extensions such as the Dedekind--MacNeille completion or the ideal completion, the chains $C_0$ and $C_1$ would both be given the same supremum, despite having different order-types. 
However, we will want to consider these chains as being different, and will thus give them different upper bounds.
Note that this immediately tells us that the extended poset that we construct will not be complete.

\sidecomment{The fact that the set $\incompset{x}_C$ is non-empty is important to prevent e.g. chains $C > D$ being considered to have m.f.i., when really they are totally comparable. Recall $\incomp$ includes $=$.}
To this end, we define a very strong sense of similarity -- \emph{\mfi} -- between two chains, and then define our equivalence between chains $C$ and $D$ to hold when there are final segments of these chains with \mfi.

\begin{definition}
    \label{def:mutually-finite-incomp}
    Chains $C$ and $D$ in a poset $P$ are said to have \emph{\mfi} if we have
    \begin{align*}
        \forall y\in D, \; 1 \leq \abs[\big]{\incompset{y}_C} < \aleph_0 \quad\text{ and }\quad \forall x\in C, \; 1 \leq \abs[\big]{\incompset{x}_D} < \aleph_0.
    \end{align*}
\end{definition}

\Cref{def:mutually-finite-incomp} intuitively says that $C$ and $D$ can only deviate from each other on small scales; the differences are local.
The following lemma begins to show how this condition forces $C$ and $D$ to strongly resemble each other.

\begin{lemma}
	\label{lem:mfi-max-iff-max}
	If saturated chains $C$ and $D$ in a poset $P$ have \mfi, then $C$ has a maximal element if and only if $D$ has a maximal element. 
	Similarly, $C$ has a minimal element if and only if $D$ has a minimal element.
\end{lemma}

\begin{proof}
	It suffices to consider maximal elements. 
	Assume for contradiction that $C$ has a maximal element and $D$ does not.
	Let $x = \max C$, and consider $I \defined \incompset{x}_D$, which we know by assumption is finite.
	\sidecomment{This starts to tell us that the m.f.i. chains look very similar.}
	There must therefore be an element $y\in D$ with $y > I$.
	The fact that $y\notin I$ means that $y\sim x$, and in particular we must have $y > x$, and so $y > C$.
	But $y > C$ in particular tells us that $\incompset{y}_C$ is empty, contradicting the definition of \mfi.
\end{proof}

As indicated above, we can now use the notion of \mfi\ to define our notion of equivalence of chains.

\sidecomment{Think: $\simeq_+$ means ``final segments with m.f.i.'' (of which sharing a maximal element can be intuitively considered a special case)}
\begin{definition}
    \label{def:equivalence}
    Non-empty saturated chains $C$ and $D$ in a poset $P$ are said to be \emph{equivalent as increasing chains}, and we write $C\simeq_+ D$, if either $C$ and $D$ have maximal elements which are equal, or neither $C$ nor $D$ has a maximal element, and there are non-empty final segments $C'\sseq C$ and $D'\sseq D$ with \mfi.
    Similarly, $C$ and $D$ are \emph{equivalent as decreasing chains}, written $C\simeq_- D$, if either they have minimal elements which are equal, or neither has a minimal element, and there are non-empty initial segments with \mfi.
\end{definition}

\sidecomment{In fact, any two maximal chains in the poset in \Cref{ex:n-cross-2} are equivalent as increasing chains.}
We may note that, as discussed earlier, the two named maximal chains $\set{(n,0) \st n\in \NN}$ and $\set{(n,1)\st n\in\NN}$ in the poset in \Cref{ex:n-cross-2} are equivalent as increasing chains, whereas $C_0$ and $C_1$ from \Cref{ex:n-cross-n-star} are not equivalent as increasing chains, despite both inducing the same downwards-closed sets.

Given the name `equivalent', it should come as no surprise that $\simeq_+$ and $\simeq_-$ are equivalence relations on the set of non-empty saturated chains of $P$.
To prove this fact, we first need to establish some properties of chains with \mfi.
We will show that if $C$ and $D$ have \mfi, then the `large scale' structures of $C$ and $D$ are identical.
To this end, we define the following equivalence relation, which allows us to consider the structure of a chain.

\begin{definition}
    \label{def:finite-distance}
    Given a chain $C$ in a poset $P$, define the symmetric relation $\finsim_C$ on $C$ by setting $x\finsim_C y$ and $y \finsim_C x$ if and only if $x \leq y$ and the interval $(x,y)\inter C$ is finite.
\end{definition}

It is immediate from definition that, for any chain $C$, the relation $\finsim_C$ is in fact an equivalence relation on elements of $C$, and that each equivalence class of $C$ is either finite, or isomorphic to one of $\ww$, $\ww^*$ and $\zz$.
The following lemma shows that \mfi\ `respects $\finsim_C$-equivalence classes'.

For $x\in C$, we denote the $\finsim_C$-equivalence class of $x$ by $[x]_C$.
Note that $C\,/\finsim_C$ naturally inherits the structure of a linear order from $C$.

\sidecomment{This lemma should hopefully not come as a surprise: if the chains have differences only on a finite scale, then when we zoom out to the infinite scale, they should look the same.}
\begin{lemma}
    \label{lem:mfi-canonical-isomorphism}
    If $C$ and $D$ have \mfi, then there is an order-isomorphism $f\from \quo{C} \;\to \quo{D}$ such that, for any $x\in C$ and $y\in D$ with $x\incomp y$, the function $f$ sends $[x]_C$ to $[y]_D$.
\end{lemma}

\begin{commentfigure}[htbp]
    \centering
    \usetikzlibrary{decorations.pathreplacing}
\begin{tikzpicture}[
  font=\normalsize, line width=.7pt, line cap=round, line join=round,
  every node/.style={inner sep=2pt},
  bound/.style={dash pattern=on 4pt off 3pt,shorten <=5pt},
  interval/.style={decorate,decoration={brace,mirror,amplitude=4pt}},
  pics/point/.style={code={
    \draw (-2pt,-2pt)--(2pt,2pt) (-2pt,2pt)--(2pt,-2pt);
  }}
]
  \coordinate (x)  at (0,4.65);
  \coordinate (y)  at (0,1.55);
  \coordinate (gx) at (3.00,4.05);
  \coordinate (gy) at (3.00,1.10);
  \coordinate (upper-x) at (3.00,5.35);
  \coordinate (upper-y) at (3.00,2.30);

  \draw (0,0)--(0,6.20) node[above=2pt] {$C$};
  \draw (3,0)--(3,6.20) node[above=2pt] {$D$};
  \draw[->] (.55,6.12)--node[above=2pt] {$f$} (2.45,6.12);

  \foreach \a/\b/\hi in {x/gx/upper-x,y/gy/upper-y} {
    \draw[bound] (\a)--(\hi);
    \draw[bound,shorten >=5pt] (\a)--(\b);
    \pic at (\a) {point};
    \pic at (\b) {point};
    \node[left=4pt] at (\a) {$\a$};
    \node[below right=2pt] at (\b) {$g(\a)$};
  }

  \draw[interval] (-.68,4.65)--(-.68,1.55)
    node[midway,left=6pt] {$I$};
  \draw[interval,preaction={draw=white,line width=2.4pt}]
    (2.62,4.05)--(2.62,1.10) node[midway,left=6pt] {$J$};

  \foreach \a/\lo/\hi in {x/4.05/5.35,y/1.10/2.30}
    \draw[interval] (3.26,\lo)--(3.26,\hi)
      node[midway,right=6pt] {$\incompset{\a}_D$};
\end{tikzpicture}
    \caption{The objects constructed in the first half of the proof of \Cref{lem:mfi-canonical-isomorphism}.}
	\label{fig:mfi-canonical-isomorphism}
\end{commentfigure}

\begin{proof}
	Note first that, if $x\in C$ and $y,z\in D$ with $x\incomp y$ and $x\incomp z$, then we have $y\finsim_D z$, as otherwise $\incompset{x}_D$ would be infinite.
    Define the function $g$ as follows.
    \begin{equation*}
        g(x)\defined\min\incompset{x}_D.
    \end{equation*}
    Note that $\incompset{x}_D$ is a finite subset (in fact, a finite interval) of the chain $D$, so does indeed have a minimum.
    We may now define $f([x]_C)\defined [g(x)]_D$; we claim that this is well-defined.
    We must show that, if $x,y\in C$ satisfy $x\finsim_C y$, then $g(x)\finsim_D g(y)$.
    
    Indeed, define the intervals
	\begin{equation*}
		I\defined [x,y]_C \quad \text{ and } \quad J\defined (g(x),g(y))_D,
	\end{equation*}
	and assume for contradiction that $I$ is finite but $J$ is infinite.
	If $z\in J\setminus I$ were comparable to all of the interval $I$, then we know from the definition of $J$ that we also have $x<z<y$, and so $\incompset{z}_C = \emptyset$, contradicting \mfi\ of $C$ and $D$.

    Therefore, all elements of $J\setminus I$ are incomparable with a finite and non-zero number of elements of the finite set $[x,y]_C$.
	Thus some element of $I$ is incomparable with infinitely many elements of $J$, contradicting the \mfi\ of $C$ and $D$.
    We deduce that $f$ is a well-defined function.

	\sidecomment{We now deviate from the notation displayed in \Cref{fig:mfi-canonical-isomorphism}.}
    We now show that $f$ is injective and surjective.
    Injectivity follows from an argument similar to the one showing that $f$ is well-defined.
    For surjectivity, take some $y\in D$; we show that $f$ hits $[y]_D$.
    Indeed, there is some $x\in C$ such that $x\incomp y$, and so $g(x)\finsim_D y$, and so $f([x]_C) = [y]_D$.
	Thus $f$ is a bijection.

	To show that $f$ is order-preserving, let $x,y\in C$ have $[x]_C < [y]_C$; we prove that $f([x]_C) < f([y]_C)$.
	Note that it in fact suffices to prove that $g(x) < g(y)$, so assume for contradiction that $g(x) > g(y)$, noting that these values must be distinct, else $f([x]_C) = f([y]_C)$.
	
	As $g(x)$ is the minimal element of $D$ which is parallel to $x$ with $g(y) < g(x)$, we must have that $x\comp g(y)$.
	If $x < g(y)$ then we find that $x < g(y) < g(x)$, a contradiction, and if $x > g(y)$ then we find that $y > x > g(y)$, also a contradiction.
	Thus $f$ is indeed order-preserving.
\end{proof}

We will want to refer to the function $f$ defined in \Cref{lem:mfi-canonical-isomorphism} on several occasions, and so we give it a name.

\sidecomment{Note that, as it relates segments via the incomparability relation, $f$ is indeed canonical.}
\begin{definition}
	\label{def:canonical-isomorphism}
	If saturated chains $C,D\sseq P$ have \mfi, then we call the function $\canon{f}{C}{D}$ as in \Cref{lem:mfi-canonical-isomorphism} the \emph{canonical isomorphism} between these two quotient posets.
\end{definition}

One of the key properties of the canonical isomorphism is that \mfi\ of chains passes via the canonical isomorphism to the elements of the quotient posets.

\sidecomment{Compare the properties derived here to those in \Cref{def:esmc}}
\begin{lemma}
	\label{lem:mfi-quotient}
	Let saturated chains $C,D\sseq P$ have \mfi, and let $\canon{f}{C}{D}$ be the canonical isomorphism.
	Then, for all $X\in\quo{C}$, the chains $X$ and $f(X)$ have \mfi.
	Moreover, if $X_1,X_2\in\quo{C}$ have $X_1 < X_2$, then $X_1 < f(X_2)$ and $f(X_1) < X_2$.
\end{lemma}

\begin{proof}
	For $X\in \quo{C}$, we know that every element $x\in X$ has $\incompset{x}_{f(X)}$ finite, and so it suffices to prove that these sets are non-empty.
	By symmetry, it thus suffices to prove that $\incompset{x}_{f(X)}$ is non-empty.
	
	Assume for contradiction that $\incompset{x}_{f(X)} = \emptyset$.
	We know by definition that $\incompset{x}_D$ is non-empty, so take $y\in \incompset{x}_D$, which we know must be in some class $Y\in \quo{D}$ with $Y \neq f(X)$.
	But, by the defining property of $f$ in \Cref{lem:mfi-canonical-isomorphism}, $f$ should send $[x]_C$ to $[y]_D$, a contradiction.

	Finally, for the ``moreover'' part of the lemma, take $x \in X_1$ and $y\in f(X_2)$ and $z \in \incompset{x}_D$, noting that $z\in f(X_1)$.
	As $y\not\finsim_D z$, there are infinitely many elements of $D$ between $y$ and $z$.
	In particular, there must be an element $w\in D$ with $w < y$ and $w > \incompset{x}_D$, as this set is finite.
	Thus $y > w > x$, and so we see $X_1 < f(X_2)$. 
	The fact that $f(X_1) < X_2$ follows similarly.
\end{proof}

Our next goal is to prove that $\simeq_-$ and $\simeq_+$ are equivalence relations. 
We will use the following lemma to prove this.

\begin{lemma}
	\label{lem:mfi-with-arbitrary}
	Let $C,D\sseq P$ be saturated chains with \mfi, and let $\canon{f}{C}{D}$ be the canonical isomorphism.
	Then, for all $x\in P$, if $\incompset{x}_C$ is finite, non-empty and contained in some $X\in\quo{C}$, then $\incompset{x}_D$ is finite and contained in $f(X)$.
\end{lemma}
\sidecomment{This roughly says that finite incomparability transfers over m.f.i., and will be a key part in proving that $\simeq_\pm$ is transitive.}

\begin{proof}
	Let $\canon{f}{C}{D}$ be the canonical isomorphism, and note that, as $\incompset{x}_C$ is a finite interval of $C$, it must be contained in a single class of $\quo{C}$.
	Let $X\in \quo{C}$ be this class.

	We first claim that $\incompset{x}_D \sseq f(X)$.
	Assume for contradiction that there is $Y \in \quo{C}$ such that $f(Y) > f(X)$ (\wlg) and there is $y\in \incompset{x}_{f(Y)}$.

	First, if there is $Z\in\quo{C}$ with $X< Z< Y$, then \Cref{lem:mfi-quotient} tells us that $f(Y) > Z$ and so $y > Z > x$, a contradiction.
	Thus $Y \covers X$ is a cover.

	Due to the definition of $\finsim$, either $X$ has no maximal element or $Y$ has no minimal element, and \Cref{lem:mfi-quotient} together with \Cref{lem:mfi-max-iff-max} tells us that $f(X)$ and $f(Y)$ share these properties with $X$ and $Y$.

	If $X$ and $f(X)$ both have no maximal elements, then there is $z \in\above{x}_X$, but we know that $f(Y) > X$, so $y > z > x$, a contradiction.

	If $Y$ and $f(Y)$ both have no minimal elements, then as $Y$ and $f(Y)$ have \mfi\ due to \Cref{lem:mfi-quotient}, we know that there is $z\in Y$ with $z < y$, and so again $y > z > x$, a contradiction.

	Thus, in any case, $\incompset{x}_D \sseq f(X)$.

	\medskip

	Assume for contradiction that the set $\incompset{x}_{f(X)}$ is infinite.
	It thus has either no minimum or no maximum, as $f(X)$ is a sub-order of $\zeta$, and assume \wlg\ that it has no maximum.
	Thus the sets $X$ and $f(X)$ both likewise have no maxima and so, letting $w\in X$ be the minimal element above $\incompset{x}_X$, there is $z\in f(X)$ with $w < z$.
	But then $x < w < z$ is below a final segment of $f(X)$, a contradiction.

	We may therefore conclude that $\incompset{x}_D$ is finite, as required.
\end{proof}

We are now ready to show that $\simeq_+$ and $\simeq_-$ are equivalence relations.

\sidecomment{This is the key result which allows us to actually define $H(P)$, as we can quotient out by these relations to get the non-principal elements.}
\begin{lemma}
    \label{lem:equivalence}
    The relations $\simeq_+$ and $\simeq_-$ are equivalence relations on the set of non-empty saturated chains of $P$.
\end{lemma}

\begin{proof}
    Note that it suffices to consider $\simeq_+$, as the result for $\simeq_-$ follows by symmetry.
	We must prove that $\simeq_+$ is reflexive, symmetric, and transitive.
	Note that reflexivity and symmetry follow immediately from \Cref{def:equivalence}, and so it remains to prove transitivity.

	Let $C,D,E\sseq P$ be saturated chains with
	\begin{equation*}
		C\simeq_+ D \simeq_+ E.
	\end{equation*}
	We must show that $C\simeq_+ E$.
	Note that if any of the three chains has a maximal element, then all three have a maximal element, and moreover these maximal elements are equal; $C\simeq_+ E$ follows immediately.
	It thus suffices to consider the case when none of these chains contain a maximal element.

	We know from \Cref{def:equivalence} that there are final segments $C'\sseq C$, $D',D'' \sseq D$, and $E''\sseq E$ such that
	\sidecomment{The idea here is simply to put these two functions together to show that $C\simeq_+ E$, but there are some annoyances with potential minimal elements which make this harder than it feels like it should be.}
	\begin{equation*}
		\canon{f}{D'}{C'} \quad\text{ and }\quad \canon{g}{D''}{E''}
	\end{equation*}
	are canonical isomorphisms.
	We trim these chains down so that they are compatible sizes.
	Take $d \in D'\inter D''$, and let $c\in C'$ and $e \in E''$ have $c\incomp d$ and $e\incomp d$. 
	Then define
	\begin{equation*}
		C^* \defined \aboveeq{c}_{C'}, \quad D^* \defined \aboveeq{d}_{D'\inter D''}, \quad E^* \defined \aboveeq{e}_{E''}.
	\end{equation*}
	Define the function 
	\begin{equation*}
		\canon{h}{C^*}{E^*} \quad\text{ by }\quad h \defined g\circ f^{-1}.
	\end{equation*}
	It follows immediately that $h$ is an order-isomorphism, and so it suffices to prove that, for all $X\in\quo{C^*}$, the sets $X$ and $h(X)$ have \mfi.

	Fix $X\in \quo{C^*}$ and take $x\in X$.
	As $D^*$ and $E^*$ have \mfi, it follows from \Cref{lem:mfi-with-arbitrary} and mutually finite incomparability of $C^*$ and $D^*$ that $\incompset{x}_{E^*}$ is finite and contained in $h(X) \in \quo{E^*}$.
	It remains to show that $\incompset{x}_{E^*}$ is non-empty.
	In fact, this will not necessarily be the case, and we may need to slightly adjust the chains.

	\sidecomment{And here is the promised annoyance about minimal elements.}
	Assume for contradiction that $\incompset{x}_{E^*} = \emptyset$ (and note the reverse direction of $\incompset{y}_{C^*} = \emptyset$ for $y\in E^*$ holds by symmetry).
	As $E^*$ is saturated, this means that either $x > E^*$ or $x < E^*$.

	Consider first the case that $x > E^*$.
	As we know that $C^*, D^*, E^*$ all have no maximal elements, it follows from \mfi\ that there are $y\in D^*$ and $z\in E^*$ with $x < y < z$.
	But we know by assumption that $x > z$, proving that $x > x$, a contradiction.
	Thus we have $x < E^*$.

	Due to a similar argument as in the case $x > E^*$, we know that $E^*$, and thus also $D^*$ and $C^*$, has a minimal element. 
	Let $y = \min E^*$, and define $C^\dagger \sseq C^*$ as
	\begin{equation*}
		C^\dagger \defined C^* \setminus \below{y}_P.
	\end{equation*}
	Note that $C^\dagger$ is also a final segment of $C$, has a minimal element, and $\min C^\dagger \incomp y$.
	This in particular shows that, for all $z\in C^\dagger$, we have
	\begin{equation*}
		1 \leq \abs[\big]{\incompset{z}_{E^*}} < \aleph_0,
	\end{equation*}
	and likewise for elements of $E^*$.
	This proves that $C\simeq_+ E$, as required.
\end{proof}

Now that we know that these relations are in fact equivalence relations, we are ready to define our chain completion $H(P)$, which we do in the following section.


\subsection{Chain extension and the domination order}
\label{subsec:completion}

The intuition behind the poset $H(P)$ is that it is simply the set of non-empty saturated chains of $P$ quotiented by the relation of equivalence as defined in \Cref{subsec:equivalence}.
However, this intuition runs into a problem in that we have not one, but two equivalence relations: $\simeq_+$ and $\simeq_-$.

To get around this issue, we define two quotient spaces, one quotienting out by $\simeq_+$ and the other by $\simeq_-$.

Once we have defined the two quotient sets, we will have to carefully glue the two resulting structures together along the `principal' elements, and define an order on the whole structure taking into account both the increasing and decreasing chains.

\sidecomment{This is basically all just technicality; we've already covered the main idea which is that we want to quotient out by these relations. We then glue the two parts together over principal elements.}
\begin{definition}
    \label{def:increasing-and-decreasing}
    Given a poset $P$, let $F(P)$ be the set of all non-empty saturated chains of $P$.
    We then define the sets of \emph{increasing} and \emph{decreasing} chains as
    \begin{align*}
        G_+(P)\defined F(P)/\simeq_+ \quad \text{ and } \quad G_-(P)\defined F(P)/\simeq_-
    \end{align*}
    respectively. Elements of $G_+(P)$ and $G_-(P)$ are thus equivalence classes of non-empty saturated chains of $P$.
\end{definition}

Note that if an equivalence class $X\in G_+(P)$ contains a chain $C$ with a maximal element $x$, then it is immediate from \Cref{def:equivalence} that all chains in $X$ have $x$ as a maximal element.
In particular, the chain $\set{x}$ is an element of the equivalence class $X$.
For this reason, we refer to such elements of $G_+(P)$ (and the similar elements of $G_-(P)$) as \emph{principal elements}, and the other elements as \emph{nonprincipal elements}.

\sidecomment{$\simeq_{\text{prin}}$ is the promised gluing over principal elements. Don't worry about it, it never shows up again beyond this sub-section.}
There is now a natural equivalence between principal elements of $G_+(P)$ and principal elements of $G_-(P)$; for $X\in G_+(P)$ and $Y\in G_-(P)$, write $X\simeq_\text{prin} Y$ (and also $Y \simeq_\text{prin} X$) if there is some $x\in P$ such that $\set{x}\in X$ and $\set{x}\in Y$.
After setting $X\simeq_\text{prin} X$ for all $X$, it is clear that $\simeq_\text{prin}$ is an equivalence relation.
Note that, if $X,Y\in G_+(P)$ both contain chains with equal maximal elements, then $X = Y$, and so there is a unique principal element corresponding to $x$ up to all of these equivalences.

\begin{definition}
    \label{def:h-p}
	Define the set \emph{chain extension} $H(P)$ of $P$ by
    \begin{align*}
        H(P)\defined \p[\big]{(G_+(P)\union G_-(P))/\simeq_{\text{prin}}}.
    \end{align*}
    Furthermore, define $N(P)$ to be the set of \emph{nonprincipal elements} of $H(P)$: the set of those equivalence classes of chains in $G_+(P)$ with no maximal element and those in $G_-(P)$ with no minimal element.
    We will write $N_+(P)$ and $N_-(P)$ for the nonprincipal increasing and decreasing chains in $N(P)$ respectively, so $N(P)=N_+(P)\union N_-(P)$.
\end{definition}

\sidecomment{We really want to say that $P\sseq H(P)$, and this is why we can.}
While $H(P)$ is technically a set of equivalence classes of equivalence classes of chains, we will abuse notation and identify the principal elements of $H(P)$ with the corresponding elements of $P$ (they are in natural bijection). 
Following this abuse of notation, we may write the following.
\begin{align}
    \label{eq:h-partition}
    H(P) = P \union N_+(P) \union N_-(P).
\end{align}

\sidecomment{In particular, properties that we care about will be preserved by $\simeq_+$, and so it doesn't matter what representative chain that we choose.}
We will furthermore generally refer to nonprincipal elements of $H(P)$ only by a representative chain.
In particular, we will write, for example, ``$C\in N_+(P)$'' to mean that $C$ is a representative chain of an equivalence class $X\in N_+(P)$ such that $\set{X}\in H(P)$.
If we wish to emphasise that we are considering an equivalence class rather than a chain, then we will use square brackets: $C$ will be a chain and $[C]_+$ and $[C]_-$ will be equivalence classes in $N_+(P)$ and $N_-(P)$ respectively.

Now that we have defined $H(P)$, our next task is to give it the structure of a poset.
We will define an ordering $\geq$ for all elements of $H(P)$, and then show that it agrees with the ordering on $P$ on the principal elements.

\sidecomment{As we only care about initial/final segments, think of domination like this: we can pass to suitable initial/final segments of representative chains, and then the sets are \emph{totally} comparable.}
\begin{definition}
    \label{def:domination}
    For $X,Y\in H(P)$, we say that $X$ \emph{dominates} $Y$, and write $X > Y$, if for all choices of representative chains $C,D$ of $X,Y$ respectively, there are non-empty final/initial segments $C'\sseq C$ and $D'\sseq D$ (final or initial depending on whether $X$ and $Y$ are increasing or decreasing elements respectively) satisfying $C' > D'$, i.e.
	\begin{equation*}
		\forall x\in C', \;\; \forall y\in D', \;\; x > y
	\end{equation*}

    We can also define $X\geq Y$ on $H(P)$ to mean $X > Y$ or $X = Y$.
\end{definition}

We need to show that domination does not depend on choice of representative and is indeed a partial order, but we spell it out.

\sidecomment{This proof is really very easy, but it's also very important. For most of the preparation of the paper, it was just asserted without proof.}
\begin{lemma}
	\label{lem:domination-independent}
	If $X,Y\in H(P)$ and there exist representative chains $C,D\sseq P$ with $C < D$, then $X < Y$, i.e.\ domination is independent of choice of representative chain.
\end{lemma}

\begin{proof}
	Let $E,F$ be arbitrary representative chains of $X,Y$.
	We focus on $X$ first.
	Assume for simplicity that $X \in N_+(P)$.
	We claim that we can take a terminal segment $E'\sseq E$ such that
	\begin{equation*}
		E' \sseq \Conv_P(C).
	\end{equation*}
	If this holds, then we can choose $F'\sseq F$ in a similar way and conclude that $E' < F'$, as required.

	If $C$ has a maximal element, then $E$ has the same maximum, and so we can take $E' = \set{\max C}$.

	If $C$ has no maximal element, then neither does $E$, and so, letting $C'\sseq C$ and $E'\sseq E$ have \mfi, take $x\in E'$. 
	$\incompset{x}_{C'}$ is finite, and so there is $y\in C'$ with $y > x$.
	
	Take $z\in C'$ arbitrary. 
	As $\incompset{z}_{E'}$ is finite, there is $w\in E'$ with $w > z$.
	Let $E'' \defined \above{w}_{E'}$.
	Then every element of $E''$ is above $w$ and below some element of $C'$, and so $E'' \sseq \Conv_P{C'}$, as required.
\end{proof}

\sidecomment{This lemma is again easy to prove, but so essential that it should be spelled out.}
\begin{lemma}
	\label{lem:domination-partial}
	The domination order is a partial order
\end{lemma}

\begin{proof}
	First, $X\not< X$.
	As domination is independent of choice of representative by \Cref{lem:domination-independent}, it suffices to prove that, given a saturated chain $C$, one cannot choose two final segments $C',C''\sseq C$ with $C' > C''$. 
	But $C'\inter C''$ is non-empty, so this is immediate.

	Second, we cannot have both $X < Y$ and $Y < X$.
	This again follows from independence of representative chains: if $C\in X$ and $D \in Y$ ar representatives and final/initial segments $C',C'' \sseq C$ and $D',D''\sseq D$ have $C' < D'$ and $C'' > D''$, then we can consider $C'\inter C''$ and $D'\inter D''$, both of which should be larger than the other, a contradiction.

	Finally, we need to prove transitivity; assume $X < Y$ and $Y < Z$. 
	Given saturated chains $C,D,E$ representing $X,Y,Z$ with final/initial segments $C',D',D'',E''$ with $C' < D'$ and $D'' < E''$, we can let $D^* = D' \inter D''$ and note that $C' < D^* < E''$, proving $X < Z$.

	Thus domination is indeed a partial order.
\end{proof}

Secondly, if $C$ and $D$ are both principal, and equivalent to $\set{x}$ and $\set{y}$ respectively, then we can take $\set{x}$ and $\set{y}$ as their respective final/initial segments, and see that the ordering in \Cref{def:domination} does agree with the order from $P$.

Domination will be the standard order that we use on $H(P)$: if we refer to an ordering of elements of $H(P)$ without further comment, the domination order is implied.

Now that $H(P)$ has been given the structure of a poset, we can prove the following two important lemmas.

\sidecomment{This is essentially just a technicality; $H(P)$ can behave strangely, but it's usually in the form of having many more incomparable elements than you expected, not this setting, which would require an abundance of comparable elements.}
\begin{lemma}
    \label{lem:h-scattered}
    If $P$ is a scattered poset, then so too is $H(P)$.
\end{lemma}

\begin{proof}
    Assume for contradiction that $H(P)$ contains a copy of the rationals.
    It is a standard result that if the rationals are finitely coloured, then one colour class must itself contain a copy of the rationals.
    This can be seen to follow from \Cref{fact:covers} (if there was no monochrome copy of $\QQ$, one finds $x_1\covers y_1$ in the poset restricted to the first colour class, and then $x_2\covers y_2$ with $x_1 > x_2 > y_2 > y_1$ in a second colour class, and so on.)

    Note that, due to equality \eqref{eq:h-partition}, we may assume \wlg\ that $N_+(P)$ contains a copy of the rationals: let $f\from \QQ\to N_+(P)$ be an order-preserving injection. 
    Now take $X,Y\in N_+(P)$ with $X < Y$.
    It is then immediate from the definition of the domination order that there is $x\in P$ with $X < x < Y$.
    Thus for any $p,q\in \QQ$ with $p<q$, there is some $x_{p,q}\in P$ such that $f(p) < x_{p,q} < f(q)$.
    A simple inductive construction then produces a copy of the dyadic rationals, and thus a copy of $\QQ$, inside $P$.
    This contradicts the fact that $P$ is scattered, therefore proving that $H(P)$ is scattered.
\end{proof}

\sidecomment{Another point, which for a long time I just asserted until I checked how long the proof was and realised that it wasn't quite as trivial as I'd assumed.}
\begin{lemma}
	\label{lem:h-chain-is-chain}
	If $C$ is a chain then the poset $H(C)$ is also a chain.
	Moreover, if $X,Y\in N(C)$ have $X < Y$ and $X\in N_-(C)$ or $Y \in N_+(C)$ (or both), then there is $x\in C$ with $X < x < Y$.
\end{lemma}

\begin{proof}
	There are three cases to consider, based on which of $N_-(C)$ and $N_+(C)$ the points $X$ and $Y$ are in.
	Assume in both cases that we have elements $X,Y\in N(C)$ which are incomparable.

	\medskip\noindent
	\textbf{Case 1: $\mathbf{X \in N_-(C)}$ and $\mathbf{Y\in N_+(C)}$} (or vice versa).

	For $X,Y$ to be incomparable, from the definition of the domination order, we must have that every final segment of the down-set $\below{Y}_C$ intersects every initial segment of the up-set $\above{X}_C$.
	Consider the interval $I \defined \below{Y}_C \inter \above{X}_C$ of $C$, which we know by assumption is non-empty.
	Moreover, we know that, as $X$ and $Y$ are non-principal, $I$ must be open.
	We may thus select $x\in I$ arbitrarily, and take the disjoint sets
	\begin{equation*}
		\above{x}_C \inter \below{Y}_C \sseq I \quad\text{and}\quad \below{x}_C \inter \above{X}_C \sseq I,
	\end{equation*}
	contradicting our assumption.
	This means that $X$ and $Y$ are comparable, a contradiction.

	\medskip\noindent
	\textbf{Case 2: $\mathbf{X, Y \in N_+(C)}$}, or, similarly, $X, Y \in N_-(C)$ (we consider the $N_+(C)$ case).

	For $X,Y$ to be incomparable in this case, they must both give the same down-set in $C$:
	\begin{equation*}
		I \defined \below{X}_C = \below{Y}_C.
	\end{equation*}
	But this implies that $I \in [X]_+$ and $I\in [Y]_+$, telling us that $X\simeq_+ Y$, contradicting our assumption that $X,Y$ were distinct elements of $N(C)$.

	\medskip\noindent
	\textbf{Finding an intermediate point.}

	If $X<Y$ have $X \in N_-(C)$ and $Y\in N_+(C)$, then the argument of Case 1 immediately gives us a point $x\in C$ with $X < x < Y$, so we need only consider the case $X,Y\in N_+(C)$ (the $N_-$ case holding by symmetry).

	If $X<Y$ and $X,Y\in N_+(C)$, then we know that $\below{X}_C \ssneq \below{Y}_C$. 
	We may thus take a point $x\in \below{Y}_C \setminus \below{X}_C$, as required, completing the proof of the lemma.
\end{proof}



\sidecomment{It's good to keep in mind just how much $H(P)$ can fail to be complete (it can really fail to be complete).}
We now prove that, if $C$ is a chain, then $H(C)$ is complete. 
It is worth noting that, for a general poset $P$, the chain extension $H(P)$ need not be chain-complete: consider the poset $P$ defined in \Cref{ex:n-cross-n-star}. 
$C_0 \not\simeq_+ C_1$, but $[C_0]_+$ and $[C_1]_+$ are both minimal upper bounds for both of $C_0$ and $C_1$.

It is easy to see, however, that for any non-empty saturated chain $C\sseq P$ with no maximum, $[C]_+$ is an upper bound for $C$, and there are no strictly smaller upper bounds for $C$.
It is for this reason that we refer to $H(P)$ as a chain extension, rather than a chain completion.

\sidecomment{Note that $D$ need not be saturated.}
\begin{lemma}
	\label{lem:h-complete}
    For any poset $P$ and non-empty chain $C\sseq P$, the poset $H(C)$ is a complete linear order.
    Moreover, if a non-empty subset $D\sseq H(C)$ has no maximum in $D$, then $\sup_{H(C)}(D)\in N_+(C)$.
\end{lemma}

\begin{proof}
	We know from \Cref{lem:h-chain-is-chain} that $H(C)$ is a chain, so it suffices to prove completeness, which follows from proving that $\sup_{H(C)}(D)$ exists; we therefore construct the supremum of an arbitrary non-empty set $D$ with no maximum and show that this supremum is in $N_+(C)$.

	We first claim that, for all $y\in D$, there is a principal element $x\in C$ in $\Conv_C(D)$, i.e.\ there exists $z\in D$ such that $y < x < z$.
	Assume for contradiction that this fails, and take $y < y_1 < y_2 < \dotsb$ in $D$; we must have $y_i \in N(C)$ for all $i$.
	For some $\eps\in\set{-,+}$, there are $i<j$ such that $y_i,y_j \in N_\eps(C)$.
	But then \Cref{lem:h-chain-is-chain} gives us a point $z\in C$ with $y_i < z < y_j$, contradicting our assumption and proving the claim.

	We may now define the set $E$ to be
	\begin{equation*}
		E \defined \set{x\in C \st \exists y \in D, \; x < y}.
	\end{equation*}
	The argument above already shows that $E$ is non-empty and has no maximum. 
	Moreover, $E$ is an initial segment of $C$ and so is in particular saturated.
	This shows that $[E]_+ \in N_+(C)$.

	We know that, for all $y\in D$, there is $x\in E$ with $x > y$, and so $y < [E]_+$ and we see that $[E]_+$ is indeed an upper bound for $D$.
	It remains to prove that there is no smaller upper bound for $D$; assume for contradiction that there is and let $X < [E]_+$ be an upper bound for $D$.
	The definition of the domination order, \Cref{def:domination}, gives us a point $x\in E$ with $X < x$.
	But we know from the definition of $E$ that there is thus a point $y\in D$ with $x < y$, contradicting the assumption that $X$ is an upper bound for $D$.
	Thus $[E]_+ = \sup_{H(C)}(D)$, as required.
\end{proof}

\sidecomment{In fancy language, $H$ is not idempotent.}
\begin{remark}
	Unlike a completion of a poset, it is often the case that
	\begin{equation*}
		H(H(P)) \neq H(P).
	\end{equation*}
	Indeed, this can already be seen for chains: if $C = \ww$ is an infinite increasing chain, then $H(C)$ contains a nonprincipal upper bound for this chain; call this element $x$.
	But, upon considering $H(H(C))$, the element $x$ is principal, and so a new nonprincipal upper bound for $\ww$ will be added.
\end{remark}

Note also that $H(P)$ interacts well with passing to sub-posets.

\sidecomment{It would be natural to call such a sub-poset $Q$ \emph{saturated}, but it doesn't seem to come up again.}
\begin{observation}
	\label{obs:embedding}
	If $P$ is a poset and $Q\sseq P$ has the property that if $C\sseq Q$ is a saturated chain, then $C$ is also saturated in $P$, then $H(Q)$ naturally embeds into $H(P)$ by sending $[X]_+ \in N_+(Q)$ to the equivalence class in $N_+(P)$ containing $X$, and similarly for $N_-(P)$.
\end{observation}

It is easy to check that the map given in \Cref{obs:embedding} is well defined.
Following this observation, we will abuse notation and consider $H(Q)$ as a subset of $H(P)$ when $Q\sseq P$ satisfies the condition of \Cref{obs:embedding}.
The usual case in which this situation will be encountered is when $Q$ is itself a saturated chain in $P$.


\subsection{Witness chains and vacillation}
\label{subsec:witnesses}

We will be in general interested in all the limit points inside a particular saturated chain, not just those at the top and bottom.
We thus introduce the following terminology.

\sidecomment{The issue on witness chains and representative chains being different is real and a bit annoying. It's mostly a technicality.}
\begin{definition}
	We say that a saturated chain $C\sseq P$ \emph{witnesses} an element $X\in N(P)$ if $X \in N(C)$.
\end{definition}

Note that witness chains and representative chains are not the same: a representative chain $C$ of $X$ has $C\in X$; if $X\in N_+(P)$, then $X = \sup_{H(C)}(C)$ is the supremum of $C$.
However, if $C$ is a witness chain of $X$, then we merely need $X\in N(C)$, and do not care about where in $N(C)$ the point $X$ occurs.

We conclude \Cref{sec:completions} by considering how we can find chains witnessing many elements of $N(P)$ simultaneously, and how this interacts with vacillation. 
We begin with the following two simple lemmas, extending part of \Cref{lem:h-chain-is-chain} from chains to posets.

\sidecomment{This begins a sequence of results where we want to say ``we can witness lots of things at the same time.'' However, we cannot witness whatever we want.}
\begin{lemma}
	\label{lem:witness-two-points}
	Given two comparable non-principal points $X,Y\in N(P)$, there is a saturated chain $C\sseq P$ witnessing both $X$ and $Y$, i.e.\ with $X,Y \in N(C)$.
\end{lemma}

\begin{proof}
	Assume without loss of generality that $X < Y$.
	The definition of the domination order tells us that we can find saturated chains $D$ and $E$ witnessing $X$ and $Y$ respectively such that $D < E$.
	We thus know that $D\union E$ is a chain, and so we may extend it to a saturated chain, which witnesses both $X$ and $Y$, as required.
\end{proof}

We can now use \Cref{lem:witness-two-points} to find intermediate points between non-principal elements.

\sidecomment{Note that we cannot necessarily find intermediate elements if $X \in N_+(P)$ and $Y\in N_-(P)$; think e.g.\ of $\ww + \ww^*$.}
\begin{lemma}
	\label{lem:find-intermediate}
	Assume that non-principal points $X,Y\in N(P)$ with $X < Y$ have $X\in N_-(P)$ or $Y\in N_+(P)$ (or both). 
	Then there exists $x\in P$ with $X < x < Y$.
\end{lemma}

\begin{proof}
	First apply \Cref{lem:witness-two-points} to find a chain $C$ witnessing both $X$ and $Y$. 
	We may then apply \Cref{lem:h-chain-is-chain} to find a point $x\in C$ with $X < x < Y$.
	These relations also hold in $H(P)$, and so the lemma is proved.
\end{proof}

We would like to extend these simple lemmas to witness an infinite number of nonprincipal points simultaneously.
However, as the following example shows, we cannot hope to do this in full generality.

\sidecomment{This example is very important when working with these objects, and thinking that I was able to witness whatever chain of nonprincipal elements I wanted led to a serious flaw in an earlier version of this paper (and an incorrect proof of the (now known to be false \cite{Lup26}) statement that every countable FAC poset has a strongly maximal chain).}
\begin{example}
	\label{ex:cannot-witness}
	Let $P$ be the poset on $\NN^2 \union \NN$ satisfying the following relations for all $a, b, c, d \in \NN$.
	\begin{itemize}
		\itemsep=0pt
		\item $(a,b) < (c,d)$ whenever $a < c$ or both $a = c$ and $b < d$,
		\item the elements of $\NN$ inherit their usual order,
		\item $(a, b) < c$ whenever $a < c$, and
		\item $c \not< (a,b)$.
	\end{itemize}
	Let $\top_\NN \in N_+(P)$ be witnessed by the copy of $\NN$, and let $X_n \in N_+(P)$ be witnessed by $\set{(n, i) \st i\in \NN}$. Then there is no chain in $P$ witnessing $\top_\NN$ and all of $X_1,X_2,\dotsc$ simultaneously.
\end{example}

\begin{proof}
	To see that the claim of \Cref{ex:cannot-witness} holds, assume for contradiction that there was such a chain $C\sseq P$.
	The fact that $C$ witnesses $\top_\NN$ implies that $n\in C$ for some $n\in\NN$. 
	But this implies that $(a,b)\notin C$ for all $a \geq n$; in particular, $C$ cannot witness $X_{n+1}$, a contradiction.
\end{proof}

However, we can prove a result in this direction. 
The key property is that the set of nonprincipal elements that we witness must be ordered like a sub-order of the integers.

\sidecomment{This is the correct statement for what we can witness. In some situations we'll be able to witness more (e.g.\ if we carefully build up the chain one non-principal element at a time and take limits before thinking about what to add next), but in general this is the statement to have in mind.}
\begin{lemma}
	\label{lem:simultaneous-witness-infinite}
	Let $I$ be a sub-order of $\zeta$ and let $X_i \in N(P)$ for all $i\in I$, with the elements $X_i$ distinct and pairwise comparable, with $i < j$ in $I$ implying that $X_i < X_j$ in $N(P)$.
	Then there exists a saturated chain $C\sseq P$ witnessing all $X_i$ simultaneously.
\end{lemma}

\begin{proof}
	We construct pairwise disjoint and comparable chains $C_i\sseq P$ for all $i\in I$ such that $C_i$ witnesses $X_i$.
	Consider $I$ as a saturated sub-chain of $\ZZ$ with the natural order, so that we can write $i-1$ and $i+1$ for elements $i\in I$.

	First, we construct some elements $x_i$ of $P$ to split up the $X_i$.
	For each $i$ such that both $i, i+1 \in I$, if either $X_i \in N_-(P)$ or $X_{i+1} \in N_+(P)$, then we may apply \Cref{lem:find-intermediate} to find a point $x_i$ with $X_i < x_i < X_{i+1}$.
	If there exists $i = \min I$ and $X_i \in N_+(P)$, then take $x_{i-1} \in P$ with $x_{i-1} < X_i$ by choosing an arbitrary element of a representative chain of $X_i$.
	Similarly, if there exists $i = \max I$ and $X_i \in N_-(P)$, then take $x_i \in P$ with $x_i > X_i$.

	\sidecomment{Chopping up the intervals like this is the key; note that we can't do this in \Cref{ex:cannot-witness}.}
	Now, for each $X_i$ in turn (in an arbitrary order), let $Q_i$ be the following interval.
	\begin{equation*}
		Q_i \defined \begin{cases}
			(x_{i-1}, X_i)_P & \text{if } X_i \in N_+(P), \\
			(X_i, x_i)_P & \text{if } X_i \in N_-(P).
		\end{cases}
	\end{equation*}
	Note that the required elements $x_i$ have been selected (even though we have not necessarily chosen $x_i$ for every $i\in I$).

	Consider the case $X_i\in N_+(P)$; the other case is similar.
	We know that $x_{i-1} < X_i$, so by definition of the domination order, there is a saturated chain $D_i \sseq Q_i$ witnessing $X_i$.
	In particular, as the posets $Q_i$ are pairwise disjoint and comparable, so too are the chains $D_i$.
	Thus we may let $C$ be an arbitrary extension of the chain $\bigcup_{i\in I} D_i$ to a saturated chain, which witnesses all $X_i$, as required.
\end{proof}

Finally, we deduce an extension of \Cref{lem:simultaneous-witness-infinite}, which is one demonstration of how the property of vacillation can be useful.

\sidecomment{Notice how the sequence in \eqref{eq:vacillating-sequence} alternates (or \emph{vacillates}!) between elements of $N_+(P)$ and $N_-(P)$. This is the condition which leads to the name ``vacillating''.}
\begin{lemma}
	\label{lem:witness-vacillation}
	If $P$ is a vacillating poset and $X_1 > X_2 > \dotsb$ is an infinite decreasing sequence in $N_+(P)$, then there are indices $i_1 < i_2 < \dotsb$ and elements $Y_1 > Y_2 > \dotsb$ in $N_-(P)$ such that
	\begin{equation}
		\label{eq:vacillating-sequence}
		X_{i_1} > Y_1 > X_{i_2} > Y_2 > \dotsb.
	\end{equation}
	Moreover, there is a saturated chain $C\sseq P$ simultaneously witnessing the elements $X_{i_j}$ and $Y_j$ for all $j\geq 1$.
	A similar result holds for an infinite increasing sequence in $N_-(P)$.
\end{lemma}

\begin{proof}
	Note first that, once we have established that condition \eqref{eq:vacillating-sequence} holds, we can deduce the existence of the required witness chain immediately from \Cref{lem:simultaneous-witness-infinite}.
	It thus suffices to construct the sequence in \eqref{eq:vacillating-sequence}.

	\sidecomment{This proof pretty much writes itself after unwrapping the definition of vacillating.}
	Assume that we have $X_1 > X_2 > \dotsb$ in $N_+(P)$, and let $D$ be a saturated chain witnessing these elements, as provided by \Cref{lem:simultaneous-witness-infinite}, and let $D'$ be defined as
	\begin{equation*}
		D' \defined D \inter \Conv_{H(P)}\p[\big]{\set{X_n \st n\geq 1} }.
	\end{equation*}
	As $P$ is vacillating and the intervals $(X_{i+1},X_i)_D$ are all infinite, they cannot all be well-founded.
	Thus there must be some internal point $Y_1 \in N_-(D')$.
	In particular, as $N(D')$ is a linear order due to \Cref{lem:h-complete}, there is some $i_1$ such that $X_{i_1} > Y_1 > X_{i_1+1}$.

	We may now repeat this process for the infinite decreasing sequence $X_{i_1+1} > X_{i_1+2} > \dotsb$ in $N_+(P)$ to find $X_{i_2}$ and $Y_2$.
	Continuing this for all finite steps gives us the sequence as in \eqref{eq:vacillating-sequence}, as required.
\end{proof}

\section{Maximal tubes and the Aharoni--Korman conjecture}
\label{sec:tubes}

In this section, we show that the Aharoni--Korman conjecture, \Cref{conj:ak}, reduces to finding a certain structure in a poset, which we call a \emph{maximal tube}.

\sidecomment{Think of a tube as a ``widening'' of a chain; instead of no incomparability, we now enforce merely finite incomparability.}
\begin{definition}
    A \emph{tube} in a poset $P$ is a set $T\sseq P$ such that every element of $T$ is comparable to all but finitely many other elements of $T$. Formally,
	\begin{equation*}
		\forall x\in T, \; \incompset{x}_P \inter T \;\text{ is finite.}
	\end{equation*}
\end{definition}

In particular, we can see that all chains are tubes.
The name ``tube'' comes from how we will find chains within tubes, and so think of tubes as the containers holding the chains.
We are in particular interested in tubes which cannot be extended, as follows, due to the properties which we will investigate throughout the rest of this section.

\sidecomment{I am somewhat convinced that maximal tubes are more natural objects than spines in infinite posets.}
\begin{definition}
    A tube $T$ in poset $P$ is called a \emph{maximal tube} if there is no element $x\in P\setminus T$ for which $T\union\set{x}$ is a tube.
    That is, every element $x\in P\setminus T$ is incomparable to infinitely many elements of $T$.
\end{definition}

The aim of this section is to prove the following result, which reduces \Cref{conj:ak} to the problem of finding maximal tubes in posets.

\begin{proposition}
	\label{prop:maximal-tube-suffices}
    Let $P$ be a countable FAC poset, and assume that there exists a maximal tube $T\sseq P$.
    Then there is a chain $C\sseq T$ which is a spine of $P$.
\end{proposition}

\begin{figure}[ht]
	\begin{center}
		\begin{tikzpicture}[
  scale=0.9,
  every node/.style={circle, fill, inner sep=1.5pt},
  label/.style={rectangle, fill=none, inner sep=1pt, font=\small}
]


\def\dy{1.0}

\node (bot) at (0, -0.5*\dy)   {};
\node (top) at (0, 6.5*\dy) {};

\node (a1) at (-3, 1*\dy) {};
\node (a2) at (-3, 2*\dy) {};
\node (a3) at (-3, 3*\dy) {};
\node (a4) at (-3, 4*\dy) {};
\node[fill=none] (a5) at (-3, 5*\dy) {};

\node (b1) at (0, 1*\dy) {};
\node (b2) at (0, 2*\dy) {};
\node (b3) at (0, 3*\dy) {};
\node (b4) at (0, 4*\dy) {};
\node[fill=none] (b5) at (0, 5*\dy) {};

\node (c1) at (3, 1*\dy) {};
\node (c2) at (3, 2*\dy) {};
\node (c3) at (3, 3*\dy) {};
\node (c4) at (3, 4*\dy) {};
\node[fill=none] (c5) at (3, 5*\dy) {};

\draw (bot) -- (a1);
\draw (bot) -- (b1);
\draw (bot) -- (c1);

\draw (a5) -- (top);
\draw (b5) -- (top);
\draw (c5) -- (top);

\draw (a1) -- (a2) -- (a3) -- (a4);
\draw[dotted] (a4) -- (a5);

\draw (b1) -- (b2) -- (b3) -- (b4);
\draw[dotted] (b4) -- (b5);

\draw (c1) -- (c2) -- (c3) -- (c4);
\draw[dotted] (c4) -- (c5);

\draw[thin] (a2) -- (b1);
\draw[thin] (a3) -- (b2);
\draw[thin] (a4) -- (b3);
\draw[thin,dotted] (a5) -- (b4);

\draw[thin] (b2) -- (c1);
\draw[thin] (b3) -- (c2);
\draw[thin] (b4) -- (c3);
\draw[thin,dotted] (b5) -- (c4);
\draw[thin] (c2) -- (b1);
\draw[thin] (c3) -- (b2);
\draw[thin] (c4) -- (b3);
\draw[thin,dotted] (c5) -- (b4);

\node[label] at (0, 6.5*\dy+0.4) [fill=none] {$\top$};
\node[label] at (0,-0.5*\dy-0.4) [fill=none] {$\bot$};

\node[label] at (-3.6, 0.3) [fill=none] {$A$};
\node[label] at (-0.6, 0.3) [fill=none] {$B$};
\node[label] at ( 2.4, 0.3) [fill=none] {$C$};

\node[label] at ( 4.6, 0.3) [fill=none] {$T$};

\draw[rounded corners=10pt, thick, draw=lightgray]
  (-3.6, 0.6*\dy) rectangle (-2.4, 5.4*\dy);
\draw[rounded corners=10pt, thick, draw=lightgray]
  (-0.6, 0.6*\dy) rectangle ( 0.6, 5.4*\dy);
\draw[rounded corners=10pt, thick, draw=lightgray]
  ( 2.4, 0.6*\dy) rectangle ( 3.6, 5.4*\dy);
\draw[rounded corners=10pt, thick, draw=lightgray]
  (-1.2,-1.4*\dy) rectangle ( 4.2, 7.4*\dy);

\end{tikzpicture}		
	\end{center}
	\caption{An example of a poset $P$ with distinguished elements $\bot$ and $\top$ and infinite chains (in the rounded rectangles, each of order type $\ww$) $A,B,C$ labelled. Here, $T\defined\set{\bot,\top}\union B\union C$ (also shown in a grey rectangle) is the unique maximal tube, but $\set{\bot, \top} \union X$ is a spine of $P$ for all $X\in\set{A,B,C}$.}
	\label{fig:maximal-tube}
\end{figure}

\Cref{fig:maximal-tube} shows an example of a maximal tube in a poset, along with an example of a spine which is not contained in any maximal tube, showing that a strong converse of \Cref{prop:maximal-tube-suffices} (``every spine is contained in a maximal tube'') does not hold.
One result we will use when working with tubes is \Cref{thm:zaguia-spine}, which is due to Zaguia \cite{zaguia2024progress}; we first restate this theorem in our terminology.

\begin{theorem}[\cite{zaguia2024progress}, Theorem 6, rephrased]
    \label{thm:zaguia-spine-reformed}
    If $P$ is itself a tube, then $P$ has a spine.
\end{theorem}

While \Cref{thm:zaguia-spine-reformed} is already known, we nevertheless give a proof using our machinery; this also serves as a demonstration of some of the useful properties of tubes.

\sidecomment{Note that this exactly captures the intuition that tubes resemble chains: when we zoom out to the level of nonprincipal limit points, they \emph{are} chains.}
\begin{lemma}
	\label{lem:tube-nonprincipal-comparable}
	If $P$ is a tube, then every $X\in N(P)$ is comparable to all of $H(P)$. In particular, $N(P)$ is a chain.
\end{lemma}

\begin{proof}
	First take $X\in N_+(P)$ and $x\in P$ (the case of $X\in N_-(P)$ is similar).
	Let $C$ be an arbitrary representative chain in $X$.
	As $\incompset{x}_P \inter C$ is finite, $x$ is comparable to a final segment of $C$, but this is exactly the property required for $x\comp X$.
	Thus $X\comp P$.

	It remains to consider pairs of limit points.
	First consider $X,Y \in N_+(P)$ (again, both in $N_-(P)$ is similar). 
	Take any witness chains $C\in X$ and $D\in Y$.
	Similarly to the previous argument, every element of $C$ is comparable to a final segment of $D$ and vice versa, so either a final segment of $C$ is above all of $D$, the converse, or there are final segments of $C$ and $D$ with mutually finite incomparability: take $x\in C$ and $y\in D$ with $x\incomp y$, and take the final segments $\aboveeq{x}_C$ and $\aboveeq{y}_D$.
	Take $z > x$ in $C$. 
	Then the fact that we are in a tube tells us that $\incompset{z}_D$ is finite, and saturation of witnesses tells us that $\incompset{z}_D$ is non-empty.
	In this case, $C\simeq_+ D$ and so $X = Y$.

	Finally, assume that $X\in N_+(P)$ and $Y\in N_-(P)$ are distinct. 
	As before, take witness chains $C\in X$ and $D\in Y$.
	Every $x\in C$ must be below some initial segment $D'$ of $D$, as otherwise $X > Y$ immediately, and similarly every $y\in D$ must be above some final segment $C'$ of $C$.
	But this implies that $C < D$, and so $X < Y$.
	This completes the proof of the lemma.
\end{proof}

We are now ready to prove \Cref{thm:zaguia-spine-reformed}.
The proof follows two main steps: we use \Cref{lem:tube-nonprincipal-comparable} to reduce to the case of posets with no internal limit points, and then we use compactness to deal with this case.

\begin{proof}[Proof of \Cref{thm:zaguia-spine-reformed}]
	\sidecomment{Note that, more generally, we can perform this operation at any internal cut-point of $H(P)$, so we may assume that $H(P)$ has no cut-points.}
	Note that \Cref{lem:tube-nonprincipal-comparable} tells us that every element of $N(P)$ is a cut-point of $H(P)$.
	Let $I \sseq N(P)$ be an initial segment and define 
	\begin{equation*}
		P_I \defined \p[\big]{I, N(P)\setminus I}_P.
	\end{equation*}
	Note that the $P_I$ partition $P$ into convex pieces and, if $I\neq J$ then $P_I \sim P_J$.
	If we can find a spine $C_I$ in each $P_I$, then their union will be a spine of $P$, as we can also take the union of the antichain partitions of each $P_I$.

	\sidecomment{Note that some $P_I$ may be empty, but this does not cause a problem.}
	We may therefore assume that $P$ has no internal limits and so all chains in $P$ have order-type appearing as a sub-order of $\zeta$.
	Fix an arbitrary maximal antichain $A_0\sseq P$, and note that, due to maximality of $A_0$, every point $x\in P$ is either in $A_0$, above some element of $A_0$, or below some element of $A_0$.
	We may thus define the \emph{height} function $\hite(x)$ as follows.
	Let $D_x$ be the longest finite chain with one end in $A_0$ and the other end $x$, noting that there can be no infinite such chain as $P$ has no internal limits, and define 
	\sidecomment{This is exactly what you would expect the (signed) height of $x$ relative to $A_0$ to be.}
	\begin{align*}
		\hite(x)\defined\begin{cases}
			\abs{D_x} - 1 & \text{ if } x > y \text{ for some } y \in A_0, \\
			1 - \abs{D_x} & \text{ if } x < y \text{ for some } y \in A_0, \\
			0 & \text{ if } x \in A_0.
		\end{cases}
	\end{align*}
	We will apply compactness, and to this end we define the poset $P_n\sseq P$ as
	\begin{equation*}
		P_n \defined \set{x\in P \st -n \leq \hite(x) \leq n}.
	\end{equation*}
	\sidecomment{The plan is just to apply compactness to $P_n$, but we need to be a little bit careful as we're constructing both the chain and the antichains, and we want to fix both.}
	Note that, as $P_n$ has neither an infinite chain nor an infinite antichain, $P_n$ is finite.
	$P_n$ therefore has a chain $C_n\sseq P_n$ and a partition into antichains $P_n = \bigcup_{i\in I_n} A_i^{(n)}$ such that $C_n$ meets every $A_i^{(n)}$.
	Note further that we cannot have $x,y\in C_n$ with $\hite(x) = \hite(y)$, as one of these two points would then admit a longer chain to $A_0$; we may thus label points in $C_n$ and antichains in the partition by their height.
	Moreover, due to finite incomparability within $P$, each point $x\in P$ is contained in only finitely many antichains.

	We may now apply compactness, diagonalising over the point and antichain corresponding to heights $0,1,-1,2,-2,\dotsc$ in order; the above discussion tells us that there are only finitely many choices of a point and an antichain for each height (including the possibility of choosing to take no point and antichain at some heights).
	Let the resulting chain and antichains be $C$ and $\set{B_i\st i\in I}$.
	We claim that $C$ is a spine of $P$ with respect to the antichains $B_i$.
	
	It is immediate from construction that $C$ is a chain, all $B_i$ are antichains, and that $C$ meets every $B_i$, so it remains to show that $\bigcup_{i\in I} B_i = P$.
	Assume for contradiction that there is $x\in P$ with
	\begin{equation*}
		x \in P \setminus \bigcup_{i\in I} B_i.
	\end{equation*}
	As $x$ is contained in only finitely many antichains of $P$, there is some integer $N$ such that every antichain of $P$ containing $x$ is contained in $P_N$.
	But this implies that $x$ would have been missing from the partition at some finite stage of the diagonalisation, when every point and antichain in $P_N$ was fixed.
	This contradiction shows that $C$ is a spine of $P$, completing the proof of \Cref{thm:zaguia-spine-reformed}.
\end{proof}

We now show that, given any spine $C$ of $T$ and corresponding partition of $T$ into antichains, we can extend it into a partition of $P$ into antichains around the same spine $C$.
This is where the condition of being a maximal tube comes in, as every element outside of $T$ is incomparable to infinitely many elements of $T$, and we leverage this fact to extend the antichains inside $T$ to all of $P$.

\sidecomment{This extends infinite incomparability from elements to antichains. As antichains are finite, they are in some way ``local'', and so this lemma shouldn't be too much of a surprise.}
\begin{lemma}
    \label{lem:extend-spine-partition}
    Let $T$ be a maximal tube in a countable FAC poset $P$, let $C$ be a spine of $T$, and let $T = \bigcup_{x\in C} A_x$ be any partition into antichains with $x\in A_x$ for all $x\in C$.
    Then for all $y\in P\setminus T$, there are infinitely many $x\in C$ such that $y$ is incomparable to all of $A_x$.
\end{lemma}

\begin{commentfigure}
	\centering
	\begin{tikzpicture}[
  x=0.0174cm,y=-0.0174cm,font=\normalsize,
  line width=0.7pt,line cap=round,line join=round,
  fibre/.style={red},
  continuation/.style={line width=1.3pt,line cap=round,
    dash pattern=on 0pt off 3pt},
  pics/outer region/.style={code={
    \draw[rounded corners=28pt] (0,0) rectangle (350,375);
  }},
  section point/.style={draw=blue,fill=white,circle,
    inner sep=0pt,minimum size=4pt,line width=0.7pt},
  pics/point/.style={code={
    \draw (-2.5pt,-2.5pt) -- (2.5pt,2.5pt)
          (-2.5pt,2.5pt) -- (2.5pt,-2.5pt);
  }}
]
  \pic at (0,25) {outer region};
  \node at (341,10) {$P$};
  \coordinate (left-x) at (90,180);
  \pic at (left-x) {point};
  \node[above left=3pt] at (left-x) {$x$};

  \foreach \offset in {125,205}
    \draw plot[domain=90:330,samples=81,variable=\t]
      ({\offset+25*sin(1.5*(\t-90))},\t);
  \node at (255,180) {$T$};

  \draw ({125+25*sin(1.5*(132-90))},132)
    .. controls (165,110) and (188,110) .. (188,140)
    .. controls (188,170) and (188,190) .. ({205+25*sin(1.5*(190-90))},190);
  \draw ({125+25*sin(1.5*(307-90))},307)
    .. controls (140,307) and (145,257) .. ({205+25*sin(1.5*(257-90))},257);

  \foreach \yy in {140,160,180}
    \draw[fibre] (163,\yy) -- (217,\yy);
  \draw[fibre] (151,200) -- (205,200);
  \foreach \yy in {242,262,282,302}
    \draw[fibre] (116,\yy) -- (176,\yy);

  \foreach \p in {(175,140),(175,160),(180,180),(195,200),
                  (170,242),(125,262),(130,282),(120,302)}
    \node[section point] at \p {};
  \node at (166,220) {$\incompset{x}_T$};
  \node[red] at (163,328) {$A_y$};
  \node[blue,anchor=east] at (97,292) {$z_y$};

  \draw[->,line width=0.9pt] (370,223) -- (425,223);

  \pic at (445,25) {outer region};
  \node at (786,10) {$P$};
  \draw[rounded corners=25pt] (475,205) rectangle (595,380);
  \draw[rounded corners=25pt] (615,50) rectangle (735,290);

  \foreach \yy in {250,286,322}
    \draw (565,\yy) -- (630,\yy-102);

  \coordinate (right-x) at (541,163);
  \pic at (right-x) {point};
  \node[above left=3pt] at (right-x) {$x$};
  \draw (541,182) -- (541,205);

  \foreach \i/\yy in {1/334,2/301,3/268} {
    \coordinate (y\i) at (541,\yy);
    \pic at (y\i) {point};
    \node[left=12pt] at (y\i) {$y_{\i}$};
  }
  \draw[continuation] (541,229) -- (541,250);
  \node at (635,351) {$\below{x}$};

  \foreach \i/\yy in {1/235,2/199,3/163} {
    \coordinate (z\i) at (658,\yy);
    \pic at (z\i) {point};
    \node[right=8pt] at (z\i) {$z_{\i}$};
  }
  \draw[continuation] (658,113) -- (658,144);
  \node at (765,128) {$\incompset{x}$};
\end{tikzpicture}
	\caption{Diagram of the proof of \Cref{lem:extend-spine-partition}. The right-hand part of the diagram is ``zoomed in'' on a section of the left-hand part.}
\end{commentfigure}

\begin{proof}
	Assume for contradiction that there is an $x\in P\setminus T$ which is incomparable to only finitely many sets $A_y$ for $y\in C$.
	As $x\notin T$ and $T$ is a maximal tube, $x$ is incomparable to infinitely many elements of $T$. As $P$ is an FAC poset, the set $A_y$ is finite, and so there are infinitely many $y\in C$ such that $x$ is incomparable to some element $z_y \in A_y$.
	We can therefore find an infinite set $D\sseq C$ such that, for all $i \in D$, there are elements $y_i,z_i \in A_i$ such that $x\comp y_i$ and $x\incomp z_i$.
	Moreover, \Cref{fact:infinite-chain} tells us that there must be an infinite chain amongst the elements $\set{z_i \st i \in D}$, and so by passing to a subset and suitably labelling the elements of $D$ with $\NN$, we can assume without loss of generality that $z_1 < z_2 < z_3 < \cdots$ (the case of $z_1 > z_2 > z_3 > \cdots$ is similar).

	\sidecomment{We essentially just keep hitting things with standard techniques. Though there was a significant amount of time where I didn't realise that maximality alone sufficed for this, and was also dragging this property kicking and screaming from a consolidated chain.}
	As $T$ is a tube, for all $i$, the element $y_i$ is incomparable with only finitely many of the elements $z_j$, and so for all sufficiently large $j$ (in terms of $i$), we have $z_j \comp y_i$. 
	We may in particular assume that $j > i$, and so $z_j > z_i$.
	As $y_i$ and $z_i$ are incomparable (both being members of the antichain $A_i$), we cannot have that $y_i > z_j$, and so we must have $y_i < z_j$ for all sufficiently large $j$ (in terms of $i$).
	Moreover, as $x$ is incomparable to $z_j$, we must have that $x > y_i$ for all $i$.

	Finally, we know that $z_1$ is incomparable to only finitely many of $y_1,y_2,\dotsc$ (as $T$ is a tube), and so there is some $j$ such that $z_1 \comp y_j$.
	But $z_1 < z_j$, and so we cannot have $y_j < z_1$, as then we would have $y_j < z_j$, contradicting the fact that $y_j$ and $z_j$ are in the same antichain, and so are incomparable.
	We thus have $z_1 < y_j$, but $y_j < x$, and so $z_1 < x$, contradicting that $x$ is incomparable to $z_1$.

	Thus for all $x\in P\setminus T$, there are infinitely many $y\in C$ such that $x$ is incomparable to all of $A_y$, as required.
\end{proof}

With \Cref{lem:extend-spine-partition} in hand, we can prove \Cref{prop:maximal-tube-suffices}.

\sidecomment{Note that we actually only add at most one point to each antichain in the first partition; infinite incomparability makes life easy.}
\begin{proof}[Proof of \Cref{prop:maximal-tube-suffices}]
	We know from \Cref{thm:zaguia-spine-reformed} that $T$ has a spine $C$, and so there is a partition $T = \bigcup_{x\in C} A_x$ of $T$ into antichains, with $x\in A_x$ for all $x\in C$.
	We may then apply \Cref{lem:extend-spine-partition} to learn that for all $y\in P\setminus T$, there are infinitely many $x\in C$ such that $y$ is incomparable to all of $A_x$.
	Indeed, define
	\begin{equation*}
		Y_y \defined \set{x\in C \st y \incomp A_x}.
	\end{equation*}

	\sidecomment{Note that we really are using countability here!}
	We claim that there is a partition $P = \bigcup_{x\in C} B_x$ of $P$ into antichains with $A_x\sseq B_x$ for all $x\in C$, and so $C$ is a spine of $P$.
	Arbitrarily enumerate $P\setminus T$ as $y_1,y_2,y_3,\dotsc$.
	For each $i$ in turn, we may then select an element
	\begin{equation*}
		x_i \in Y_{y_i} \setminus\set{x_1,x_2,\dotsc,x_{i-1}},
	\end{equation*}
	where we note that this is always possible as $Y_{y_i}$ is infinite.
	We may then define $B_x$ to be $A_x$ together with all $y_i$ such that $x=x_i$, noting that there is at most one such $y_i$. 
	Then $\set{B_y \st y\in C}$ is a partition of $P$ into antichains, and so $C$ is a spine of $P$, as required.
\end{proof}

We have therefore reduced \Cref{thm:main} to the problem of finding maximal tubes in posets.
The remainder of the paper consists of a series of reductions towards this goal.

\section{Reducing to scattered posets}
\label{sec:structural}

In this section, we present a proof of \Cref{thm:structural}, a structural result which, as well as being a general result on FAC posets, allows us to reduce the problem of finding a spine in a general FAC poset to that of finding a spine in a scattered FAC poset.

Recall that $\eta$ is the order type of $\QQ$ and that $\cI(X)$ is the set of all non-empty intervals of $X$.
We begin with the key definitions of this section.

\sidecomment{Like a lot of the definitions in this paper, this one went through a lot of different revisions. Its current form is quite delicate. The idea is just to keep ripping out single points and putting in copies of $\eta$ for as long as you can.}
\begin{definition}
    \label{def:esmc}
    A chain $C\sseq P$ is said to be \emph{$\eta$-maximal} if $C$ has order type $\eta$, and the only chain $D\sseq P$ with order type $\eta$ and $C\sseq D$ is $C$ itself.
    An \emph{$\eta$-replacement} from one $\eta$-maximal chain $C$ to another, $C'$, consists of a function $f\from C\to \cI(C')$ satisfying all of the following.
    \begin{itemize}
        \item For all $x\in C$, the interval $f(x)\sseq C'$ is either the singleton $\set{x}$, or has order type $\eta$.
        \item For all $x,y\in C$ and $z\in f(y)$, if $x<y$ then $x < z$, and if $y <x$ then $z<x$.
    \end{itemize}
    We will say in this case that $f$ \emph{witnesses} the replacement, and we write $C\etalt C'$.
    Moreover, say that the replacement is \emph{trivial} if $f(x) = \set{x}$ for all $x\in C$.
	Finally, say that $C$ is an \emph{\esmc} if it admits no non-trivial $\eta$-replacement.
\end{definition}

We remark here that we do not necessarily have $\bigcup_{x\in C} f(x) = C'$, as can be seen by the following example.
\sidecomment{In short: we might need to throw in some more points to maintain maximality after a replacement.}
Identify $C$ with $\QQ$ and assume that $f$ is the identity everywhere except for at $0$, where $f(0) = X$ has order type $\eta$. 
Assume that the only other element of $P$ except for $C$ and $X$ is a single element $a$ with $a > X$, $a > 0$ and $a < q$ for all $q\in C$ with $q > 0$.
Then $C$ is $\eta$-maximal, as every element of $X$ is incomparable to an element of $C$, and $C\union \set{a}$ has $a \covers 0$, so is not isomorphic to the rationals.
Moreover, $(C\setm{0})\union X$ is not $\eta$-maximal, as we can add $a$ to this chain and it does still have order-type $\eta$.
The only option is thus to have $a\notin\bigcup_{x\in C}f(x)$.

We now move on to prove some results about these replacements and chains.
The goal of this section is to prove that \esmc s exist.

\begin{proposition}
    \label{prop:esmc-exists}
    A countable FAC poset $P$ is either scattered or has an \esmc.
\end{proposition}

\sidecomment{Note that this really is a non-trivial application of Zorn. Usually the upper bound constructed is just the union of the steps, or something else obvious. Here it's not even clear that a sensible definition of the limit is non-empty.}
We need to prove that there is a chain which is maximal under $\etalt$ amongst $\eta$-maximal chains; for this, we will use Zorn's lemma.
It thus suffices to prove that $\etalt$ is a partial order on the set of $\eta$-maximal chains, and any $\etalt$-chain of $\eta$-maximal chains of $P$ has an upper bound.
We will show first that $\etalt$ is a partial order, and then move on to show that Zorn's lemma applies.
However, before either of these, we prove the following two simple lemmas, which give some useful properties of $\eta$-replacements.

\begin{lemma}
	\label{lem:etalt-comparability}
	If $C,D\sseq P$ are $\eta$-maximal chains with $C \etalt D$ witnessed by $\etawitness{f}{C}{D}$, then if $x,y\in C$ have $x < y$, then $f(x) < f(y)$.
	In particular, $f(x) \inter f(y) = \emptyset$.
\end{lemma}

\sidecomment{For a while this was part of the definition of $\eta$-replacements, but it turns out it's implied by the other parts.}
\begin{proof}
	Take $x,y\in C$ with $x < y$, as in the statement.
	Then, as $C$ is $\eta$-maximal, there is a point $z\in C$ with $x < z < y$.
	Take $u \in f(x)$ and $v \in f(y)$ and note that \Cref{def:esmc} tells us that $z < y$ implies $z < v$ and $z > x$ implies $z > u$.
	Therefore $u < v$ and so $f(x) < f(y)$; it is then immediate that $f(x) \inter f(y) = \emptyset$.
\end{proof}

\begin{lemma}
	\label{lem:etalt-incomparability}
	Assume that $C,D\sseq P$ are $\eta$-maximal chains with $C\etalt D$ witnessed by $\etawitness{f}{C}{D}$.
	If $x\in C$ and $y\in f(x)$, then $y\incomp x$.
\end{lemma}
\sidecomment{This statement should hopefully be clear by staring at a picture of what's happening.}

\begin{proof}
	Assume for contradiction that $x\in C$ and $y\in D$ are as in the statement of the lemma with $x\leq y$ (the case $x \geq y$ is entirely similar).
	If $f(x) = \set{x}$, then $y=x$ and we are done, so assume that $f(x)$ has order type $\eta$.
	In particular, the set $Y \defined \above{y}_{f(x)}$  has order type $\eta$.
	But we know from \Cref{def:esmc} that $Y \comp C\setm{x}$ and so, as $Y > y \geq x$, we have $Y \comp C$.
	But then $C\union Y$ has order type $\eta$, contradicting the $\eta$-maximality of $C$ and proving the lemma.
\end{proof}

We next show that $\etalt$ is transitive, and also deduce the form of a witness function.
Note that we abuse notation slightly in the following: $g[f(x)]$ is a set of intervals of a chain, but we implicitly consider its union when taking the convex hull $\Conv_E(g[f(x)])$.

\sidecomment{In short: we can combine replacement functions in sequence in ``the obvious way.''}
\begin{lemma}
	\label{lem:etalt-transitive}
	$\etalt$ is a transitive relation on the set of $\eta$-maximal chains of a poset $P$.
	Moreover, if $C$, $D$, and $E$ are $\eta$-maximal chains in $P$ with $C\etalt D$ and $D\etalt E$ witnessed by $\etawitness{f}{C}{D}$ and $\etawitness{g}{D}{E}$ respectively, then $C\etalt E$ is witnessed by the function $\etawitness{h}{C}{E}$ sending $x$ to $\Conv_E(g[f(x)])$.
\end{lemma}

\begin{proof}
	Assume that chains $C,D,E$ and functions $f,g,h$ are as in the statement of the lemma.
    We must show that the function $h$ witnesses that $C\etalt E$.

	\sidecomment{The rationals are very nice to work with for all of this for this kind of reason: as soon as we get two points, we get a copy of the rationals again!}
	Indeed, fix $x\in C$. 
	If $\abs{f(x)} \geq 2$, then we know that $f(x)$ has order type $\eta$ and \Cref{lem:etalt-comparability} tells us that $g$ maps distinct elements of $f(x)$ to disjoint sets, and so $\Conv_E(g[f(x)])$ is an infinite interval of $\eta$, and it is immediate from its construction that it does not have endpoints, and so it must itself be an interval of $E$ of order type $\eta$, as required.

	Otherwise, $\abs{f(x)} = 1$ and so $f(x) = \set{x}$, but then $h(x) = g(x)$ which is either $\set{x}$ or is an interval of order type $\eta$, again as required.

	Finally, if $x, y\in C$ have $x < y$ and $z \in h(y)$, we must show that $x < z$.
	\sidecomment{Just chase some definitions around.}
	First note that as $C$ has order type $\eta$, there is $w\in C$ with $x < w < y$.
	We claim that $x < f(w) < h(y)$.
	Indeed, $x < f(w)$ follows immediately from the definition of $\eta$-replacement.
	We know from \Cref{lem:etalt-comparability} that $f(w) < f(y)$, and so applying the definition of $\eta$-replacement again, we find that for all $w'\in f(w)$ and all $y' \in f(y)$, we have $w' < g(y')$, and thus $f(w) < h(y)$, completing the proof of this lemma.
\end{proof}

The property proved in the next lemma demonstrates that $\etalt$ is antisymmetric, but will be used directly on several occasions.

\sidecomment{This is pretty easy to prove, but it's very useful when considering long sequences of $\eta$-replacements: a point that leaves can never come back.}
\begin{lemma}
	\label{lem:etalt-antisymmetric}
	Let $C,D,E\sseq P$ be $\eta$-maximal chains with $C\etalt D \etalt E$.
	If there is a point $x$ with $x\in C$ and $x\notin D$, then $x\notin E$.
\end{lemma}

\begin{proof}
	Assume for contradiction that $x \in C$, $x\notin D$ and $x\in E$. 
	Then $C\etalt D$ is witnessed by a function $f$ with $x\notin \bigcup_{u\in C}f(u)$, and so in particular $f(x)\neq\set{x}$.
	Let $D\etalt E$ be witnessed by $g$.
	Note that \Cref{lem:etalt-transitive} tells us that $C \etalt E$ is witnessed by the function $\etawitness{h}{C}{E}$ sending $y\in C$ to
	\begin{equation*}
		h(y) \defined \Conv_E(g[f(y)]).
	\end{equation*}
	
	Consider the set $E\setminus h(x)$. 
	If we have $x\in E\setminus h(x)$, then $h(x)$ is comparable with $x$, and we know by definition of $\etalt$ that $h(x)$ is comparable with $C\setm{x}$, and thus is comparable with $C$, but this contradicts $\eta$-maximality of $C$.
	Thus we must have $x\in h(x)$.

	In this case, either $X\defined h(x)\inter (>x)$ or $Y\defined h(x)\inter (<x)$ must contain a set of order type $\eta$; \wlg\ say it is $Z\sseq X$.
	But then $Z$ is comparable to all of $C$, and so we again contradict $\eta$-maximality of $C$, and the lemma is proved.
\end{proof}

We are now ready to prove that $\etalt$ is a partial order; with the results that we have already proved, this is not difficult.

\begin{lemma}
	\label{lem:etalt-partial-order}
	$\etalt$ is a partial order on the set of $\eta$-maximal chains of a poset $P$.
\end{lemma}

\sidecomment{We've done two of the three difficult steps already, we just need to stick them together and note that reflexivity is obvious.}
\begin{proof}
	We must show that $\etalt$ is reflexive, antisymmetric, and transitive.
	\Cref{lem:etalt-transitive} already tells us that $\etalt$ is transitive, and antisymmetry follows quickly from \Cref{lem:etalt-antisymmetric}: we need to show that $C\etalt D\etalt C$ implies that $C=D$.
	If $C\neq D$ then, as both these chains are $\eta$-maximal, there must be some $x\in C$ such that $x\notin D$.
	But then as $D\etalt C$, \Cref{lem:etalt-antisymmetric} tells us that $x\notin C$, a contradiction.

	It remains to show that $\etalt$ is reflexive, but it is clear that $C\etalt C$ for any $\eta$-maximal chain $C$: this is witnessed by the trivial function $f$ sending $x$ to $\set{x}$.
	Thus $\etalt$ is a partial order, as required.
\end{proof}

Our next task in proving \Cref{prop:esmc-exists} is to prove that any $\etalt$-chain has an upper bound. 
We first deal with increasing sequences indexed by $\ww$, and then show that every $\etalt$-chain has a countable cofinal subchain.

\sidecomment{This is where the real work starts (after writing down the correct definition, which is also real work!) The limiting object is just ``points that are eventually in all the chains'', but proving this works is not simple.}
\begin{lemma}
	\label{lem:countable-etalt-chain-upper-bound}
	Let $P$ be a countable FAC poset and let $C_1,C_2,\dotsc \sseq P$ be $\eta$-maximal chains satisfying
    $$C_1\etalt C_2\etalt C_3\etalt\cdots.$$
	Then there is an $\eta$-maximal chain $C\sseq P$ such that $C_n\etalt C$ for all $n$.
\end{lemma}

\begin{proof}
	We first fix witness functions of $C_n\etalt C_m$, though at some points in this argument we will also consider other, arbitrary witness functions.
	Define
	\begin{equation*}
		\cW_{n,m} \defined \set[\big]{\etawitness{f}{C_n}{C_m} \st f \text{ witnesses } C_n \etalt C_m}.
	\end{equation*}
	Fix $f_n \in \cW_{n,n+1}$ for each $n$.
	For $n\leq m$, let $f_{n,m}\in \cW_{n,m}$ be obtained by composing $f_n,\dotsc,f_{m-1}$ as in \Cref{lem:etalt-transitive} (i.e.\ the composition of $g_1$ and $g_2$ sends $x$ to $\Conv_D(g_2[g_1(x)])$ for the relevant chain $D$) and let $f_{n,n}\in \cW_{n,n}$ set $x$ to $\set{x}$ for all $x$.
	
	We construct an auxiliary digraph $G$ as follows.
	\begin{align*}
		V(G) &\defined \set{(n,x) \st 1\leq n < \ww, x\in C_n}, \\
		\p[\big]{(n,x), (m,y)} \in E(G) &\iff n < m,\, x \neq y, \text{ and } \exists g\in \cW_{n,m} \text{ with } y \in g(x).
	\end{align*}
	We write $(n,x)\to(m,y)$ when there is an edge in $E(G)$ from $(n,x)$ to $(m,y)$.

	Note that the digraph $G$ is transitive: if $(n,x)\to(m,y)\to(\l,z)$, then \Cref{lem:etalt-antisymmetric} tells us that $z\neq x$ and so, by composition of witness functions, $(n,x) \to (\l,z)$.
	The key property of $G$ is captured by the following claim.

	\sidecomment{This really is the key property that shows that the idea ``take things eventually in all the chains'' can work, the rest is somewhat book-keeping, though still not entirely trivial.}
	\begin{claim}
		\label{claim:eta-no-infinite-branch}
		$G$ has no infinite directed path $a_1\to a_2 \to a_3\to\dotsb$.
	\end{claim}

	\begin{proof}
		Assume for contradiction that we had an infinite path
		\begin{equation*}
			(n_1,x_1) \to (n_2,x_2) \to (n_3,x_3) \to\dotsb.
		\end{equation*}
		Transitivity of $G$ tells us that, for all $i<j$, we have $(n_i,x_i) \to (n_j,x_j)$, and so $x_i \neq x_j$.
		Moreover, there is a witness function $g\in\cW_{n_i,n_j}$ with $x_j\in g(x_i)$, and so $x_i \incomp x_j$ by \Cref{lem:etalt-incomparability}.
		But then $\set{x_i \st i < \ww}$ is an infinite antichain, contradicting the fact that $P$ is an FAC poset.
	\end{proof}

	Say that a vertex $(n,x)\in V(G)$ is \emph{stable} if it has no outward edge.
	This implies that $x\in C_m$ for all $m\geq n$.
	An immediate corollary of \Cref{claim:eta-no-infinite-branch} is that, for all $(n,x)\in V(G)$, either $(n,x)$ is itself stable, or there is $(m,y) \in V(G)$ with $(n,x)\to (m,y)$ and $(m,y)$ stable, as otherwise we could easily construct an infinite directed path in $G$.

	\sidecomment{The intuition should be that $L$ is the genuine limit object we want, but $L$ as written isn't necessarily $\eta$-maximal.}
	Define the set $L\sseq P$ as follows.
	\begin{equation*}
		L \defined \liminf_{n\to\infty} C_n
		= \set{x\in P\st x\in C_m\text{ for all sufficiently large }m}.
	\end{equation*}
	We will construct the chain $C$ required in the lemma statement as an extension of $L$.
	Note that we already know that $L$ is non-empty.
	To use $L$, we need to establish some of its properties.

	\begin{claim}
		\label{claim:properties-of-L}
		$L$ is a chain with order type $\eta$.
	\end{claim}
	\sidecomment{This is mostly just book-keeping. There are a few ways through this, but this seems like one of the quicker ones.}

	\begin{proof}
		We first show that $L$ is a chain.
		Take $x,y\in L$.
		Then there is $n\in\ww$ with $x,y\in C_n$, so $x\comp y$ and thus $L$ is a chain.
		So it suffices to prove that $L$ has order-type $\eta$.

		If $x,y\in L$ have $x<y$, then choose $n$ large enough that $x,y\in C_n$.
		$C_n$ itself has order type $\eta$, and so we can choose $z\in C_n$ with $x<z<y$.
		Let $(m,z') \in V(G)$ be a stable point with either $(m,z') = (n,z)$ or $(n,z)\to(m,z')$.
		By the definition of $\eta$-replacement, we know that $x<z'<y$ and $z'\in L$.
		A similar argument shows that $L$ has no endpoints.
		As $L$ is countable, has no covers or endpoints, we know that $L$ is indeed isomorphic to the rationals, as required.
	\end{proof}

	Let $C$ be an arbitrary $\eta$-maximal chain containing $L$, which we know exists by Zorn's lemma.
	\sidecomment{This just means there's one fewer parameter to drag around the whole time.}
	We will show that $C$ is the required upper bound: we must show that $C_n\etalt C$ for all $n$.
	In fact, we show that $C_1\etalt C$, as the same argument holds for all $n$.

	For each $x\in C_1$ we define a non-empty set $L_x \sseq L$.
	If $(1,x)\in V(G)$ is stable, then we take $L_x \defined \set{x}$.
	\sidecomment{We define these sets $L_x$ so that we can build the witness function. They are exactly the elements you'd expect.}
	Otherwise, choose $n > 1$ with $x\notin C_n$.
	Then $f_{1,n}(x)$ has order-type $\eta$, as it cannot be $\set{x}$.
	For each $y\in f_{1,n}(x)$, let $s(y)$ be a point such that either $(m,s(y)) = (n,y)$ or $(n,y)\to(m,s(y))$ for some $m\geq n$ (depending on $y$) and $(m,s(y))$ is a stable point of $V(G)$.
	Note that we take $s(y) = y$ if $(n,y)$ is itself stable.
	In this case, set
	\begin{equation*}
		L_x \defined \set{s(y) \st y \in f_{1,n}(x)}.
	\end{equation*}
	For $y,z\in f_{1,n}(x)$ with $y<z$, note that $s(y),s(z) \in C_m$ for some sufficiently large $m$, and, in particular, there is $g\in \cW_{n,m}$ such that $s(y)\in g(y)$ and $s(z)\in g(z)$ and so $s(y) < s(z)$ by \Cref{lem:etalt-comparability}.
	This means that the function $s$ is order-preserving.
	Thus $L_x$ has order type $\eta$ whenever it is not a singleton.

	\sidecomment{At this point we just need to check that the actual formal conditions of $\eta$-replacement hold.}
	A similar argument shows that, if $w < x$ in $C_1$, we have that $L_w < x$ and $w < L_x$.
	We may then define the function $\etawitness{h}{C_1}{C}$ by
	\begin{equation*}
		h(x) \defined \Conv_C(L_x).
	\end{equation*}
	It follows that $h(x)$ is either a singleton or has order-type $\eta$ (as the convex hull of a sub-order of $\eta$ without endpoints) and that $w < x$ implies $h(w) < x$ and $w < h(x)$, and so $h$ witnesses an $\eta$-replacement.
	This shows that $C_1\etalt C$, completing the proof of \Cref{lem:countable-etalt-chain-upper-bound}.
\end{proof}

\sidecomment{I've always thought of this step as something of an annoyance. If we knew that any $\etalt$-chain-of-chains was countable, then we'd be done, but this fact doesn't seem obvious.}
However, to apply Zorn, we need to deal with arbitrary chains, not just those indexed by $\ww$.
The following lemma allows us to reduce to the case of chains indexed by $\ww$, so \Cref{lem:countable-etalt-chain-upper-bound} suffices.

\begin{lemma}
	\label{lem:etalt-chain-countable-cofinal}
	Let $P$ be a countable poset. 
	Then every non-empty $\etalt$-chain of $\eta$-maximal chains of $P$ either has a greatest member or has a countable cofinal increasing sequence.
\end{lemma}

\begin{proof}
	Assume for contradiction that $\cS$ is a non-empty chain of $\eta$-maximal chains of $P$ ordered by $\etalt$, and assume that $\cS$ has no largest element.

	\sidecomment{Idea of this proof: for all $x$, $\cS$ splits into three intervals: those chains not containing $x$ before it appears, those containing $x$, and those not containing $x$ after it has been removed. Once you remove $x$, you can't add it back. There are only countably many $x$, so pick something in the no-$x$ side for each $x$ that ever gets removed.}
	We say that a replacement $A\etalt B$ \emph{removes $x$} if $x\in A$ and $x\notin B$.
	For every $x\in P$ for which there is a replacement in $\cS$ removing $x$, fix an arbitrary pair $A_x, B_x \in \cS$ such that $A_x \etalt B_x$, $x\in A_x$ and $x\notin B_x$.
	Let $\cF$ be the family of all such chains $B_x$.

	We claim that $\cF \sseq \cS$ is cofinal and countable.
	Firstly, the elements of $\cF$ are indexed by the countable poset $P$, and so $\cF$ is likewise countable.

	We need to show that, for all $C\in\cS$, there is $x$ such that $B_x$ exists and we have $C\etalt B_x$.
	As $\cS$ has no maximum, there is some $D\in\cS$ such that $C\etalt D$.
	The $\eta$-maximality of $C$ implies that there is some element $y\in C\setminus D$.
	We consider $B_y$.
	We cannot have $B_y = C$ as $y\notin B_y$ and $y\in C$.
	This means that either $B_y \etalt C$ or $C \etalt B_y$. 
	We claim the latter holds.

	Assume for contradiction that $B_y \etalt C$.
	This means that $A_y \etalt B_y \etalt C$, but of these three chains, $y$ is present in only $A_y$ and $C$, contradicting \Cref{lem:etalt-antisymmetric}.
	Thus $C \etalt B_y$, as required.

	\sidecomment{Note that we have not proven that $\cS$ itself is countable, and it seems possible that it can in fact be uncountable.}
	Enumerate $\cF$ with indices in $\ww$ arbitrarily as $D_1,D_2,\dotsc$, possibly with repetitions, and let
	\begin{equation*}
		E_n \defined \max_{\etalt}\set{D_1,D_2,\dotsc,D_n}.
	\end{equation*}
	Then $E_1 \etalt E_2 \etalt \dotsb$ is a cofinal sequence in $\cS$, as claimed.
\end{proof}

We may now deduce \Cref{prop:esmc-exists}.

\sidecomment{This is just gluing together things that we've already proved.}
\begin{proof}[Proof of \Cref{prop:esmc-exists}]
	We know from \Cref{lem:etalt-partial-order} that $\etalt$ is a partial order, and we claim that every $\etalt$-chain of chains has an $\eta$-maximal chain as an upper bound.
	Indeed, by \Cref{lem:etalt-chain-countable-cofinal}, such a chain either has a greatest member or a countable cofinal increasing sequence.
	In the latter case, \Cref{lem:countable-etalt-chain-upper-bound} gives an upper bound for the sequence, and this is also an upper bound for the original chain.
	
	Thus all $\etalt$-chains of chains have upper bounds and we may apply Zorn's lemma to find an $\eta$-maximal chain $C$ which is maximal under $\etalt$.
	
	This chain $C$ is then by definition an \esmc, as required.
\end{proof}

With these results in hand, we may now deduce the structural result \Cref{thm:structural}.
Recall that $\RRP \defined \RR\union\set{-\infty, +\infty}$ is the order-completion of the reals.

\sidecomment{There's not really too much content in this proof; the work has been done in building the \esmc. However, there are still quite a few conditions to check and this takes some space even if every individual step is easy.}
\begin{proof}[Proof of \Cref{thm:structural}]
	Let $P$ be a countable FAC poset.
	If $P$ is scattered then we are done, so assume that $P$ admits an embedding of $\QQ$.

	We know from \Cref{prop:esmc-exists} that $P$ has an \esmc\ $C$.
	For notational convenience, identify $C$ with $\QQ$.
	For all $r\in \RRP$, define the set $A_r$ as follows.
	\begin{equation}
		A_r \defined \p[\big]{\;\below{r}_C, \, \above{r}_C\,}_P\label{eq:define-a},
	\end{equation}
	where the relations between $r\in\RRP$ and $\QQ$ are as expected from the natural embedding $\QQ\sseq\RRP$.
	We now prove the required properties of the sets $A_r$ in turn.

	First, for all $q\in \QQ$, we have that $q \in A_q$, and so $A_q$ is non-empty for all $q\in \QQ$.

	Second, for $r, s \in \RRP$ with $r < s$, we have that $A_r < A_s$; indeed, if $x\in A_r$ and $y\in A_s$, then there is some $q\in \QQ$ with $r < q < s$, and for this $q$ by \eqref{eq:define-a} we have $x < q < y$, and so $x < y$.

	Third, take some $x \in P\setminus \bigcup_{r\in \RRP} A_r$. We must show that there is some infinite interval $I_x \sseq \RRP$ such that $x$ is incomparable to $A_r$ for all $r\in I_x$.
	Note that as soon as an interval of $\RRP$ contains two points, it is infinite.
	Indeed, we may define
	\sidecomment{As this interval is a sub-interval of $\RRP$, either it's empty or a singleton and $x$ should be in some $A_r$, or it's infinite and we're done.}
	\begin{equation*}
		I_x \defined \Conv_{\RRP} \set{r \in \RRP \st A_r \neq \emptyset \text{ and } x \incomp A_r},
	\end{equation*}
    where by $x\incomp A_r$ we mean that $x$ is incomparable to all of $A_r$.
	$I_x$ is certainly an interval, and moreover $I_x$ is not empty, as otherwise $x$ would be comparable to all of $C$, and so we would have $x\in A_r$ for some $r$.
	Finally, if $I_x$ were finite, then we would have $I_x = \set{r}$ for some $r\in \RRP$ ($I_x$ having two elements quickly implies that it is in fact infinite).
	In fact, we must have $r\in \QQ$, as otherwise $x$ would again be comparable to all of $C$.
	As $x\notin A_r$, we thus know that there is some $q \in\QQ$ with either $q > r$ and $q \not > x$, or $q < r$ and $q \not < x$.
	But then $x$ must be incomparable with the entire interval between $r$ and $q$, contradicting the assumption that $I_x$ is finite.

    Fourth, $A_r$ is convex, as if $x,y\in A_r$ and $z\in P$ satisfy $x<z<y$, then for all $q\in C$, if $q > r$, then $q > y > z$, and if $q < r$, then $q < x < z$, and so $z \in A_r$.

	Finally, $A_r$ is scattered for all $r\in\RRP$, as, if $r\notin \QQ$ and $A_r$ admitted an embedding of $\eta$, then the chain $C$ is not $\eta$-maximal, and if $r\in\QQ$, then we could perform a non-trivial $\eta$-replacement on $C$ by replacing $r$ with the copy of $\eta$ in $A_r$ and extending to an $\eta$-maximal chain, contradicting the fact that $C$ is an \esmc.

	This completes the proof of \Cref{thm:structural}.
\end{proof}

Finally, we can deduce that the existence of a maximal tube in a countable vacillating FAC poset $P$ is implied by the existence of a maximal tube in an arbitrary countable vacillating scattered FAC poset.

\begin{corollary}
	\label{cor:scattered-reduction}
	Assume that every countable vacillating scattered FAC poset has a maximal tube. Then every countable vacillating FAC poset has a maximal tube.
\end{corollary}

\begin{proof}
	Let $P$ be a countable vacillating FAC poset. If $P$ is scattered then we are done, so assume that $P$ admits an embedding of $\QQ$ and apply \Cref{thm:structural} to find a family of scattered sets $(A_r\st r\in \RRP)$ satisfying the conclusion of the theorem.
	\sidecomment{In short, just apply the structural result, find a maximal tube in each blob, glue them all together, and it quickly follows that this has the properties we want.}
	
	For all $r\in \RRP$, note that $A_r$ inherits the property of vacillation from $P$, and let $T_r$ be a maximal tube of $A_r$.
	We claim that
	\begin{equation*}
		T\defined\bigcup_{r\in \RRP} T_r
	\end{equation*}
	is a maximal tube of $P$.

	There are two properties that we must show: that every element $x\in T$ is comparable to all but finitely many elements of $T$, and that every element $y \in P \setminus T$ is incomparable to infinitely many elements of $T$.
	Indeed, take $x\in T_r$ for some $r\in\RRP$.
	Then $x$ is comparable to all but finitely many elements of $T_r$, and is comparable to $A_s$ for all $s\neq r$, as required.

	Finally, if $y\in P\setminus T$, then if $y\in A_r$ for some $r\in\RRP$, then $y$ is incomparable to infinitely many elements of $T_r$, and if $y\notin A_r$ for all $r\in\RRP$, then there is some infinite interval $I_y$ such that $y$ is incomparable to $A_r$ for all $r\in I_y$, and thus is incomparable to infinitely many elements of $\bigcup_{r\in I_y} T_r$, as required.
\end{proof}

\section{Separating chains}
\label{sec:separating}

The goal of this section is to replace the property of being vacillating with a stronger property, in which $N_+(P)$ is well-founded and $N_-(P)$ is co-well-founded.
This property will be important to track, so we give it a name.

\begin{definition}
	\label{def:quasifdd}
	A poset $P$ is called \emph{quasifounded} if $N_+(P)$ is well-founded and $N_-(P)$ is co-well-founded.
\end{definition}

\begin{rambling}
	One possible lens through which to compare quasifoundedness and vacillation is the following.
	Vacillation states that there is no saturated chain $C$ in $P$ consisting of an infinite decreasing sequence of elements of $N_+(C)$ and avoiding $N_-(C)$ internally, and the dual.
	Quasifoundedness strengthens this condition by removing the requirement that the chains must be saturated.
\end{rambling}

To be more precise, this section is dedicated to proving the following result.

\sidecomment{This is a bit of a mouthful, but note we just step from ``quasifounded etc.'' to ``vacillating etc.''}
\begin{proposition}
	\label{prop:quasifdd-reduction}
	If every countable quasifounded scattered FAC poset contains a maximal tube, then every countable vacillating scattered FAC poset contains a maximal tube.
\end{proposition}

Before going any further, and to help build a little intuition for quasifoundedness, we prove the following two simple lemmas.

\begin{lemma}
	\label{lem:quasifounded-implies-vacillating}
	If a scattered poset $P$ is quasifounded, then $P$ is vacillating.
\end{lemma}

\sidecomment{This really is a triviality, but it's good to write it down somewhere.}
\begin{proof}
	Assume for contradiction that $P$ is quasifounded but not vacillating. Then there is a saturated chain $C\sseq P$ with, without loss of generality, $C = \bigoplus_{i\in \ww} C_i$, where each $C_i$ is infinite and co-well-ordered.
	However, each $N_-(C_i)$ is non-empty by \Cref{lem:h-complete}, and so this immediately gives an infinite increasing sequence in $N_-(C)$, contradicting quasifoundedness and proving the lemma.
\end{proof}

\sidecomment{If we don't have any internal increasing limits, then there's nothing to break up sequences of decreasing limits, and so vacillation and quasifoundedness coincide.}
\begin{lemma}
	\label{lem:vacillating-no-internal-is-quasifounded}
	If a scattered vacillating poset $P$ has no internal point of $N_+(P)$ (i.e.\ if $X\in N_+(P)$ and $x\in P$ then $X \not< x$), then $P$ is quasifounded.
	A similar result holds if there is no internal point of $N_-(P)$.
\end{lemma}

\begin{proof}
	Assume for contradiction that the poset $P$ in the lemma is not quasifounded.
	Then $P$ has either an infinite decreasing sequence in $N_+(P)$ or an infinite increasing sequence in $N_-(P)$. 
	The former is impossible as no element of $N_+(P)$ is below any element of $P$, so we know that there is an infinite increasing sequence $X_1 < X_2 < \dotsb$ in $N_-(P)$.
	Let $C$ be a saturated chain witnessing all of these points, as guaranteed by \Cref{lem:simultaneous-witness-infinite}.
	We know that $N_+(C)$ has no point internal to $C$, and so each interval $(X_i, X_{i+1})_P$ is co-well-founded.
	But we thus find a saturated chain
	\begin{equation*}
		\bigcup_{i\geq 1} (X_i, X_{i+1})_C
	\end{equation*}
	which is a well-ordered sum of co-well-ordered chains, contradicting the vacillation of $P$, proving the lemma.
\end{proof}

\sidecomment{This is another one of those key definitions which looks natural, but it's difficult to think to write it down. This one in particular was suggested by AI as a much faster way through than the enormous slog I was previously embarking on.}
The key to proving \Cref{prop:quasifdd-reduction} is to break a vacillating poset $P$ down into quasifounded parts in which we can find maximal tubes.
To do this, we will find suitable quasifounded parts $Q\sseq P$ in which a maximal tube in $Q$ is in some sense ``locally maximal'' in $P$. 
More formally, we use the following definition.

\begin{definition}
	\label{def:separating}
	A non-empty tube $T \sseq P$ is a \emph{separating tube} if, for all $x\in P\setminus T$, exactly one of the following three cases holds.
	\begin{enumerate}
		\itemsep=0pt
		\item $x < T$,
		\item $x > T$, or
		\item $\abs[\big]{\incompset{x}_P \inter T} = \infty$.
	\end{enumerate}
\end{definition}

\sidecomment{Just think ``separating $=$ locally maximal''}

\Cref{def:separating} can be interpreted as saying that, for any $x\in P \setminus T$, either the conditions of $T$ being a maximal tube apply, or $x$ is comparable to all of $T$.
In this way, $T$ naturally breaks $P$ into four parts (some of which might be empty): $T$ itself, the elements above $T$, the elements below $T$, and the set of elements which see $T$ as a maximal tube.
If we can always find separating tubes, then we will be able to build up a maximal tube piece by piece: find a separating tube $T$, break the poset up into the parts above and below $T$, and recurse on those parts.
Indeed, this is how we will prove \Cref{prop:quasifdd-reduction}.

Recall that a poset $Q$ is \emph{open} if it contains neither minimal nor maximal elements.
We will find these separating tubes by taking maximal tubes in convex open $Q\sseq P$. 

There are two key properties of open posets which we will use.
First, \Cref{lem:maximal-becomes-separating} says that open convex quasifounded posets yield separating tubes.
Second, \Cref{lem:find-convex-open-subposet} says that, if a countable vacillating FAC poset contains a saturated chain of order type $\zeta$, then we can find open convex quasifounded posets in $P$.
However, before getting to these results, we need two technical lemmas, which we get out of the way now.

\sidecomment{Note that $\zeta$ is the minimal non-empty open chain.}
\begin{lemma}
	\label{lem:find-zeta}
	Let $C$ be a non-empty open vacillating scattered chain. 
	Then there is a saturated sub-chain $D\sseq C$ with order type $\zeta$.
\end{lemma}

\begin{proof}
	\sidecomment{This feels as though this maybe should be easy, though of course $\QQ$ has no saturated copy of $\zeta$, so we're going to need to really use the fact that we're scattered here, so we build a copy of the rationals for a contradiction.}
	Assume for contradiction that $C$ contains no saturated copy of $\zeta$. 
	We prove that this implies that $C$ has two saturated sub-chains $C_1,C_2 \sseq C$ with $C_1 < C_2$, both of which are open (and inherit the properties of being vacillating and scattered).
	Iterating this property yields a copy of the dyadic rationals, contradicting the fact that $C$ is scattered.
	It thus suffices to find $C_1,C_2$ as claimed.

	We first claim that $N_+(C)$ has no minimal element.

	Assume for contradiction that $X\in N_+(C)$ is minimal, and consider $D \defined \below{X}_C$. 
	This chain can have no internal element in $N_+(D)$, and so \Cref{lem:vacillating-no-internal-is-quasifounded} implies that $D$ is quasifounded.
	Noting that $N_-(D)$ is non-empty as $C$ is open, $N_-(D)$ has a maximal element $Y$.
	But then the interval $(Y,X)_C$ has no internal nonprincipal element, and so we must have $(Y,X)_C \cong \zeta$, contradicting our assumption.
	Thus $N_+(C)$ has no minimal element.

	Take $X_1 > X_2 > \dotsb$ to be an infinite decreasing sequence in $N_+(C)$.
	As $C$ is vacillating, \Cref{lem:witness-vacillation} gives us $Y_1,Y_2 \in N_-(C)$ and $i < j$ such that $X_i > Y_1 > X_j > Y_2$.
	Indeed, $C_1 \defined (Y_2,X_j)_C$ and $C_2 \defined (Y_1,X_i)_C$ are two disjoint saturated open sub-chains, as required.
\end{proof}

\sidecomment{We can actually find infinitely many pairwise comparable copies of $\zeta$.}
\begin{lemma}
	\label{lem:find-open}
	Assume that $P$ is a vacillating poset which is not quasifounded. 
	Then there are pairwise disjoint saturated chains $C_1,C_2,\dotsc$ in $P$ which are all open and pairwise comparable. In particular, their union is itself a chain.
\end{lemma}

\begin{proof}
	\sidecomment{The work is partially done in \Cref{lem:witness-vacillation}, which hands us just the sequence of limit points we need. This isn't a coincidence: this lemma is the reason that \Cref{lem:witness-vacillation} was written.}
	As $P$ is not quasifounded, we may assume that there is an infinite decreasing sequence in $N_+(P)$; the case of an infinite increasing sequence in $N_-(P)$ is similar.
	Let $X_1 > X_2 > \dotsb$ be an infinite decreasing sequence of elements of $N_+(P)$.
	\Cref{lem:witness-vacillation} tells us that we can find $i_1 < i_2 < \dotsb$ and $Y_1 > Y_2 > \dotsb$ in $N_-(P)$ such that
	\begin{equation*}
		X_{i_1} > Y_1 > X_{i_2} > Y_2 > \dotsb.
	\end{equation*}
	This is exactly condition \eqref{eq:vacillating-sequence}, but we repeat it here for convenience.
	Moreover, \Cref{lem:witness-vacillation} gives us a chain $C$ witnessing all of these elements simultaneously.

	Let $C_j \defined (Y_j, X_{i_j})_C$.
	These are open saturated chains satisfying the conditions in question, as required.
\end{proof}

We now return to our two main lemmas for this section, stating and proving them in turn.

\sidecomment{This formalises the intuition ``separating $=$ locally maximal''.}
\begin{lemma}
	\label{lem:maximal-becomes-separating}
	If $P$ is a poset and $Q\sseq P$ is non-empty, open, and convex and $T\sseq Q$ is a maximal tube of $Q$, then $T$ is a separating tube of $P$.
\end{lemma}

\begin{proof}
	We will show that $T$ is open.
	First, we claim that this is sufficient to prove the lemma.

	\sidecomment{The openness condition is necessary here. Silly counterexample: let $Q$ be a singleton.}
	Assume that $T$ is open.
	We show that $T$ is a separating tube of $P$ by showing that every $x\in P\setminus T$ is either above $T$, below $T$, or incomparable to infinitely many elements of $T$.
	We first show that if $T\sseq\compset{x}_P$, then $x$ is either above $T$ or below $T$; if there are $y,z\in T$ with $y<x<z$, then we know by convexity that $x\in Q$, and so by maximality of $T$ we have $x\in T$, a contradiction.

	Thus assume for contradiction that there is $x\in P\setminus T$ with
	\begin{equation*}
		0 < \abs{\incompset{x}_P \inter T} < \infty.
	\end{equation*}
	Note that $x\notin Q$, as otherwise this would contradict the definition of $T$ as a maximal tube of $Q$.
	We will show that there are $y, z \in T$ with $y < x < z$, which will contradict the convexity of $Q$.

	Assume for contradiction that there is no $y\in T$ with $y < x$; the case of $y > x$ is similar.
	We know that there is $w\in T$ with $w\incomp x$ by assumption.
	$T$ is open and so there are $w > w_1 > w_2 > \dotsb$ forming an infinite decreasing sequence in $T$.
	We cannot have $x < w_i$ for some $i$, as then we would have $x < w_i < w$, contradicting $x\incomp w$.
	As $x$ is incomparable with only finitely many elements of $T$, we must have $x > w_i$ for some $i$.
	But then $y = w_i$ is as claimed.
	Thus there are $y,z \in Q$ with $y < x < z$, contradicting the convexity of $Q$.

	\medskip

	\sidecomment{A maximal tube should in some way reflect the structure of the poset it's in, and this is one example of that.}
	It remains to prove that $T$ is indeed open.
	Assume for contradiction that $y \in T$ is maximal (the case of a minimal element is similar).
	As $Q$ is open, there must be an element $x\in Q$ with $x > y$.
	As $T$ is a maximal tube in $Q$, we know that $x$ is incomparable with infinitely many elements of $Q$: the set $\incompset{x}_Q \inter T$ is infinite.
	As $y\in T$, we know that $y$ is incomparable with only finitely many elements in $T$.
	There must therefore be an element $z\in \incompset{x}_Q \inter T$ which is comparable with $y$.
	As $y$ is maximal in $T$ and $z\in T$, we cannot have $z > y$, so we know that in fact $z < y$.
	But then $z < y < x$, contradicting the fact that $z \in \incompset{x}_Q$, i.e.\ $z \incomp x$.
	This implies that $T$ is open, and so \Cref{lem:maximal-becomes-separating} is proved.
\end{proof}

\sidecomment{This proof has real ideas in it! The steps here are the argument which this whole section relies on. The key idea is taking a maximal element of $\cA$. Once we've done that, the rest of the proof isn't too hard to push through.}
\begin{lemma}
	\label{lem:find-convex-open-subposet}
	Let $P$ be a vacillating scattered FAC poset and assume that there is a saturated chain $C\sseq P$ with $C\cong \zeta$.
	Then there is a non-empty $Q\sseq P$ which is convex, open, and quasifounded.
\end{lemma}

\begin{proof}
	The key idea in this proof is to take $\cA$ to be the family of antichains $A$ of $P$ for which $\incompset{A}_P$ contains a saturated copy of $\zeta$, and then take a maximal element of $\cA$.
	More formally, define $\cA$ as follows.
	\begin{equation*}
		\cA \defined \set[\big]{A\sseq P \st A \text{ is an antichain, } \exists\, C \sseq \incompset{A}_P \text{ saturated such that } C \cong \zeta}.
	\end{equation*}
	\sidecomment{All we need to do once we've defined $A$ is find a copy of $\zz$ and an element $x$ in $\incompset{A}_P$ with $x$ incomparable to the whole $\zz$.}
	Let $A$ be a maximal element of $\cA$ under inclusion, which we know must exist as $P$ has no infinite antichains.
	Let $C\sseq\incompset{A}_P$ be saturated and with order type $\zeta$.
	We first claim that every bounded interval of $\Conv_P(C)$ is quasifounded.

	Assume for contradiction that this is not the case.
	There are thus $c_1,c_2\in C$ and $D \sseq \Conv_P(C)$ such that $c_1 < D < c_2$ and $D$ is a saturated chain such that $N_+(D)$ contains an infinite decreasing sequence $X_1 > X_2 > \dotsb$ (the case of an increasing sequence in $N_-(D)$ is similar).
	As $P$ is vacillating, we can apply \Cref{lem:find-open} to find an infinite decreasing sequence $E_1 > E_2 > \dotsb$ of saturated open sub-chains of $D$.

	\sidecomment{These conditions look a bit technical, but they follow immediately from $D$ being a chain.}
	Every point in $D$ is parallel to some non-empty interval of the finite set $(c_1,c_2)_C$, and if $x,y \in D$ have $x < y$, then we have
	\begin{align*}
		\min\p[\big]{\incompset{x}_P \inter (c_1,c_2)_C} &\leq \min\p[\big]{\incompset{y}_P \inter (c_1,c_2)_C}, \text{ and}\\
		\max\p[\big]{\incompset{x}_P \inter (c_1,c_2)_C} &\leq \max\p[\big]{\incompset{y}_P \inter (c_1,c_2)_C}.
	\end{align*}
	As $(c_1,c_2)_C$ is finite, the intervals $\incompset{z}_P \inter (c_1,c_2)_C$ must stabilise for $z$ below some point in $D$.
	Thus there are $c_3,c_4 \in C$ with $c_1 < c_3 \leq c_4 < c_2$ and there is $n\geq 1$ such that, for all $m\geq n$ and all $z\in E_m$, we have
	\begin{equation*}
		\incompset{z}_C = [c_3,c_4]_C.
	\end{equation*}
	In particular, note that $E_n \sseq \incompset{c_3}_P$.

	But then \Cref{lem:find-zeta} tells us that there is a saturated chain $F\sseq E_n$ with $F\cong \zeta$.
	We would then have $F \sseq \incompset{(A\union \set{c_3})}_P$, contradicting the maximality of the antichain $A$, as $c_3 \incomp A$.

	We thus conclude that all bounded intervals of $\Conv_P(C)$ are quasifounded.
	\sidecomment{Now we can make the last little step down to proper quasifoundedness by just passing to a proper interval.}

	If $\Conv(C)$ itself is not quasifounded, then we claim that we can just pass to a smaller convex open sub-poset of $P$ which is quasifounded.
	Indeed, if $\Conv(C)$ is not quasifounded, then \Cref{lem:find-open} gives us an infinite set of pairwise comparable open chains in $P$.
	In particular, we can find three such chains with $G_1 < G_2 < G_3$.
	$\Conv(G_2)$ is then a bounded open interval of $\Conv_P(C)$, and so we know it is quasifounded.

	This proves the lemma.
\end{proof}

Finally, we combine the above lemmas into the following useful corollary.

\sidecomment{This is just gluing things together.}
\begin{corollary}
	\label{cor:find-separating}
	Assume that every countable scattered quasifounded FAC poset contains a maximal tube.
	If $Q\sseq P$ is a non-quasifounded convex sub-poset of $P$, a countable scattered vacillating FAC poset, then $Q$ contains a separating tube of $P$.
\end{corollary}

\begin{proof}
	We know from \Cref{lem:find-open} that there is a saturated open chain $C\sseq Q$, and so \Cref{lem:find-zeta} gives us a saturated chain in $Q$ of order type $\zeta$.
	We may therefore apply \Cref{lem:find-convex-open-subposet} to find a sub-poset $Q'\sseq Q$ which is non-empty, convex, open, and quasifounded.
	By the assumption in the statement, $Q'$ contains a tube $T$ which is maximal in $Q'$, and \Cref{lem:maximal-becomes-separating} then tells us that $T$ is separating in $P$, as required.
\end{proof}

We need one final definition before moving to the final proof of this section.

\begin{definition}
	\label{def:separating-family}
	Call a family $\cF$ of tubes of $P$ a \emph{separating family} if they are pairwise disjoint and comparable separating tubes of $P$.
\end{definition}

\sidecomment{All we need to do is build the biggest separating family that we can, and then stick all the pieces together. There are a few things to check, but they all go through easily.}
With these lemmas in hand, we can now prove \Cref{prop:quasifdd-reduction}.

\begin{proof}[Proof of \Cref{prop:quasifdd-reduction}]
	As the union of a nested family of separating families is itself a separating family, Zorn's lemma applies to give us a maximal separating family $\cF_0$.

	Note that $\cF_0$ (and indeed any separating family) inherits a linear order from $P$, as the separating tubes it contains are pairwise comparable. 
	Let $\cI\sseq \cF_0$ be an initial segment in this inherited order, and define the subposet $Q_\cI \sseq P$ by
	\begin{equation*}
		Q_\cI \defined \p[\Big]{\bigcup\, \cI, \; \bigcup\, \p{\cF_0 \setminus \cI}}_P\;.
	\end{equation*}
	We claim that $Q_\cI$ is a convex quasifounded poset: if it were not quasifounded, then by \Cref{cor:find-separating} we could find a separating tube of $P$ inside $Q_\cI$, contradicting the maximality of $\cF_0$.
	As $Q_\cI$ is quasifounded, the assumption in the statement tells us that $Q_\cI$ contains a tube $T_\cI$ maximal in $Q_\cI$.
	We therefore define
	\begin{equation*}
		T \defined \bigcup \cF_0 \,\union\, \bigcup_\cI T_\cI,
	\end{equation*}
	where the second union is over all initial segments $\cI$ of $\cF_0$.
	We claim that $T$ is a maximal tube in $P$.

	First, $T$ is indeed a tube: $T$ is a union of pairwise comparable tubes, and so if $x\in T$, then $\incompset{x}_T$ is contained in either some $T_\cI$ or an element of $\cF_0$ and is in either case finite.

	Secondly, $T$ is maximal. 
	Assume for contradiction that there is $x\in P\setminus T$ with $\incompset{x}_P \inter T$ finite.
	We know then that, for all $S\in \cF_0$, $x$ is comparable to $S$, else $\incompset{x}_P \inter S \sseq \incompset{x}_P \inter T$ would be infinite.
	Thus there is some $\cI$ such that $x\in Q_\cI$.
	But $T_\cI \sseq Q_\cI$ is maximal in $Q_\cI$, so either $x\in T_\cI$ or $\incompset{x}_{Q_\cI}\inter T_\cI \sseq \incompset{x}_P \inter T$ is infinite.

	Thus $T$ is indeed a maximal tube in $P$, as required.
\end{proof}

We close this section with the following remark, giving an alternative characterisation of quasifounded posets.

\sidecomment{This statement is well within the reach of modern AI systems; I haven't proved it here because if I kept adding new things to prove, this paper would be even more infinitely long than it is now.}
\begin{remark}[Alternative characterisation of quasifoundedness]
	\label{rem:alternative-quasifdd}
	There is another natural definition of quasifoundedness: say that a poset $P$ has a \emph{founded decomposition} if the elements of $P$ can be partitioned $P = A \sqcup B$ where the sub-poset of $P$ induced by $A$ is well-founded and by $B$ is co-well-founded. 
	Note that there may be arbitrary relations between $A$ and $B$.

	Then, if $P$ is countable or an FAC poset (or both), then $P$ is quasifounded if and only if it has a founded decomposition.
\end{remark}

\Cref{rem:alternative-quasifdd} does require proof, though we do not provide this here as it is not immediate and not needed for our results.

\section{Chain consolidation}
\label{sec:consolidation}

In this section, we define the final operation we will consider on our way to proving \Cref{thm:main}: consolidation.
The definition of the consolidation order is chosen so that a consolidation-maximal chain (a ``consolidated chain'') will naturally yield a maximal tube via a simple conversion discussed in \Cref{sec:finding-tubes}.
However, before getting there, we must define the consolidation order and prove that it is a partial order in vacillating posets, which is an undertaking requiring some effort.
We now give a brief outline of this section, which is broken into several subsections, before continuing.

First, in \Cref{subsec:define-consolidation}, we define the consolidation order on maximal chains of a poset $P$.
To do this, we will need to examine in more detail the structure of $N(P)$, in particular defining certain elements of this set to be \emph{atomic}, meaning that they cannot be broken up into smaller parts.
For example, a copy of $\ww$ is always atomic, as it is order-isomorphic to all cofinal subsets of itself.
The consolidation order $\consolidatesto$ can then be defined as an operation which, to a certain extent, preserves the structure and ordering of these atomic elements.

\sidecomment{This is difficult to state before \Cref{def:atomic}; essentially $C\consolidatesto D\consolidatesto C$ implies that $N(C)$ and $N(D)$ are the same.}
With the definition of $\consolidatesto$ in hand, our next goal is to show that this order is, if the ambient poset is vacillating, a partial order on infinite saturated chains (up to local changes). 
In fact, we first motivate why vacillation might be relevant: in \Cref{subsec:consolidation-non-vacillating}, we give an example of a non-vacillating poset in which the consolidation order is not a partial order, as it fails to be antisymmetric.
Then, in \Cref{subsec:consildation-vacillating}, we prove that, when $P$ is vacillating, $\consolidatesto$ is indeed a partial order.

Finally, in \Cref{subsec:consolidated-chain-exists}, we prove that we can always find a consolidated chain in a quasifounded poset.


\subsection{Definition of the consolidation order}
\label{subsec:define-consolidation}

Before giving the definition of consolidation, we need to define a particular subset of $H(P)$, consisting of those nonprincipal elements which are in some sense minimal.
For example, in the poset $\ww^2$, there are many infinite increasing chains, but the poset can be broken down into copies of $\ww$, the simplest saturated infinite increasing chain.
This can be formalised into the following definition.
\sidecomment{As we'll be working almost exclusively with vacillating posets, it won't hurt to think of atomic chains as saturated copies of $\ww$ and $\ww^*$. See \Cref{subsec:consolidation-non-vacillating} for where this goes wrong in the non-vacillating case.}

\begin{definition}
    \label{def:atomic}
    An element of $N_+(P)$ is said to be an \emph{atomic} increasing chain if it does not admit an order-preserving injection from the poset $\ww+1$.
    We write $\atomicincr(P)$ for the set of elements of $N(P)$ with atomic increasing chains as representatives.
    Similarly, an element of $N_-(P)$ is \emph{atomic} if there is a representative chain which does not admit an order-preserving injection from the poset $(\ww+1)^*$, and we write $\atomicdecr(P)$ for the set of such chains.

	The \emph{atomic profile} of a saturated chain $C\sseq P$ is the set $\atomic(C)$.
\end{definition}

Note that the notion of an atomic chain is well-defined in that, if chains $C,D\sseq P$ have $C\simeq_+ D$ and $C$ contains no copy of $\ww+1$, then there is a final segment $D'$ of $D$ which contains no copy of $\ww + 1$; this follows easily from \Cref{lem:mfi-canonical-isomorphism}.
\sidecomment{Alternatively, look at chain $C_0$ in \Cref{fig:example-2}. That is $\zz\lextimes\ww$, but it serves the same purpose as the example described here.}

Note that in the poset $\ww\lextimes\ww^*$ consisting of an infinite increasing sequence of copies of $\ww^*$, the only saturated infinite increasing sub-chains are cofinal and contain a copy of $\ww\lextimes\ww^*$ -- these are the atomic increasing chains in this case. 
However, we may note that $\ww\lextimes\ww^*$ is not vacillating.
Indeed, in vacillating posets, the only atomic increasing chains have final segments isomorphic to $\ww$.

\sidecomment{With the right mental picture of vacillation and atomic chains, the fact that \Cref{lem:atomic-is-omega} holds should be very clear. In fact, this lemma was only added to the manuscript at the last minute (and was just asserted inline beforehand).}
\begin{lemma}
	\label{lem:atomic-is-omega}
	Let $P$ be a scattered vacillating poset and let $X\in\atomicincr(P)$ be an atomic increasing element.
	Then $X$ has a witness chain $C$ which is order-isomorphic to $\ww$.
\end{lemma}

\begin{proof}
	Let $C$ be a witness chain of $X$ admitting no embedding of $\ww + 1$.
	We will show that $C$ has a final segment order isomorphic to $\ww$; assume for contradiction that it does not.

	Note that $\atomicincr(C) = N_+(C) =\set{X}$ as otherwise we could embed $\ww+1$ into $C$.
	Note further that, if $N_-(C) = \emptyset$, then $C$ is well-ordered, and so its order type is an infinite ordinal which does not admit an embedding of $\ww+1$, i.e.\ $C\cong\ww$, as required.

	Knowing that $N_-(C)\neq \emptyset$, we split into two cases based on the following: is $N_-(C)$ co-well-ordered?

	If $N_-(C)$ is co-well-ordered, then it has a maximal element $Y$. 
	But then $C\inter \above{Y}_C$ has no internal limits, and so is a sub-order of $\zeta$.
	The fact that $C$ has no maximal element then implies that $C$ has a final segment order-isomorphic to $\ww$, as required.

	\sidecomment{See \Cref{fig:example-2} for a counterexample without vacillation. (Strictly, a counterexample to the order-reversed version of this.)}
	If $N_-(C)$ is not co-well-ordered, then it has an infinite increasing sequence $Y_0 < Y_1 < Y_2 < \dotsb$.
	But as $N_+(C) = \set{X}$, the chain $D \defined \Conv_C(\set{Y_n \st n < \ww})$ is an infinite increasing sequence of decreasing chains, witnessing the failure of vacillation in $P$, a contradiction.

	Thus $C$ must have a final segment $C' \cong \ww$.
	We find that $C'$ is thus a witness of $X$ with the properties posited by \Cref{lem:atomic-is-omega}.
\end{proof}

Now that we know that $\atomicincr(P)$ and $\atomicdecr(P)$ are well-defined, we would like to prove that these objects are non-empty for non-trivial $P$.
To do this, we will perform an induction on scattered chains by appealing to the following classical result of Hausdorff \cite{hausdorff1908classification}.

\sidecomment{Sadly, the version of \Cref{thm:hausdorff-classification} for scattered posets (as opposed to chains) is not at all helpful for proving anything here, so we have to reduce to chains.}
\begin{theorem}[\cite{hausdorff1908classification}]
    \label{thm:hausdorff-classification}
    Let $\cB$ be the class of well-orderings and reverse well-orderings. 
    The class of scattered linear orderings is the least class which contains $\cB$, and is closed under lexicographic sums with index set lying in $\cB$.
\end{theorem}

We are now ready to prove that non-trivial scattered posets always have non-trivial atomic profiles.

\sidecomment{This is a lemma which seems obvious, and it's quite easy to overlook the fact that this is something that you might even need to prove.}
\begin{lemma}
    \label{lem:exists-atomic}
    If $P$ is a scattered poset containing an infinite increasing sequence, then $\atomicincr(P)$ is non-empty.
\end{lemma}

\begin{proof}
    Note first that a scattered poset $P$ contains an infinite increasing sequence if and only if it contains some saturated chain with this property.
    It thus suffices to prove that $\atomicincr(C)\neq \emptyset$ for any scattered chain $C$ which admits an embedding of $\ww$.

    We induct on chains by means of \Cref{thm:hausdorff-classification}; let $\cB$ be as defined therein.
    For the base case, we must prove the result for chains $C$ which are well-ordered or co-well-ordered.
    If $C$ is well-ordered and infinite, then we can simply take the chain consisting of the bottom $\ww$ elements of $C$, which clearly admits no embedding from $\ww + 1$.
    If $C$ is co-well-ordered, then by definition $C$ does not admit an embedding of $\ww$, and so the condition holds vacuously.

	\sidecomment{These cases are all easy, there's just a few of them.}
    For the inductive step, we take $P = \bigoplus_{i\in I} C_i$, where $I\in \cB$ and each $C_i$ either has $\atomicincr(C_i)\neq\emptyset$ or does not admit an embedding of $\ww$.
    If any $C_i$ has $\atomicincr(C_i)\neq\emptyset$, then we can find these elements in $\atomicincr(P)$, which is thus non-empty.
    We may therefore assume that no $C_i$ admits an embedding of $\ww$, i.e.\ $C_i$ are all co-well-ordered.

    If $I$ is also co-well-ordered, then so too is $P$, and we are (vacuously) done.
    Thus assume that $I$ is well-ordered and infinite.
    We claim that $X\defined \bigcup_{j\in J} C_j$ is an element of $\atomicincr(P)$, where $J$ consists of the bottom $\ww$ elements of $I$.
    Indeed, this poset admits an embedding of $\ww$: take one element from each $C_j$.
    Furthermore, this poset admits no embedding of $\ww+1$; if it did, then some $C_j$ would have to admit an embedding of $\ww$, a contradiction.
    Thus, in all cases, either $P$ is co-well-ordered or $\atomicincr(P)$ is non-empty, as required.
\end{proof}

We are now ready to give the definition of our new ordering.

\sidecomment{It's probably helpful to make sure you fully understand the definition for atomic chains (thinking of them as $\ww$ and $\ww^*$) before moving on to chains.}
\begin{definition}
    \label{def:cofinal-comparability}
	For elements $X,Y\in \atomicincr(P)$, we say that $X$ \emph{consolidates} to $Y$ and write $X \consolidatesto Y$ if
	\begin{equation*}
		\forall\, y \in \below{Y}_P, \; \exists\, x \in\below{X}_P \; \text{ such that } x > y.
	\end{equation*}
	Similarly, for $X, Y \in \atomicdecr(P)$, we write $X \consolidatesto Y$ if
	\begin{equation*}
		\forall\, y \in \above{Y}_P, \; \exists\, x \in\above{X}_P \; \text{ such that } x < y.
	\end{equation*}
    For saturated chains $C,D\sseq P$, we say that $C$ \emph{consolidates} to $D$, or that $D$ is a \emph{consolidation} of $C$, and write $C\consolidatesto D$, if there is an injective order-preserving map $\consolidationwitness{f}{C}{D}$ sending $\atomicincr(C)$ to $\atomicincr(D)$ and $\atomicdecr(C)$ to $\atomicdecr(D)$ (i.e.\ $f$ \emph{preserves signs}) such that, for all $X\in\atomic(C)$, we have $X\consolidatesto f(X)$.
    We say that $f$ \emph{witnesses} the fact that $C\consolidatesto D$.
\end{definition}

\begin{rambling}
	The definition of consolidation for chains is intended to be a ``natural'' extension of the notion for chains of order type $\ww$ and $\ww^*$.
	However, this definition changed upwards of ten times in the process of developing the proof, so the details are delicate.
	For example, $f$ is only defined on atomic elements; for a long time, I was trying to make a (more complex) definition work which extended $f$ to non-atomics as well, though eventually it turned out that just looking at atomics seems to be the right way forward.
\end{rambling}

The following remark highlights a potentially unintuitive point about \Cref{def:cofinal-comparability}.
\sidecomment{This does look strange, but I really think it is the right way of doing things, as otherwise you'd have to treat increasing and decreasing elements separately. Perhaps ``moving away from the limit points'' is a good way of thinking about $\consolidatesto$.}

\begin{remark}
	\label{rem:consolidation-ordering}
    Note that, for elements $X,Y\in \atomicincr(P)$, if $X\consolidatesto Y$ and $C,D$ are witness chains of $X,Y$ both with order-type $\ww$, then for every element of $C$, except perhaps a finite initial segment, there is a smaller element of $D$.
	In particular, if $C > D$, then $X \consolidatesto Y$, as opposed to the case when $X,Y\in\atomicdecr(P)$, where $C < D$ implies $C\consolidatesto D$.

    This convention on the ordering allows for more natural statements, and statements which still hold after reversing the order of $P$.
    An example of such a statement is ``if $\eps\in\set{-,+}$ and $X,Y\in \atomiceps(P)$ are domination-incomparable non-principal elements with $X\consolidatesto Y$ and $Y\centernot{\consolidatesto} X$, then for any witness chains $C,D$ there is an $x\in C$ incomparable to infinitely many elements of $D$.''
\end{remark}

We will show that $\consolidatesto$ is a preorder; note that it is not a partial order as we can have two different maximal chains $C,D\sseq P$ with $N(C) = N(D)$, and so the trivial function between $\atomic(C)$ and $\atomic(D)$ witnesses $C\consolidatesto D$ and $D\consolidatesto C$.

\begin{lemma}
	\label{lem:atomic-consolidation-partial}
	If $P$ is a scattered vacillating poset, then consolidation is a partial order on $\atomic(P)$.
\end{lemma}

\begin{proof}
	As reflexivity is clear, it suffices to prove transitivity and antisymmetry of $\consolidatesto$.
	By symmetry, it suffices to consider $\atomicincr(P)$.
	
	\sidecomment{Transitivity is easy: just rearrange the definitions.}
	We first deal with transitivity. 
	If $X,Y,Z \in \atomicincr(P)$ have $X \consolidatesto Y \consolidatesto Z$, then we know that
	\begin{align*}
		\forall\, z \in \below{Z}_P, \; \exists\, y \in\below{Y}_P \; &\text{ such that } z < y. \\
		\forall\, y \in \below{Y}_P, \; \exists\, x \in\below{X}_P \; &\text{ such that } y < x.
	\end{align*}
	But it is immediate from combining the above that, for all $z\in \below{Z}_P$ there is $x\in\below{X}_P$ such that $z < x$, as required.

	\sidecomment{Note that consolidation fails to be a partial order on $\atomic(P)$ if we drop the assumption of vacillation: take $\bigoplus_{n<\omega}(\omega^*\sqcup\set{x_n})$ -- the two natural top elements consolidate to each other.}
	For antisymmetry, assume that $X,Y\in\atomicincr(P)$ have $X\consolidatesto Y\consolidatesto X$. 
	By \Cref{lem:atomic-is-omega}, we can find saturated chains $C,D\sseq P$, both of order type $\omega$, in the equivalence classes given by $X$ and $Y$.
	Following \Cref{def:equivalence,def:increasing-and-decreasing}, we must show that $C$ and $D$ have \mfi.
	Take $x\in C$. 
	We know from $Y\consolidatesto X$ that there is $y\in\below{Y}_P$ such that $y > x$.

	\Cref{lem:domination-independent} tells us that $y < Y$ implies that $y < D'$ for some final segment $D'\sseq D$, as domination is independent of the choice of representative chain.
	In particular, there is $z\in D$ such that $y < z$.
	Thus $x < z$ and so $x$ is comparable with a final segment of $D$, and thus is incomparable with only finitely many elements of $D$.

	\sidecomment{This last part is to deal with the ``$\incompset{x}$ non-empty'' part of the condition.}
	If we take $x\in C$ and $y\in D$ with $y > x$, we know that there must be $z\in C$ with $z \incomp y$, as $\incompset{y}_C$ is non-empty due to saturation and $x < y < X$.
	Let $C' = \aboveeq{z}_C$ and $D' = \aboveeq{y}_D$.
	It follows that $C'$ and $D'$ have \mfi, as required.
\end{proof}

\begin{lemma}
    \label{lem:chain-consolidation-transitive}
    If $P$ is a scattered vacillating poset, then the order $\consolidatesto$ is a preorder on the set of saturated chains of $P$.
	In particular, if $C,D,E\sseq P$ are saturated chains and $f$ witnesses $C\consolidatesto D$ and $g$ witnesses $D\consolidatesto E$, then $g\circ f$ witnesses $C\consolidatesto E$.
\end{lemma}

\sidecomment{The actual content here is that $\consolidatesto$ is transitive on saturated chains. It pretty much follows straight from the definition.}
\begin{proof}
	First note that $C\consolidatesto C$ is witnessed by the identity function on $\atomic(C)$, and so it remains to show that $\consolidatesto$ is transitive on saturated chains.
    Take saturated chains $C$, $D$, and $E$ in the poset $P$ such that $C\consolidatesto D$ and $D\consolidatesto E$.
    We will show that $C\consolidatesto E$.

	Let $\consolidationwitness{f}{C}{D}$ and $\consolidationwitness{g}{D}{E}$ witness these two consolidations.
	We claim that the function $\consolidationwitness{g\circ f}{C}{E}$ witnesses $C \consolidatesto E$.
	It is immediate that $g\circ f$ is injective and preserves signs and order, and \Cref{lem:atomic-consolidation-partial} tells us that $X\consolidatesto g\circ f(X)$, as required.
\end{proof}

We have already seen that $\consolidatesto$ is not a partial order. 
However, as observed earlier, the immediate obstruction to this seems trivial; cases wherein $\atomic(C) = \atomic(D)$.
For this reason, one may hope that $\consolidatesto$ is a partial order when considered only on the set of atomic profiles of saturated chains (i.e.\ quotienting the set of saturated chains out by the equivalence relation $\atomic(C) = \atomic(D)$).
However, in the next section we prove that this is not the case.

In the next section, we present two counterexamples to two statements which may at first sound as though they may be true.


\subsection{Two instructive counterexamples}
\label{subsec:consolidation-non-vacillating}

While this section is not strictly necessary for building our proofs, we believe that it is nevertheless instructive in how the objects we consider do and do not behave, and can help build understanding.

We now give an example of a poset $P$ for which consolidation is not a partial order on the set of atomic profiles.

\begin{proposition}
    \label{ex:no-finite-cut}
    Let $P$ be the poset on $\ZZ\times 2\times\NN$ with ordering $\leq$ defined to be the minimal partial order which contains all of the following for all $z\in\ZZ$, $i\in 2$, and $n\in \NN$.
    \begin{itemize}
        \item $(z,i,n) \leq (z,i,n+1)$.
        \item $(z,i,n) \leq (z+1,i,0)$.
        \item $(z,0,n) \leq (z,1,n) \leq (z+1,0,n)$.
    \end{itemize}
	Then $\consolidatesto$ is not a partial order on the set of maximal chains of $P$.
\end{proposition}

\begin{rambling}
	One point to note for the poset $P$ in \Cref{ex:no-finite-cut} is that, despite having largest antichain of size 2, $P$ has no finite \emph{cut} -- a set of points which every maximal chain must intersect: any point of $P$ sits naturally inside a copy of $\ww$, and so has an integer height in such a copy. 
	Thus any finite set $S$ of points in $P$ has a maximum height $h$. 
	Take a maximal chain consisting of the point at height $h+1$ in every copy of $\ww$; this certainly misses $S$.

	There was some time before finding this example where I believed that finite cuts might be the way forward in the proof.
\end{rambling}

A Hasse diagram of the poset defined in \Cref{ex:no-finite-cut} is shown in \Cref{fig:example-2}.

\begin{figure}[htbp]
    \centering
    \begin{tikzpicture}[scale=1]
    \pgfmathsetmacro\baseoff{0.5}
    \pgfmathsetmacro\width{0.75}
    \pgfmathsetmacro\toppos{3}
    \pgfmathsetmacro\yoff{0.1}
    \pgfmathsetmacro\pointtop{3}
    \definecolor{grey}{rgb}{0.7,0.7,0.7}

    \foreach \xbase/\ybase [count = \i from 1] in {0/0, 3/1.5, 0/3, 3/4.5, 0/6} {
        \coordinate (l1\i) at (\xbase, \ybase + \toppos - \yoff);
        \coordinate (l2\i) at (\xbase - \width, \ybase - \yoff);
        \coordinate (l3\i) at (\xbase + \width, \ybase - \yoff);
        \draw[color=grey,line width = 1mm] (l1\i) -- (l2\i) -- (l3\i) -- (l1\i);

        \foreach \j in {1,...,20} {
            \pgfmathsetmacro\propn{1 - 4/(\j + 3)}
            \node (v\i\j) at (\xbase,\ybase + \propn * \pointtop) {$\bullet$};
        }
    }

    \foreach \j in {1,...,20} {
        \pgfmathsetmacro\propn{1 - 4/(\j + 3)}
        \coordinate (v0\j) at (3, -1.5 + \propn * \pointtop);
        \coordinate (v6\j) at (3, 7.5 + \propn * \pointtop);
        \draw (v0\j) -- (v1\j) -- (v2\j) -- (v3\j) -- (v4\j) -- (v5\j) -- (v6\j);
    }
    \draw (0,-0.5) -- (0,9.5);
    \draw (3,0) -- (3,9);

    \draw [dotted, thick] (0,-1.5) -- (0,10.5);
    \draw [dotted, thick] (3,-1) -- (3,10);

	\node (X0) at (-1.2, 1.5) {$X_{-1}$};
	\node (X1) at (-1.2, 4.5) {$X_{0}$};
	\node (X2) at (-1.2, 7.5) {$X_{1}$}; 
	\node (Y0) at (4.2, 3) {$Y_{-1}$};
	\node (Y1) at (4.2, 6) {$Y_{0}$};
	\node (C0) at (0, 11) {$C_0$};
	\node (C1) at (3, 11) {$C_1$};
	\node (top0) at (-0.7, 10) {$\top_0$};
	\node (top1) at (3.7, 10) {$\top_1$};
	\node (bot0) at (-0.7, -1) {$\bot_0$};
	\node (bot1) at (3.7, -1) {$\bot_1$};
\end{tikzpicture}
    \caption{A Hasse diagram of the poset defined in \Cref{ex:no-finite-cut}. The regions in grey triangles are order-isomorphic to $\ww$, and this structure continues infinitely in both directions.
	The structures are labelled to match the notation used in \Cref{ex:no-finite-cut} and its proof.}
    \label{fig:example-2}
\end{figure}

\begin{proof}[Proof of \Cref{ex:no-finite-cut}]
	We prove that, on the poset $P$ as defined in the proposition, the consolidation order $\consolidatesto$ fails to be antisymmetric.

	\sidecomment{This is essentially all just book-keeping; the chains $C_0$ and $C_1$ admit non-trivial consolidations to each other. A consequence of this is that there is a non-trivial witness of a consolidation from $C_0$ to itself (and likewise for $C_1$): just send every copy of $\ww$ to the copy one step down the chain.}
	Define the chains $C_0$ and $C_1$ by $C_i\defined\set{(z,i,n)\st z\in \ZZ, n\in\NN}$, and note that these are both maximal.
	We claim that $\bot_0\in\atomic(C_0)$ and $\bot_1\in\atomic(C_1)$ have $\bot_0 \consolidatesto \bot_1$ and $\bot_1 \consolidatesto \bot_0$ but $\bot_0 \neq \bot_1$, which demonstrates that $\consolidatesto$ is not antisymmetric on $N(P)$.

	Both $N(C_0)$ and $N(C_1)$ have order type $\ZZP = \ZZ\union\set{\bot,\top}$ where $\bot < \ZZ < \top$.
	For $z\in \ZZ$, let $X_z\in N(C_0)$ be (represented by) the chain $\set{(z, 0, n) \st n\in \NN}$, and $Y_z\in N(C_1)$ be $\set{(z,1,n) \st n\in \NN}$.
	Write the top and bottom elements of $N(C_i)$ as $\top_i$ and $\bot_i$ respectively.
	See \Cref{fig:example-2} for a diagram.

	Then we may define $\consolidationwitness{f}{C_0}{C_1}$ by sending $\bot_0 \mapsto \bot_1$ and $X_{z+1}$ to $Y_z$ for all $z\in\ZZ$.
	Similarly, $\consolidationwitness{g}{C_1}{C_0}$ maps the corresponding bottom elements to each other and sends $Y_z$ to $X_z$ for all $z\in \ZZ$.
	We may note that $X_{z+1} \consolidatesto Y_z \consolidatesto X_z$, and both $f$ and $g$ are order-preserving and injective.
	Thus $f$ and $g$ do indeed witness $C_0 \consolidatesto C_1$ and $C_1 \consolidatesto C_0$, as required.

	However, every element of $C_0$ has infinite incomparability in $C_1$: if $x\in X_{z+1}$, then $x$ is incomparable with a final segment of $Y_z$. 
	Similarly, every element of $C_1$ has infinite incomparability in $C_0$.
	Thus the atomic non-principal elements coming from $X_z$ and $Y_z$ are all distinct, and so $\atomic(C_0) \neq \atomic(C_1)$, and the result is proved.
\end{proof}

Next, we present an example of a well-founded poset $P$ for which $\atomicincr(P)$ contains an infinite antichain.
In fact, this antichain consists entirely of minimal elements.

\sidecomment{Another useful example to keep in mind. This one's a bit annoying to draw a diagram for, though, so we don't. If you only care about domination incomparability, it turns out that you can even get an uncountable set of incomparable elements coming from a poset looking a bit like $\ww\times\ww$.}
\begin{proposition}
	Let $P$ be the poset on $\NN\times\NN$ with ordering $\leq$ defined to be the minimal relation which is a partial order containing all of the following comparisons for $a,b\in\NN$.
	\begin{itemize}
		\item $(a,b) < (a,b+1)$.
		\item $(a,b) < (a + b + x, b)$ for all $x\geq 1$.
	\end{itemize}
	Then $P$ is an FAC poset and $\atomicincr(P)$ contains infinitely many minimal elements which are pairwise consolidation-incomparable.
\end{proposition}

\begin{proof}
	Note first that if $(a,b),(c,d)\in P$ have $(a,b) \leq (c,d)$ then $a\leq c$ and $b\leq d$.
	We first prove that $P$ is an FAC poset.

	Indeed, let $A\sseq P$ be an antichain, and let $b\in\NN$ be minimal such that there is a point $(a,b)\in A$.
	We then know from the definition of $\leq$ that $A$ cannot contain any point $(c,d)$ with $c > a+b$ and $d\geq b$, as we would have $(a,b) \leq (c,d)$.
	But this implies that, if $(x,y) \in A$, then $x \leq a+b$.
	There cannot be two distinct points $(x_1,y_1),(x_2,y_2)\in A$ with $x_1 = x_2$, which implies that $\card{A}\leq a+b+1$.
	In particular, $A$ is finite, as required.

	Define the chain $C_n$ to be
	\begin{equation*}
		C_n \defined \set{(n,x) \st x\in\NN}.
	\end{equation*}
	We claim that the elements $[C_n]_+ \in \atomicincr(P)$ are pairwise distinct and consolidation-incomparable.

	Fix $n < m$. We then have
	\begin{equation*}
		(n,x)\incomp (m,y) \quad\text{ for all }x,y > m - n.
	\end{equation*}
	Thus $C_n$ and $C_m$ have entirely incomparable tails, and so are indeed distinct and consolidation incomparable.

	Noting that every infinite chain in $P$ is order-isomorphic to $\ww$, we see that elements $[C_n]_+ \in \atomicincr(P)$ are all minimal, as required.
\end{proof}


\subsection{Consolidation is a partial order on atomic profiles in vacillating posets}
\label{subsec:consildation-vacillating}

Now that we know that there are non-vacillating posets on which consolidation is not a partial order on the set of atomic profiles, we work towards showing that it is in fact a partial order whenever $P$ is vacillating.
In particular, the goal of the rest of the section is to prove the following.

\sidecomment{Another reason why vacillation is an important property here.}
\begin{proposition}
	\label{prop:consolidation-partial}
	If $P$ is a vacillating scattered poset, then consolidation is a partial order on the set of atomic profiles of saturated chains of $P$.
\end{proposition}

The key to the proof of \Cref{prop:consolidation-partial} is the following lemma.

\begin{lemma}
	\label{lem:self-consolidation-trivial}
	If $P$ is a vacillating scattered poset and $\consolidationwitness{f}{C}{C}$ witnesses $C\consolidatesto C$, then $f$ is the identity.
\end{lemma}

\sidecomment{This proof is very direct: $f(X) \neq X$ hands us a way of constructing a sequence which contradicts vacillation.}
\begin{proof}
	Assume for contradiction that $f$ is not the identity.
	There is thus some $X\in\atomicincr(C)$ such that $f(X)\neq X$ (the case $X\in\atomicdecr(C)$ is entirely similar).
	Define $X_0 = X$ and, for all $n\geq 0$, the element $X_{n+1} \defined f(X_n) \in \atomicincr(C)$.
	We claim that $X_{n+1} < X_n$ for all $n$.

	We know from the fact that $f$ witnesses a consolidation that $X_n \consolidatesto X_{n+1}$, and so as these elements belong to the same chain, we must have $X_n \geq X_{n+1}$.
	\sidecomment{The proof of \Cref{lem:self-consolidation-trivial} in some sense deals directly with the obstruction demonstrated by \Cref{ex:no-finite-cut}: a non-trivial consolidation $C\to C$ must have a decreasing sequence resembling that in the example.}
	Consider the minimal $n$ for which $X_{n+1} = X_n$, if it occurs.
	Then (noting $n\geq 1$ by assumption) we have $f(X_{n-1}) = X_n = X_{n+1} = f(X_n)$, contradicting the injectivity of $f$, as $X_{n-1} \neq X_n$.

	Therefore $X_0 > X_1 > X_2 > \dotsb$.
	We claim that this sequence contradicts vacillation of $P$: it suffices to show that 
	\begin{equation}
		\label{eq:no-decreasing-in-consolidation}
		\Conv_{H(C)}\p[\big]{\set{X_n \st n \geq 0}} \inter \atomicdecr(C) = \emptyset,
	\end{equation}
	as then the saturated chain $\Conv_{H(C)}\p[\big]{\set{X_n \st n \geq 0}} \inter C$ has no internal decreasing non-principal elements and contains an infinite decreasing sequence of elements of $\atomicincr(C)$.

	\sidecomment{Any decreasing element would have to get sent ``the wrong way'' by $f$, contradicting order preservation.}
	Recall from \Cref{lem:h-chain-is-chain} that $H(C)$ is a chain, and assume for contradiction that there is $Y\in\atomicdecr(C)$ with $X_n > Y > X_{n+1}$ for some $n$.
	We know that $f(X_n) = X_{n+1}$ and $f(Y) \geq Y$ as $f$ sends $\atomic(C)$ to itself.
	We also know that $f$ is order-preserving, and so $f(X_n) > f(Y)$.
	But then $f(Y) \geq Y > X_{n+1} = f(X_n) > f(Y)$, a contradiction.

	Thus \eqref{eq:no-decreasing-in-consolidation} holds, contradicting the vacillation of $P$, as required.
	Therefore $f$ must indeed be the identity, proving \Cref{lem:self-consolidation-trivial}.
\end{proof}

With \Cref{lem:self-consolidation-trivial} in hand, we can now prove \Cref{prop:consolidation-partial}.

\begin{proof}[Proof of \Cref{prop:consolidation-partial}]
	We know from \Cref{lem:chain-consolidation-transitive} that consolidation is a preorder on the set of atomic profiles of saturated chains, so it remains to show that, under the hypothesis of vacillation, $\consolidatesto$ is antisymmetric.
	Assume for contradiction that there are chains $C,D\sseq P$ with $\atomic(C) \neq \atomic(D)$ and functions $\consolidationwitness{f}{C}{D}$ and $\consolidationwitness{g}{D}{C}$ witnessing $C\consolidatesto D \consolidatesto C$.

	\sidecomment{The key step here is to reduce from considering chains to just elements of $\atomicincr(P)$, which we know is antisymmetric under the ordering $\consolidatesto$.}
	As $\atomic(C) \neq \atomic(D)$, we may assume without loss of generality that there is $X\in\atomicincr(C)$ with $f(X) \neq X$.
	Let $Y = g(f(X)) \in \atomicincr(C)$.
	We claim that $X\neq Y$.
	Indeed, \Cref{lem:atomic-consolidation-partial} tells us that consolidation is a partial order on $\atomicincr(P)$ and so $X \consolidatesto f(X) \consolidatesto Y$ and $X = Y$ would imply that $f(X) = X$, a contradiction.

	But we know from \Cref{lem:chain-consolidation-transitive} that $g\circ f$ witnesses a consolidation $C\consolidatesto C$, and the above tells us that $g\circ f$ is not the identity, contradicting \Cref{lem:self-consolidation-trivial}, which applies as we have assumed that $P$ is vacillating.
	This proves \Cref{prop:consolidation-partial}, as required.
\end{proof}


\subsection{Producing a consolidated chain}
\label{subsec:consolidated-chain-exists}

In this section, we will primarily be interested in the case when $P$ is a countable scattered quasifounded FAC poset.
The goal of this section is to show that such a poset must contain a chain which admits no further consolidation.
For convenience, we introduce the following terminology.

\begin{definition}
	\label{def:consolidated}
	Call a consolidation $C\consolidatesto D$ \emph{trivial} if $\atomic(C) = \atomic(D)$, and say that a saturated chain $C$ is \emph{consolidated} if it admits no non-trivial consolidation.
\end{definition}

Recall from \Cref{prop:consolidation-partial} that $\consolidatesto$ is transitive on the set of atomic profiles of a scattered vacillating poset $P$, which includes all countable scattered quasifounded FAC posets.
Our first goal is to prove that, on such a poset, we in fact have that consolidation is a well-founded partial order on $\atomiceps(P)$.
We prove this in two steps, starting with the following.

\sidecomment{A slightly surprising thing about this lemma is that it's false if we instead say that $X_i\consolidationof X_{i+1}$, and $X_i\centernot{\consolidatesto} X_{i+1}$.}
\begin{lemma}
    \label{lem:domination}
    Let $P$ be a vacillating scattered poset, and $X_1,X_2,\dotsc$ be an infinite collection of pairwise domination-incomparable elements of $\atomicincr(P)$ such that, for all $i$, we have $X_i\consolidatesto X_{i+1}$, and $X_i\centernot{\consolidationof} X_{i+1}$.
    Then, given any family of representative chains $C_i\in X_i$, there is an infinite antichain $A\sseq P$ which intersects every $C_i$ exactly once.
\end{lemma}

\sidecomment{This lemma was in fact one of the first things I proved about the consolidation order (before I even had the definition pinned down at all), and was one of the first hints that it could be useful.}

\begin{proof}
    We firstly claim that, for each $i$, there is an $x_i\in C_i$ such that $x_i$ is not below any element of any $C_j$ for $j > i$.
    Assume for contradiction that there was no such $x$ in some $C_n$; let $y_1,y_2,\dotsc$ be a cofinal increasing sequence in $C_n$.
    By assumption, each $y_i$ is below some $z_i \in C_r$ for some $r > n$.
    But then, as $X_j\consolidatesto X_{j+1}$ for all $j$ and $\consolidatesto$ is transitive due to \Cref{lem:atomic-consolidation-partial}, we find that $X_{n+1}\consolidatesto X_r$, and so $y_i<z_i<w_i$ for some $w_i\in C_{n+1}$.
    As the chains are atomic, this implies that $X_n\consolidationof X_{n+1}$, a contradiction.
    Thus there are indeed elements $x_i\in C_i$ as claimed.
    Let $D_i\sseq C_i$ be defined as $D_i = (\geq x_i)\inter C_i$ for all $i$.

    We now find our infinite antichain inductively.
    Let $a_1\in D_1$ be arbitrary.
    Assume that we have picked $a_i\in D_i$ for all $i \leq n$.
    No element of $D_{n+1}$ is above any $a_i$ by definition of the sets $D_j$, and so it suffices to find a point of $D_{n+1}$ which is not below any of the $a_i$.
    $D_{n+1}$ is not cofinally below any $a_i$, as $D_i$ does not dominate $D_{n+1}$, and so there is a final segment $S_i\sseq D_{n+1}$ which is incomparable to $a_i$.
    Then $S\defined S_1\inter\dotsc\inter S_n$ is a non-empty final segment of $D_{n+1}$, and so we may pick an arbitrary point $a_{n+1}\in S$ to continue construction of our antichain.
    This concludes the induction, and hence we have a set $A=\set{a_1,a_2,\dotsc}$ such that, for all $n$, $\set{a_1,a_2,\dotsc,a_n}$ is an antichain, i.e.\ $A$ is an infinite antichain, as required.
\end{proof}

The key property of \Cref{lem:domination} that we will use is the following corollary.
\sidecomment{This property is indeed key: once the orders are well-founded, life is much easier!}

\begin{corollary}
	\label{cor:consolidation-well-founded}
	If $P$ is a scattered quasifounded FAC poset and $\eps\in\set{-,+}$ then $\atomiceps(P)$ is a well-founded poset under the relation $\consolidationof$.
\end{corollary}

\begin{proof}
	Recall from \Cref{prop:consolidation-partial} that, on vacillating scattered $P$, the consolidation order is a partial order. 
	All it remains to prove, then, is that $\consolidationof$ is a well-founded partial order on $\atomicincr(P)$; the case of $\atomicdecr(P)$ is entirely similar (and note from \Cref{rem:consolidation-ordering} that $\atomicdecr(P)$ is well-founded, not co-well-founded).

	\sidecomment{The main idea here is to use \Cref{lem:domination} to pass from $\consolidatesto$ to $\geq$, where we already have well-foundedness.}
	Assume for contradiction that $\atomicincr(P)$ were not well-founded under $\consolidationof$, i.e.\ there is an infinite sequence $X_1 \consolidatesto X_2 \consolidatesto \dotsb$ with all $X_i$ being distinct elements of $\atomicincr(P)$.
	As $P$ has no infinite antichain, we know from \Cref{lem:domination} that there must be $i < j$ such that $X_i > X_j$. 
	In fact, as $\consolidatesto$ is a partial order, every infinite sub-sequence of $X_1,X_2,\dotsc$ must admit some domination relation.
	It follows from standard Ramsey theory that there must be an infinite chain of dominations $X_{i_1} > X_{i_2} > \dotsb$.
	But this contradicts the fact that $P$ is quasifounded, as this requires that $\atomicincr(P)$ is well-founded under domination.
	Thus $\consolidationof$ is indeed a well-founded partial order on $\atomiceps(P)$.
\end{proof}

With \Cref{cor:consolidation-well-founded} in hand, we can now find a consolidated chain.

\sidecomment{The intuition behind the proposition is clear, but actually getting it to work formally is quite the headache. Technically quasifounded implies scattered so we could drop that assumption, but I prefer thinking about it this way.}
\begin{proposition}
	\label{prop:consolidated-exists}
	Let $P$ be a scattered quasifounded FAC poset. Then $P$ contains a maximal consolidated chain.
\end{proposition}

\begin{commentfigure}[htbp]
    \centering
    \begin{tikzpicture}[
  font=\normalsize,line width=.7pt,line cap=round,line join=round,
  every node/.style={inner sep=1.5pt},
  enclosure/.style={rounded corners=12pt},
  strip/.style={rounded corners=4pt},
  continuation/.style={line width=1.3pt,line cap=round,dash pattern=on 0pt off 3pt},
  pics/point/.style={code={\draw (-1.8pt,-1.8pt)--(1.8pt,1.8pt)
                                      (-1.8pt,1.8pt)--(1.8pt,-1.8pt);}},
  pics/witness/.style={code={\draw (-.15,0)--(.15,0)--(0,.72)--cycle;}}
]
  \draw[enclosure] (0,0) rectangle (3.9,5.75);
  \node[anchor=south west] at (0,5.77) {$P$};
  \draw (.27,5.26) .. controls (.27,3.82) and (.95,2.85) .. (1.76,2.85)
                    .. controls (2.57,2.85) and (3.2,4.03) .. (3.2,5.26);
  \draw[continuation] (1.76,5.05)--(1.76,5.44);
  \pic at (.82,3.86) {witness};
  \pic at (1.45,3.28) {witness};
  \pic at (2.20,3.28) {witness};
  \foreach \dy in {0,.14,.28}
    \draw (1.02,{3.86+\dy})--(1.23,{3.64+\dy});
  \node[above right] at (2.20,4.00) {$X_\alpha$};
  \node at (2.9,3.18) {$Q$};
  \draw[strip] (1.48,.32) rectangle (2.04,2.85);
  \foreach \y in {.64,1.18,1.52,1.76,2.46} \pic at (1.76,\y) {point};
  \node[left] at (1.40,1.70) {$D_\alpha$};

  \draw[->] (4.00,3.03)--(4.47,3.03);

  \begin{scope}[xshift=4.58cm]
    \draw[enclosure] (0,0) rectangle (3.90,5.75);
    \node[anchor=south west] at (0,5.77) {$P$};
    \draw[rounded corners=30pt] (.61,1.01) rectangle (3.73,4.33);
    \node at (.31,2.58) {$P_I$};
    \draw (2.05,4.33) .. controls (1.72,4.61) and (1.72,5.05) .. (1.72,5.53);
    \node[right] at (1.80,5.25) {$D\setminus I$};
    \pic at (2.05,4.33) {point};
    \node[anchor=south east] at (3.84,4.31) {$\min(D\mathord{\setminus}I)$};
    \pic at (2.05,3.96) {point};
    \node[right] at (2.13,3.96) {$p$};
    \draw[strip] (1.68,1.28) rectangle (2.42,3.67);
    \foreach \y in {3.51,3.08,2.62,2.29,2.07} \pic at (2.05,\y) {point};
    \draw[continuation] (2.05,1.45)--(2.05,1.84);
    \node[right] at (2.47,3.51) {$\max E$};
    \node[right] at (2.47,2.58) {$E$};
    \draw (2.05,.15)--(2.05,1.01);
    \node[right] at (2.14,.54) {$I$};
  \end{scope}

  \draw[->] (8.58,3.03)--(9.05,3.03);

  \begin{scope}[xshift=9.16cm,yshift=.20cm]
    \draw[enclosure] (0,-.60) rectangle (4.12,5.55);
    \node[anchor=south west] at (0,5.57) {$P_I$};
    \coordinate (top) at (1.64,5.55);
    \pic at (top) {point};
    \node[above right] at (1.70,5.60) {$\min(D\mathord{\setminus}I)$};
    \draw (1.64,5.40) .. controls (1.98,5.02) and (1.37,4.75) .. (1.62,4.33);
    \node[right] at (1.86,4.96) {$C_I^+$};
    \pic at (1.64,4.13) {point};
    \node[right] at (1.72,4.16) {$p$};

    \draw[strip] (1.26,2.14) rectangle (2.23,4.02);
    \node[left] at (1.19,3.06) {$E\setminus J$};
    \node[right] at (2.28,3.51) {$\max E$};
    \foreach \y in {3.40,2.87,2.34} {
      \pic at (1.83,\y) {point};
      \draw (1.72,\y+.34)--(1.80,\y+.14);
    }
    \draw (1.87,1.60) ellipse [x radius=.74,y radius=.54];
    \node[left] at (1.08,1.61) {$Q_J$};
    \draw (1.77,1.88) .. controls (1.75,1.54) and (1.90,1.46) .. (1.87,1.20);
    \node[right] at (1.92,1.59) {$C_I^-$};
    \draw[strip] (1.51,-.39) rectangle (2.23,1.06);
    \pic at (1.87,1.06) {point};
    \node[right] at (2.23,.85) {$Z\!=\!\max J$};
    \node[left] at (1.43,.34) {$J$};
    \foreach \y in {.72,.40,.12,-.16} \pic at (1.87,\y) {point};
  \end{scope}
\end{tikzpicture}
    \caption{Some of the objects constructed in \Cref{prop:consolidated-exists}. Each part of the image zooms in on the previous part. Fragments of $C_I^-$ are displayed inside $E\setminus J$.}
    \label{fig:prop-consolidated-exists}
\end{commentfigure}

\begin{proof}
	\sidecomment{Idea: just hit everything in sight with well-foundedness: pull out the minimal elements and add them onto the chain, and keep doing this until we can't do it any more.}
	We iteratively use the well-foundedness property established in \Cref{cor:consolidation-well-founded} to iteratively produce a chain $D$ of elements of $\atomic(P)$ which will give rise to a consolidated chain.
	We first build chains $D_\alpha\sseq\atomicincr(P)$ for ordinals $\alpha$.
	Firstly, $D_0$ is the empty set.

	Assume that we have constructed $D_\alpha\sseq \atomicincr(P)$ for some ordinal $\alpha$.
	Consider the set 
	\begin{equation*}
		Q \defined \atomicincr\p[\big]{\above{D_\alpha}_{P}}.
	\end{equation*}
	If $Q$ is empty, then stop the process and output the chain $D_\alpha$.
	As $\atomicincr(P)$ is well-founded under consolidation and $Q$ is a sub-poset, $Q$ must have a $\consolidationof$-minimal element $X_\alpha$.
	Let $D_{\alpha + 1} = D_\alpha \union \set{X_\alpha}$.
	\sidecomment{One can use Bonnet's theorem to prove that the ordinals in this proof are all countable, but we don't need to do this.}
	At limit steps, define $D_\lambda \defined \bigcup_{\alpha < \lambda}D_\alpha$.

	Note that this process must terminate at some ordinal stage, and so eventually outputs some chain $D\sseq\atomicincr(P)$.

	For every initial segment $I$ of $D$, let $P_I$ be the poset 
	\begin{equation*}
		P_I\defined \p[\big]{I, D\setminus I}_P.
	\end{equation*}
	\sidecomment{Once we're done with $\atomicincr(P)$, we then just fill in each gap by looking at $\atomicdecr(P)$. Well-foundedness tells us that we do not destroy the existing witnesses.}
	We will now apply a similar argument to above in $\atomicdecr(P_I)$ and thus construct disjoint saturated chains $C_I^-, C_I^+ \sseq P_I$ which together will witness both the top element $\min(D\setminus I)$ of $H(P_I)$ and the chain we build in $\atomicdecr(P_I)$.

	Indeed, running the same process as above in $P_I$ with the ordering reversed (but still choosing a $\consolidationof$-minimal atom, see \Cref{rem:consolidation-ordering}) gives us a chain $E\sseq\atomicdecr(P_I)$ satisfying
	\begin{equation*}
		\forall Y \in E, \quad Y \text{ is maximal in } \below{\above{Y}_{P_I}}_{\atomicdecr(P_I)}.
	\end{equation*}

	We now construct an element $p\in P_I$ to separate $C_I^+$ and $C_I^-$, to ensure that they are distinct.
	If $\max E$ and $\min (D\setminus I)$ both exist, then, noting that $\max E < \min(D\setminus I)$, we may apply \Cref{lem:find-intermediate} to find $p\in P_I$ with
	\begin{equation}
		\label{eq:define-p}
		\max E < p < \min (D\setminus I).
	\end{equation}
	If either of the two extrema in \eqref{eq:define-p} do not exist, then simply choose $p$ to satisfy only the bounds corresponding to extrema which do exist.
	If neither extremum exists, then there is nothing to do and we skip this step.

	If $I = D$ then set $C_I^+ = \emptyset$.
	Otherwise, there is an ordinal $\alpha$ such that $X_\alpha = \min(D\setminus I)$. 
	We know from the construction of $X_\alpha$ that it has a representative chain entirely above $I$, and so we can take $C_I^+$ to be the tail of this chain in $\above{p}_{P_I}$ and this must witness $\min\p[\big]{D \setminus I}$.
	We now construct $C_I^-\sseq \below{p}_{P_I}$ to be a chain witnessing all elements of $E$ simultaneously.
	To construct $C_I^-$, let $J\neq \emptyset$ be an initial segment of $E$, and let
	\begin{equation*}
		Q_J \defined (J, E\setminus J)_{P_I} \inter \below{p}_{P_I}.
	\end{equation*}
	$J$ has a maximal element $Z$, which is the bottom element of $H(Q_J)$, and so we can find a saturated chain in $Q_J$ witnessing $Z$.
	Let $C_I^-$ be the union of these chains over all choices of $J$ and extend to a saturated chain, which witnesses all elements of $E$ simultaneously.

	We can then define the chain $C$ to be a maximal chain extending
	\begin{equation*}
		\bigcup_I \p[\big]{C_I^- \union C_I^+},
	\end{equation*}
	where the union ranges over all initial segments of $D$.

	We claim that $C$ is consolidated; assume for contradiction that there were a non-trivial consolidation $C\consolidatesto C'$ witnessed by a function $\consolidationwitness{f}{C}{C'}$.
	As the consolidation is non-trivial, either $\atomic(C)$ is strictly contained in $\atomic(C')$ or $f$ is not the identity function: assume first that $f=\text{id}$ and $\atomic(C) \ssneq \atomic(C')$.
	In fact, we prove the following.
	
	\begin{claim}
		\label{claim:no-extra-atomics}
		If $F\sseq P$ is a saturated chain with $D\sseq\atomicincr(F)$, then $\atomicincr(F) = D$.
	\end{claim}

	\begin{proof}
		Assume for contradiction that there is $Y \in \atomicincr(F) \setminus D$.
		By well-foundedness, we may moreover assume that $Y$ is minimal in $\atomicincr(F) \setminus D$.		

		We know from the construction of $D$ that there is an ordinal $\alpha$ such that $\below{Y}_D = D_\alpha$.
		The existence of $Y$ tells us that $\atomicincr\p[\big]{\above{D_\alpha}_{P}}$ is non-empty, as we can take a representative of $Y$ in $F$ lying entirely above $D_\alpha$, and so it contains $Y$.
		This means that an element $X_\alpha$ was chosen to be consolidation-minimal in $\atomicincr\p[\big]{\above{D_\alpha}_{P}}$.
		But we know that
		\begin{equation*}
			X_\alpha = \min \p[\big]{D \setminus \below{Y}_D} > Y.
		\end{equation*}
		This tells us that $X_\alpha > Y$ and both $Y$ and $X_\alpha$ are in $\atomicincr\p[\big]{\above{D_\alpha}_{P}}$, and so $X_\alpha$ is not even domination-minimal in this set (let alone consolidation-minimal), a contradiction.
		Thus $\atomicincr(F) = D$, as claimed.
	\end{proof}
	
	A similar result to \Cref{claim:no-extra-atomics} holds for $\atomicdecr(C)$ by applying the dual argument to each $P_I$ and the corresponding chain $E\sseq\atomicdecr(P_I)$.
	This shows that we cannot have $f = \text{id}$ in our hypothesised non-trivial consolidation.

	\sidecomment{This looks a bit technical but it really is just ``assume something moved, it contradicts minimality.''}
	We may thus assume that $f$ is not the identity on $\atomic(C)$, and so, assuming that $f$ moves an element of $\atomicincr(C)$, take $X\in\atomicincr(C)$ minimal such that $f(X) \neq X$ (noting that minimality is valid as $\atomicincr(C)$ is well-ordered by domination).
	Again, a similar argument runs in each $P_I$ for $\atomicdecr(C)$.

	But then, by construction, there is some ordinal $\alpha$ such that $X = X_\alpha$ was chosen to be a minimal element under the consolidation order:
	\begin{equation*}
		X \in \min_{\consolidationof} \atomicincr\p[\big]{\above{D_\alpha}_P}.
	\end{equation*}
	As we chose $X$ to be minimal in $\atomicincr(C) \sseq D$ such that $f(X) \neq X$, we know that, for all $Y \in D_\alpha$, we have $f(X) > f(Y) = Y$, i.e.\ $f(X) \in \atomicincr\p[\big]{\above{D_\alpha}_P}$.
	But $f(X) \consolidationof X$, contradicting the above minimality of $X$.

	This tells us that $C$ is consolidated, as required.
\end{proof}

\section{Finding a maximal tube}
\label{sec:finding-tubes}

We now prove the final piece of our proof of \Cref{thm:main}, which is to deduce the existence of a maximal tube from the existence of a consolidated chain.

\sidecomment{We've done most of the real work by this point: this section is mostly just juggling definitions and some relatively straightforward arguments.}
Given a chain $C$ in a poset $P$, one may attempt to convert $C$ to a maximal tube of $P$ by starting with $T = C$ and then adding to $T$ all elements of $P\setminus C$ which are incomparable with only finitely many elements of $C$.
However, this need not create a maximal tube as in constructing $T$ there may now be elements of $T$ which are incomparable to infinitely many other elements of $T$.
Despite this, it turns out that if $C$ is a consolidated chain -- that is, it admits no non-trivial consolidation -- then $T$ as defined above is in fact a maximal tube.
We now formalise and prove this fact.

\begin{lemma}
    \label{lem:find-maximal-tube}
    Let $C$ be a maximal consolidated chain in a vacillating scattered FAC poset $P$, and let $T$ be defined as
    \begin{equation}
		\label{eq:define-tube}
		T \defined \set{x\in P \st \incompset{x}_C \text{ is finite}}.
	\end{equation}
    Then $C\sseq T$ and $T$ is a maximal tube.
\end{lemma}

\begin{commentfigure}[ht]
	\begin{center}
		\begin{tikzpicture}[
  y=.74cm,font=\normalsize,every node/.style={inner sep=1pt},
  line width=0.7pt,line cap=round,line join=round,
  continuation/.style={line width=1.3pt,line cap=round,
    dash pattern=on 0pt off 3pt},
  pics/point/.style={code={
    \draw (-2pt,-2pt) -- (2pt,2pt)
          (-2pt,2pt) -- (2pt,-2pt);
  }}
]
  \def\paneloutline{%
    \draw[rounded corners=8pt] (-.90,-.24) rectangle (10.14,8.66);
    \node[anchor=north east] at (9.90,8.47) {$P$};
    \coordinate (x) at (1.15,4.4);
    \draw (0.1,0.15) -- (0.1,8.25);
    \draw (8.75,0.15) -- (8.75,8.25);
    \draw (x)
      .. controls (2.00,6.38) and (3.45,7.4) .. (5.6,7.4)
      .. controls (6.7,7.4) and (8.0,7.45) .. (8.75,7.05);
    \draw (x)
      .. controls (2.05,1.53) and (2.83,1.87) .. (3.65,1.54)
      .. controls (4.47,1.21) and (4.94,1.17) .. (5.3,0.75)
      .. controls (5.9,0.05) and (7.5,0.1) .. (8.75,0.6);
    \node[right=6pt] at (0.1,7.85) {$T$};
    \node at (2.93,6.16) {$Q$};

    \pic at (x) {point};
    \node[left=4pt] at (x) {$x$};
    }

  \begin{scope}[x=.51cm]
    \paneloutline
    \coordinate (c) at (4.05,4.35);
    \draw[line width=.5pt,shorten >=5pt] (3.53,.15)
      .. controls (3.25,1.7) and (4.05,2.85) .. (c);
    \draw[line width=.5pt,shorten <=5pt] (c)
      .. controls (4.05,5.12) and (3.825,5.63) .. (3.825,6.2)
      .. controls (3.825,6.87) and (4.24,7.56) .. (4.24,8.25);
    \node[right=3pt] at (4.24,8.02) {$C$};
    \coordinate (z) at (6.45,6.30);
    \coordinate (y) at (6.25,2.05);
    \coordinate (w) at (9.16,4.40);
    \foreach \name in {c,z,y,w} \pic at (\name) {point};
    \foreach \name in {c,z,y} \node[left=4pt] at (\name) {$\name$};
    \node[right=3pt] at (w) {$w$};
    \draw[shorten <=6pt,shorten >=6pt] (z) -- (w);
    \draw[shorten <=6pt,shorten >=6pt] (y) -- (w);

  \end{scope}

  \begin{scope}[xshift=6.53cm,x=.68cm]
    \paneloutline
    \coordinate (z-apex) at (3.825,6.2);
    \draw (2.7,2.0) -- (z-apex) -- (4.95,2.0) -- cycle;
    \node at (4.75,4.25) {$J$};
    \def\chainX{4.05}
    \foreach \i/\y in {1/2.55,2/3.05,3/3.55} {
      \coordinate (z-\i) at (\chainX,\y);
      \pic at (z-\i) {point};
      \node[left=3pt] at (z-\i) {$z_{\i}$};
    }
    \begin{scope}[line width=.5pt]
      \draw[shorten >=4.5pt] (3.53,0.15)
        .. controls (3.53,1.15) and (\chainX,1.75) .. (z-1);
      \foreach \a/\b in {1/2,2/3}
        \draw[shorten <=4.5pt,shorten >=4.5pt] (z-\a) -- (z-\b);
      \draw[shorten <=4.5pt] (z-3) -- (\chainX,4.35)
        .. controls (\chainX,5.12) and (3.825,5.63) .. (z-apex)
        .. controls (3.825,6.87) and (4.24,7.56) .. (4.24,8.25);
    \end{scope}
    \node[right=3pt] at (4.24,8.02) {$C$};
    \foreach \x in {3.45,4.05}
      \draw[continuation] (\x,3.95) -- (\x,4.35);

    \draw[rounded corners=5pt] (5.5,1.0) rectangle (8.25,6.3);
    \node at (7.85,5.95) {$D$};
    \coordinate (X) at (6.9,5.3);
    \draw[shorten >=8pt] (5.9,1.5) -- (X);
    \draw[shorten <=8pt] (X) -- (7.9,1.5);
    \draw (5.9,1.5) -- (7.9,1.5);
    \pic at (X) {point};
    \node[right=8pt] at (X) {$X$};
    \node at (7.65,3.75) {$Y$};

    \foreach \i/\y in {1/2.0,2/2.5,3/3.0} {
      \coordinate (y-\i) at (6.65,\y);
      \pic at (y-\i) {point};
      \node[right=3pt] at (y-\i) {$y_{\i}$};
    }
    \foreach \x in {6.65,7.15}
      \draw[continuation] (\x,3.4) -- (\x,3.9);
  \end{scope}
\end{tikzpicture}		
	\end{center}
	\caption{Some of the key objects in the proof of \Cref{lem:find-maximal-tube}. The left-hand panel displays the objects in \Cref{claim:q-convex} and the right-hand panel in the proof after the claim.}
	\label{fig:find-maximal-tube}
\end{commentfigure}

\begin{proof}
	First note that it is immediate that $C\sseq T$, and that if $T$ as in \eqref{eq:define-tube} is a tube, then it must be maximal, as all $x\notin T$ have $\incompset{x}_C$ infinite. 
	It thus suffices to prove that $T$ is a tube in $P$.

	\sidecomment{The whole idea here is to conclude that $C$ is not consolidated (for contradiction), and we'd better start building a chain $D$ that we can consolidate to.}
    Take an arbitrary $x \in T$, and assume for contradiction that $x$ is incomparable with an infinite set $Q \defined \incompset{x}_T$ of elements of $T$.

	\begin{claim}
		\label{claim:q-convex}
		$Q$ is a convex subset of $P$, i.e.\ for all $y,z \in Q$ and $w\in P$ with $y < w < z$, we have $w \in Q$.
	\end{claim}

	\begin{proof}
		First note that $Q$ is convex as a subset of $T$: if $w\in T$ and $y < w < z$ and $w\sim x$, then either $x > w > y$ or $x < w < z$, in either case contradicting $y,z\in Q = \incompset{x}_T$.
		It thus suffices to prove that $w\in T$, i.e.\ $\incompset{w}_C$ is finite.

		Take $c\in\incompset{w}_C$: we find $c\not > z$ and $c \not < y$, and thus
		\begin{equation*}
			c\in \incompset{y}_C \quad\text{or}\quad c\in\incompset{z}_C \quad\text{or}\quad y\leq c \leq z.
		\end{equation*}
		As $y,z \in T$, we know that $\incompset{y}_C$ and $\incompset{z}_C$ are both finite.
		Finally, $y\leq c\leq z$ implies that $c\in\incompset{x}_C$, which is also finite. 
		We thus find that
		\begin{equation*}
			\incompset{w}_C \sseq \incompset{x}_C \union \incompset{y}_C \union \incompset{z}_C,
		\end{equation*}
		which is finite, as required.
	\end{proof}

	Note moreover that $Q$ is convex as a subset of $T$ (i.e.\ if $y,z\in Q$ and $w\in T$ have $y<w<z$, then $w\in Q$).
    As $P$ is an FAC poset, \Cref{fact:infinite-chain} tells us that there is an infinite chain in $Q$.
    Assume that $Q$ contains an infinite increasing sequence (the decreasing case is similar), and thus let $D$ be an infinite saturated chain in $Q$ such that $N_+(D)$ is non-empty.
	
	It therefore follows from \Cref{lem:exists-atomic} that there is some $X\in \atomicincr(D)$, and as $P$ is vacillating we thus know that $X$ has a final segment of order type $\ww$.
	We may thus find elements $y_0 < y_1 < y_2 < \cdots$ of $T$ all incomparable to $x$ satisfying $y_i \covered y_{i+1}$ for all $i$, where we note that, by \Cref{claim:q-convex}, the fact that these covers hold in $Q$ implies that they are covers in $P$.
    Each $y_i$ is incomparable to only finitely many elements of $C$ by definition; we thus define
    \begin{align*}
        J_i\defined \incompset{y_i}_C \quad \text{ and } \quad I\defined \incompset{x}_C.
    \end{align*}

	\sidecomment{We need to actually get into the weeds to find this consolidation. We're going to replace one copy of $\ww$ with another, and the bulk of this proof is just going through the checklist of proving that these two objects have the required properties. The proof itself is largely routine.}
	We now enumerate several properties of these sets.
	\begin{enumerate}
		\item \label{item:j-max-increase} $\max J_{i+1} \geq \max J_i$, as otherwise we would have $\max J_i > \max J_{i+1}$, which would imply that $\max J_i > y_{i+1}$, which would contradict the facts that $y_{i+1} > y_i$ and $y_i\incomp \max J_i$.
		\item \label{item:j-min-increase} $\min J_{i+1} \geq \min J_i$, as otherwise we would have $\min J_{i+1} < \min J_i$, which would imply that $\min J_{i+1} < y_i$, which would contradict the facts that $y_i < y_{i+1}$ and $y_{i+1} \incomp \min J_{i+1}$.
		\item \label{item:j-no-gaps} There is no $z\in C$ with $J_i < z < J_{i+1}$, as otherwise $y_i < z < y_{i+1}$, contradicting $y_i \covered y_{i+1}$.
		\item \label{item:i-no-gap} For all $i$, there is no $z\in C$ with $I < z < J_i$ or $J_i < z < I$, as otherwise $x < z < y_i$ or $y_i < z < x$, contradicting $x\incomp y_i$.
	\end{enumerate}
    Let $J\defined\bigcup_{i<\ww} J_i$ and $Y\defined \set{y_i \st i < \ww}$, noting that \cref{item:j-no-gaps} implies that $J$ is a saturated chain in $C$.
    Note that the above \cref{item:j-max-increase,item:j-min-increase,item:j-no-gaps} together tell us that $J$ is order-isomorphic to a subset of $\ww$.
    Note that $Y$ is a saturated chain in $P$.

	If $J$ is finite, then we can perform a non-trivial consolidation by removing $J$ from $C$, inserting $Y$, and extending to a saturated chain.
	As $J$ is finite and $Y$ is infinite, this contradicts our assumption that $C$ is consolidated.
	(Let $C'$ be a saturated chain extending $(C\setminus J) \union Y$. We can witness the consolidation with the identity function, but $[Y]_+ \in \atomicincr(C')\setminus \atomicincr(C)$, so the consolidation is non-trivial.)
    Thus we may assume that $J$ is order-isomorphic to $\ww$.

    Define for each $i$ the point $z_i\in J$ to be the minimal point above $J_i$, so $y_i < z_i$; note that such a point must exist as $J_i$ is finite.
    As $I$ is finite, \cref{item:i-no-gap} tells us that
	\sidecomment{Condition \eqref{eq:above-I} should be obvious from the $J_i$ being ``increasing to infinity'' in a certain sense, and $I$ being both finite and not too far from any $J_i$: the $J_i$ must extend upwards, and thus eventually outgrow $I$.}
	\begin{equation}
		\label{eq:above-I}
		\exists z\in J \quad\text{ with }\quad z > I.
	\end{equation}
	Thus there exists $k$ such that $z_i > \max(I)$ for all $i>k$, and so for $i > k$ we have $z_i > x$.
    Then for all $i,j > k$, we have $y_i \not> z_j$ as otherwise we would have $y_i > z_j > x$, contradicting $y_i\incomp x$.
	But now we see that $J \consolidatesto Y$, so it suffices to prove that this consolidation is non-trivial, as we may send $J$ to $Y$ and hold all else constant.

	\cref{item:i-no-gap} and \eqref{eq:above-I} together tell us that $z\in J_n$ for all sufficiently large $n$, which implies that $z\incomp y_n$.
	In fact, this argument applies to all $z\in J$ with $z > I$ (with the choice of $n$ depending on $z$).
	Thus $J$ and $Y$ do not have mutually finite incomparability, and so $Y\not\simeq_+ J$, and the consolidation $J\consolidatesto Y$ is non-trivial.
	Therefore $C$ admits a non-trivial consolidation, contradicting our assumption that $C$ is consolidated.
	This proves the lemma.
\end{proof}





We can now prove the following result which, together with \Cref{prop:maximal-tube-suffices}, immediately implies \Cref{thm:main} that all countable vacillating FAC posets have spines.

\begin{proposition}
	\label{prop:exists-maximal-tube}
	Let $P$ be a countable vacillating FAC poset. Then $P$ contains a maximal tube.
\end{proposition}

\sidecomment{This really is just gluing things together, feeding the output of one proposition into the input of the next. The hard work has already been done.}
\begin{proof}
	First, by \Cref{cor:scattered-reduction}, it suffices to consider the case when $P$ is a countable scattered vacillating FAC poset.
	Then \Cref{prop:quasifdd-reduction} tells us that it suffices to consider the case wherein $P$ is in fact quasifounded, i.e.\ $P$ is a countable scattered quasifounded FAC poset.

	We know from \Cref{prop:consolidated-exists} that any countable scattered quasifounded FAC poset contains a maximal consolidated chain, and so we may apply \Cref{lem:find-maximal-tube} to find the required maximal tube in any such poset.
	This implies that our original poset $P$ contains a maximal tube, as required.
\end{proof}

The proof of \Cref{thm:main} is thus immediate:

\begin{proof}[Proof of \Cref{thm:main}]
	\Cref{prop:exists-maximal-tube} tells us that $P$ contains a maximal tube $T$, and then \Cref{prop:maximal-tube-suffices} tells us that there is a chain $C\sseq T$ which is a spine of $P$.
	Therefore $P$ has a spine, as required.
\end{proof}

\section{Concluding remarks and open problems}
\label{sec:conclusion}

In this paper, we have proved a sequence of structural results concerning FAC posets, which has culminated in a proof that the Aharoni--Korman conjecture holds for a very large class of posets. 
Together with the negative result in \cite{Hol26}, this conjecture is now much better understood than it was before.
However, despite all that is known, there are still many questions left open, and several new questions raised by the investigations carried out here.

First, as \Cref{conj:ak} fails in general, it is natural to ask what weaker statements might be true; we have shown here that the statement is true with an additional condition on the poset $P$.
However, one could also consider weakening the condition of being a spine, as follows. 
Given a family $\cF$ of sets, an element $A\in \cF$ is said to be \emph{strongly maximal} if every $B\in\cF$ has $\abs{B\setminus A} \leq \abs{A\setminus B}$.
Intuitively, one cannot take $A$, remove a few elements, add back in more elements than were removed, and remain in the family $\cF$.

A subset $C\sseq P$ is thus a \emph{strongly maximal chain} if it is a chain of $P$ and, for all chains $D\sseq P$, we have $\abs{D\setminus C} \leq \abs{C \setminus D}$. 
Note in particular that strongly maximal chains are necessarily maximal chains.
One may define \emph{strongly maximal antichains} similarly.
We make the following simple observation.

\begin{observation}
	If the chain $C\sseq P$ is a spine of $P$, then it is also a strongly maximal chain of $P$.
\end{observation}

Our result \Cref{thm:main-on-first-page} thus implies that every countable vacillating FAC poset has a strongly maximal chain.
As was noted and discussed by Aharoni \cite{aharoni1991infinite}, it is a consequence of \Cref{thm:infinite-greene-kleitman} that every poset with no infinite chain must contain a strongly maximal antichain.
This result was then extended by Aharoni and Berger \cite{aharoni2011strongly} to a more general class of posets in the countable case.
Strongly maximal chains, however, seem to have received little attention, and the following question is natural, serving as both a relaxation of \Cref{conj:ak} and the dual of the above result of Aharoni.

\begin{question}
	\label{q:smc}
	Let $P$ be a poset with no infinite antichain. Must $P$ have a strongly maximal chain?
\end{question}

We remark that, since first posing \Cref{q:smc}, the construction from \cite{Hol26} has been extended by Lupini \cite{Lup26} to produce a countable (non-vacillating) FAC poset with no strongly maximal chain.

Following on from this, we have on many occasions made use of the fact that we worked with a countable poset $P$.
For higher cardinalities, even the structure of linear orders becomes much more complex, and so it would seem to be very far from trivial to generalise the arguments presented here, if they even generalise at all.

\begin{question}
    \label{q:higher-cardinalities}
    Let $P$ be a vacillating FAC poset of cardinality $\kk$.
    For which $\kk$ must $P$ have a spine?
\end{question}

Indeed, an immediate obstruction would be the lack of a structural result similar to \Cref{thm:structural} for higher cardinalities, and such a result could well require various set-theoretic assumptions.

Finally, it is also natural to consider whether the techniques presented here are enough to prove \Cref{conj:ak-general} in the case of vacillating posets.

\begin{question}
    \label{q:ak-general-vacillating}
    Let $P$ be a countable vacillating FAC poset, and let $k\geq 2$ be an integer.
    Must there be disjoint chains $C_1,\dotsc,C_k \sseq P$ and a partition $P = \set{A_i \st i\in I}$ into disjoint antichains such that each $A_i$ meets $\min\set{\abs{A_i}, k}$ of the chains $C_j$?
\end{question}

While the structural techniques developed in the proof of \Cref{thm:main} could well be of assistance in resolving \Cref{q:ak-general-vacillating} if it is true, it seems as though there would also be the need for some genuine new ideas as well.

\section{Acknowledgements}

The author would like to thank B\'{e}la Bollob\'{a}s for his many valuable comments.
Thanks are also due to Nikolai Beluhov and George Bergman for each pointing out several inaccuracies in an earlier version of the manuscript, and for further suggestions which significantly improved the presentation of the paper.
Much of the work presented in this paper was carried out while the author was funded by the internal graduate studentship of Trinity College, Cambridge.






\bibliographystyle{acm}  
\renewcommand{\bibname}{Bibliography}
\bibliography{main}

@article{aharoni1991infinite,
  title={Infinite matching theory},
  author={Aharoni, Ron},
  journal={Discrete Mathematics},
  volume={95},
  number={1-3},
  pages={5--22},
  year={1991},
  publisher={Elsevier}
}

@article{aharoni1992greene,
  title={Greene-{K}leitman's theorem for infinite posets},
  author={Aharoni, Ron and Korman, Vladimir},
  journal={Order},
  volume={9},
  pages={245--253},
  year={1992},
  publisher={Springer}
}

@article{zaguia2024progress,
  title = {{Some progress on the Aharoni--Korman conjecture}},
  journal = {Discrete Mathematics},
  volume = {347},
  number = {10},
  year = {2024},
  author = {Imed Zaguia},
}

@article{duffus2002intervals,
  title={Some progress on the {A}haroni--{K}orman conjecture},
  author={Duffus, Dwight and Goddard, Ted},
  journal={Discrete mathematics},
  volume={250},
  number={1-3},
  pages={79--91},
  year={2002},
  publisher={Elsevier}
}

@article{hausdorff1908classification,
  title={Grundz{\"u}ge einer {T}heorie der geordneten {M}engen},
  author={Hausdorff, Felix},
  journal={Mathematische Annalen},
  volume={65},
  number={4},
  pages={435--505},
  year={1908},
  publisher={Springer-Verlag Berlin/Heidelberg}
}

@article{lawson1987ordered,
  title={Ordered sets with no infinite antichains},
  author={Lawson, Jimmie D and Mislove, Michael and Priestley, Hilary},
  journal={Discrete Mathematics},
  volume={63},
  number={2-3},
  pages={225--230},
  year={1987}
}

@article{duffus1981complete,
  title={Complete ordered sets with no infinite antichains},
  author={Duffus, Dwight and Pouzet, Maurice and Rival, Ivan},
  journal={Discrete Mathematics},
  volume={35},
  number={1-3},
  pages={39--52},
  year={1981},
  publisher={Elsevier}
}

@article{van2022counterexample,
  title={Counterexample to a conjecture of {A}haroni and {K}orman},
  author={van der Zypen, Dominic},
  journal={arXiv preprint arXiv:2205.02296},
  year={2022},
  note={2 pages}
}

@article{aharoni2022strongly,
  title={Strongly maximal matchings and strongly minimal covers},
  author={Aharoni, Ron},
  journal={arXiv preprint arXiv:2206.02576},
  year={2022},
  note={3 pages},
}

@article{aharoni1995strongly,
  title={Strongly perfect infinite graphs},
  author={Aharoni, Ron and Loebl, Martin},
  journal={Israel Journal of Mathematics},
  volume={90},
  number={1},
  pages={81--91},
  year={1995},
  publisher={Springer}
}

@article{aharoni2011strongly,
  title={Strongly maximal antichains in posets},
  author={Aharoni, Ron and Berger, Eli},
  journal={Discrete mathematics},
  volume={311},
  number={15},
  pages={1518--1522},
  year={2011},
  publisher={Elsevier}
}

@phdthesis{goddard1996ordered,
  title={Ordered sets: Colorings and complexity},
  author={Goddard, Ted},
  year={1996},
  school={Emory University},
  type = {{PhD Thesis}},
  month = {June},
}

@article{dilworth1950decomposition,
  title={A Decomposition Theorem for Partially Ordered Sets},
  author={Dilworth, RP},
  journal={Annals of Mathematics},
  volume={51},
  number={1},
  pages={161--166},
  year={1950},
  publisher={JSTOR}
}

@article{abraham1987note,
  title={A note on {D}ilworth's theorem in the infinite case},
  author={Abraham, Uri},
  journal={Order},
  volume={4},
  pages={107--125},
  year={1987},
  publisher={Springer}
}

@article{perles1963dilworth,
  title={On {D}ilworth’s theorem in the infinite case},
  author={Perles, Micha A},
  journal={Israel Journal of Mathematics},
  volume={1},
  pages={108--109},
  year={1963},
  publisher={Springer}
}

@article{aharoni1984konig,
  title={{K}{\"o}nig's duality theorem for infinite bipartite graphs},
  author={Aharoni, Ron},
  journal={Journal of the London Mathematical Society},
  volume={29},
  number={2},
  pages={1--12},
  year={1984},
  publisher={Wiley Online Library}
}

@article{greene1976structure,
  title={The structure of {S}perner k-families},
  author={Greene, Curtis and Kleitman, Daniel J},
  journal={Journal of Combinatorial Theory, Series A},
  volume={20},
  number={1},
  pages={41--68},
  year={1976},
  publisher={Elsevier}
}

@article{aharoni2009menger,
  title={Menger’s theorem for infinite graphs},
  author={Aharoni, Ron and Berger, Eli},
  journal={Inventiones mathematicae},
  volume={176},
  number={1},
  pages={1--62},
  year={2009},
  publisher={Springer}
}

@article{diestel2003countable,
  title={The countable {E}rd{\H{o}}s--{M}enger conjecture with ends},
  author={Diestel, Reinhard},
  journal={Journal of Combinatorial Theory, Series B},
  volume={87},
  number={1},
  pages={145--161},
  year={2003},
  publisher={Elsevier}
}

@article{aharoni1987menger,
  title={Menger's theorem for countable graphs},
  author={Aharoni, Ron},
  journal={Journal of Combinatorial Theory, Series B},
  volume={43},
  number={3},
  pages={303--313},
  year={1987},
  publisher={Elsevier}
}

@article{aharoni1994menger,
  title={Menger's theorem for a countable source set},
  author={Aharoni, Ron and Diestel, Reinhard},
  journal={Combinatorics, Probability and Computing},
  volume={3},
  number={2},
  pages={145--156},
  year={1994},
  publisher={Cambridge University Press}
}

@misc{mehta2025formal,
  title={{Disproof of the Aharoni-Korman conjecture}},
  author={Mehta, Bhavik},
  year={2025},
  publisher = {GitHub},
  journal = {GitHub repository},
  howpublished = {\url{https://github.com/b-mehta/AharoniKorman}},
  note={A formal verification of the counterexample to the {Aharoni--Korman} conjecture. Accessed 2026-06-17}
}

@article{Hol26,
  title={The Aharoni--Korman conjecture is false},
  author={Hollom, Lawrence},
  journal={Israel Journal of Mathematics},
  pages={1--14},
  year={2026},
  publisher={Springer}
}

@article{HR26, 
  title={Counterexamples to Conjectures on Strong Maximality and Minimality}, 
  volume={33}, 
  url={https://www.combinatorics.org/ojs/index.php/eljc/article/view/v33i2p11}, 
  DOI={10.37236/14948}, 
  number={2}, 
  journal={The Electronic Journal of Combinatorics}, 
  author={Hollom, Lawrence and Randall Shaw, Benedict}, 
  year={2026}, 
  pages={P2.11} 
}

@book{stanley2011enumerative,
  title={{Enumerative Combinatorics, Volume 1, Second Edition}},
  author={Stanley, Richard P},
  publisher={Cambridge studies in advanced mathematics},
  year={2011}
}

@misc{Lup26,
  author       = {Lupini, Martino},
  title        = {A countable scattered poset without a strongly maximal chain},
  year         = {2026},
  month        = oct,
  howpublished = {Zenodo},
  note         = {Preprint, version 1},
  doi          = {10.5281/zenodo.23458450},
  url          = {https://zenodo.org/records/23458450}
}

@misc{Lean,
  title={{The structure of FAC posets: a Lean verification}},
  author={Hollom, Lawrence},
  year={2026},
  publisher = {GitHub},
  journal = {GitHub repository},
  howpublished = {\url{https://github.com/LawrenceHollom/lean-fac-posets}},
  note={A formal verification of the results of this paper. Accessed 2026-11-21}
}






\end{document}